\documentclass[11pt]{amsart}
\usepackage{amsmath,amssymb,amsthm,tikz}
\usepackage[numbers,comma,square,sort&compress]{natbib}
\usepackage{mathabx}   
\RequirePackage{color}
\RequirePackage[colorlinks,urlcolor=my-blue,linkcolor=my-red,citecolor=my-green]{hyperref}

\numberwithin{figure}{section}

\usepackage{enumerate}

\usepackage{scalerel}    
\usepackage{stmaryrd}   
\usepackage{mathabx}   

\definecolor{my-blue}{rgb}{0.0,0.0,0.6}
\definecolor{my-red}{rgb}{0.5,0.0,0.0}
\definecolor{my-green}{rgb}{0.0,0.5,0.0}
\definecolor{nicos-red}{rgb}{0.75,0.0,0.0}

\newtheorem{theorem}{\sc Theorem}[section]
\newtheorem{lemma}[theorem]{\sc Lemma}
\newtheorem{proposition}[theorem]{\sc Proposition}
\newtheorem{corollary}[theorem]{\sc Corollary}

\newtheorem{definition}[theorem]{\sc Definition}
\numberwithin{equation}{section}
\theoremstyle{remark}
\newtheorem{remark}[theorem]{Remark}
\newtheorem{example}[theorem]{Example}

\newcommand{\be}{\begin{equation}}
\newcommand{\ee}{\end{equation}}
\newcommand{\beq}{\begin{equation}}
\newcommand{\eeq}{\end{equation}}

\newcommand{\nn}{\nonumber}
\providecommand{\abs}[1]{\vert#1\vert}

\newcommand{\fl}[1]{\lfloor{#1}\rfloor} 
\newcommand{\ce}[1]{\lceil{#1}\rceil}

  \def\cD{\mathcal{D}}

\def\cG{\mathcal{G}}
\def\cH{\mathcal{H}}

\def\cR{\mathcal{R}}
\def\cU{\mathcal{U}}

\def\cK{\mathcal{K}}

\def\cA{\mathcal{A}}  
\def\cY{\mathcal{Y}}

\def\esssup{\mathop{\mathrm{ess\,sup}}}

\font \mymathbb = bbold10 at 11pt

\newcommand{\one}{\mbox{\mymathbb{1}}}

\def\kS{\mathfrak{S}}

\def\bE{\mathbb{E}} 

\def\bP{\mathbb{P}}

\def\bR{\mathbb{R}}
\def\bZ{\mathbb{Z}}

 \def\Z{\bZ}  \def\R{\bR}
  \def\uvec{\mathbf{u}}\def\vvec{\mathbf{v}}

\def\wvec{\mathbf{w}}

\def\w{\omega}

\def\e{\varepsilon}
\def\ind{\mathbf{1}}
\def\ddd{\displaystyle}

\def\m1{\mathbf{1}}

 \def\Vvv{{\rm\mathbb{V}ar}}  \def\Cvv{{\rm\mathbb{C}ov}}

 \def\wt{\widetilde}  \def\wh{\widehat} \def\wb{\overline} \def\wc{\widecheck}
 
\def\E{\bE}
\def\P{\bP} 
\def\north{S_{\mathcal N}}  \def\south{S_{\mathcal S}}  \def\east{S_{\mathcal E}}  
\def\west{S_{\mathcal W}}  
 
\def\OSP{(\Omega, \kS, \P)}

\def\funct lp{L} 
\def\funct lpbar{\bar L} 

\def\range{\mathcal R}

\def\Uset{\mathcal U}

\def\cA{\mathcal A}

\def\Gpl{G}
\def\Gpp{G}

\def\gpp{g_{\text{\rm pp}}}
\def\gpl{g_{\text{\rm pl}}}

\def\B{{B}}

\def\cE{{\mathcal E}}

\DeclareMathOperator{\dist}{dist}

\DeclareMathOperator{\Var}{Var}

\DeclareMathOperator{\ri}{ri}    

\def\cH{\mathcal{H}}

\DeclareMathOperator{\conv}{co}   

\definecolor{darkgreen}{rgb}{0.0,0.5,0.0}
\definecolor{darkblue}{rgb}{0.0,0.0,0.3}
\definecolor{nicosred}{rgb}{0.65,0.1,0.1}
\definecolor{light-gray}{gray}{0.7}
\allowdisplaybreaks[1]

\usepackage{filemod}

\def\cif1{v}   
\def\sigmawc{\wc\sigma}   
\def\loexp{\chi}  
\def\trexp{\zeta}  

\def\deq{\overset{d}=}

\def\qedex{\hfill$\triangle$} 
   \def\tincr{t}   \def\Iincr{I}
\def\Xw{X}  
\def\Yw{Y}  
 \def\ul{\underline}  
\def\exitt{\tau} 
\def\Gppo{{^o\Gpp}}

\newcommand\bbullet{{\raisebox{0.5pt}{\scaleobj{0.6}{\bullet}}}} 
\newcommand\brbullet{{\raisebox{-1pt}{\scaleobj{0.5}{\bullet}}}} 
\newcommand\cbullet{{\raisebox{1pt}{\scaleobj{0.6}{\bullet}}}}

\newcommand\dbullet{{\scaleobj{0.7}{\bullet}}}   

\newcommand\blfootnote[1]{%
  \begingroup
  \renewcommand\thefootnote{}\footnote{#1}%
  \addtocounter{footnote}{-1}%
  \endgroup
}

\begin{document}

\begin{center}
\Large

{\bf Variational formulas, Busemann functions, and fluctuation exponents for the corner growth model with exponential weights}  \\[14pt]

\normalsize 

Lecture notes  by Timo Sepp\"al\"ainen\footnote{Department of Mathematics, 480 Lincoln Drive, University of Wisconsin--Madison, Madison, WI 53711, USA.  URL {\tt http://www.math.wisc.edu/$\sim$seppalai} \ The author was partially supported by  National Science Foundation grant DMS-1602486   and by the Wisconsin Alumni Research Foundation.}  \\[13pt]

\end{center} 

{\it Abstract.}  These lecture notes discuss several  related features of the exactly solvable two-dimensional corner growth model with exponentially distributed weights.  A  key property of this model is the availability of a fairly explicit stationary version  that possesses useful independence properties.   With the help of couplings and estimates,  we prove the existence of Busemann functions for this model,  and the precise values of the longitudinal and transversal  fluctuation exponents for the stationary corner growth model.   The Busemann functions in turn furnish extremals for variational formulas that describe limiting shape functions. 

Early versions of these  notes were  used at  the Research School on  Random Structures in Statistical Mechanics and Mathematical Physics   at  CIRM in Marseille Luminy (March 2017) and developed for  the Proceedings of the    AMS Short Course on Random Growth Models (January 2017).   

\tableofcontents

\blfootnote{Version of \today}

\section{Introduction}



{\bf Overview.} 
These notes discuss a circle of ideas associated with  random growth models.  The results are developed in the context of one of the oldest  exactly solvable models in the 1+1 dimensional Kardar-Parisi-Zhang (KPZ) class, namely the {\it corner growth model}   (CGM) with exponential weights. 
 The centerpiece of the development is an increment-stationary stochastic process that represents a stationary version of the random growth model.    Section \ref{sec:var}  describes how this  stochastic process arises as the extremal in a variational formula for the limiting shape function.  This part of the text treats a very general   last-passage percolation (LPP) model  on the $d$-dimensional integer lattice.  
 
  From Section \ref{v:s-stat-cgm} onwards we restrict ourselves to the exponential CGM in two dimensions.   Section \ref{v:s-stat-cgm} constructs the stationary version of the exponential CGM and shows how it enables an explicit calculation of    the limiting shape function of both the point-to-point and the point-to-line version of the CGM.  Additionally we prove the shape theorem and show that macroscopically geodesics are straight lines. 
  
     Section \ref{sec:bus} proves  the existence of limiting Busemann functions.  These Busemann functions are    extremals for the variational formulas described in   Section \ref{sec:var}.
As an application of the results on Busemann functions we prove a CGM version of the midpoint problem.  
  
  A feature of the  KPZ class is that two  fluctuation exponents take universal values for all models in the class.    The longitudinal exponent $\loexp$  describes the magnitude of fluctuations of particle current in conservative particle systems,   last-passage  and first-passage times in growth models,  and free energy fluctuations in positive temperature polymers.  The transversal exponent $\trexp$  describes the magnitude of  spatial correlations and the   fluctuations of random paths such as geodesics in percolation models    and   polymer paths.  
  
     In 1+1 dimensions the  values of the KPZ  exponents are  $\loexp=\tfrac13$  and $\trexp=\tfrac23$.   So far these exact exponent values have been proved only for  solvable models in the KPZ class, with the possible exception of the class of zero-range processes treated in \cite{bala-komj-sepp-12}.   
    In  Section \ref{sec:exp} we prove $\loexp=\tfrac13$  and $\trexp=\tfrac23$ in the 1+1 dimensional stationary exponential CGM.   The key to these coupling proofs is the explicit stationary version of the CGM.  
  
  Appendix \ref{app:coupl} develops  some useful couplings for the CGM.  
  
 This text does not have to be read in its entirety in linear order.   Section  \ref{sec:var} sets the stage but strictly speaking  is  not needed for the remainder.  Section \ref{v:s-stat-cgm} is a prerequisite for   Sections  \ref{sec:bus}  and  \ref{sec:exp}, but these last two sections   are independent of each other. 
 
 \medskip


Assuming standard background in measure theory and  probability theory,  
these notes are fairly  self-contained,    except for the following points in  Section \ref{sec:var}.  We take for granted  the existence of  limiting shape functions (Theorem \ref{v:t-lln2}).  This existence theorem  is an application of subadditive ergodic theory, together with some estimation to produce limit functions defined on $\R^d$. Three further results are quoted from the literature:   Lemma \ref{v:lm-xih}, Theorem \ref{cc-ergthm} which is  an ergodic theorem for cocycles, and the existence of a minimizer in Theorem \ref{th:K-var}.  

%


\medskip


{\bf Literature.}   Here are some brief notes on the history of the topics of these notes, without any claim to completeness. 

  The first derivation of  the limit shape of the corner growth model with exponential weights came in  Rost's seminal 1981 paper \cite{rost} on the hydrodynamic limit of the totally asymmetric simple exclusion process (TASEP)  with step initial condition.   An even earlier instance of work on similar limit shapes involved the longest common subsequence of a random permutation  \cite{hamm, loga-shep-77, vers-kero-77}.  The   Poissonization of this problem can be cast  as a last-passage growth model on a  homogeneous planar Poisson point process. 
The study of last-passage percolation began to flourish in the 1990s.  A  number of papers appeared on limit shapes    and large deviations, examples of which include  \cite{aldo-diac95, cohn-elki-prop-96, deus-zeit-95, deus-zeit-99,  jock-prop-shor-98, kimj-96, sepp-96, sepp98ebp, sepp98mprf, sepp-ptrf-98}.   

The notion that the fluctuations of these models do not follow the standard central limit theorem,  but have a different exponent $1/3$, was a well-established conjecture  from  physics,  but mathematical proof was lacking.  The first rigorous results of KPZ fluctuations  came at the turn of the millennium in the breakthrough papers \cite{baik-deif-joha-99, joha}. 

Variational formulas of the type presented in Section \ref{sec:var} arose first in homogenization and random walk in random environment (RWRE).  See equation (4.19) in the review paper \cite{kosy-07} and also  Definition 4 on page 5 of \cite{rose-phd-06}.      Versions of these formulas for   last-passage percolation came  in  \cite{geor-rass-sepp-16}.  Another related variational formula was introduced independently in undirected first-passage percolation  in \cite{kris-16}.

Busemann function is a notion from geometry.  It was introduced into first-passage percolation by  Hoffman \cite{hoff-05} and Newman \cite{newm-icm-95}.  Newman  and coauthors developed an argument for the existence of Busemann functions based on the coalescence of directed geodesics.  This method  relies on an assumption on the curvature of the limit shape.   For the exponential corner growth model the existence of Busemann functions  was first proved by Ferrari and Pimentel \cite{ferr-pime-05} with Newman's approach.  (The limiting Busemann function appears in equation (38) of \cite{ferr-pime-05}.)  Newman's method was applicable because the limit shape of the exponential corner growth model is known explicitly.   In Section \ref{sec:bus} of these notes we apply  an alternative method that uses  a stationary version of the LPP process.  This approach has been applied to the CGM with general weights in \cite{geor-rass-sepp-17-buse}.    
See \cite{bakh-cato-khan-14, cato-pime-13, ferr-mart-pime-09} for a further selection of papers that use Busemann functions   to study percolation, growth models,  and interacting particle systems.


The seminal papers \cite{baik-deif-joha-99, joha} on KPZ fluctuations  proved  Tracy-Widom limits with refined combinatorics and asymptotic analysis.  A more robust probabilistic proof of only the fluctuation exponents appeared a few years later in  \cite{cato-groe-06}.   This approach was adapted to the corner growth model in \cite{bala-cato-sepp}  and to a positive-temperature polymer model in \cite{sepp-12-aop-corr}.  


Review articles and lecture notes on   KPZ universality include the following: 
 \cite{boro-gori-16, corw-12-rev,  johalect-05}.  
Fluctuation exponents in both the KPZ class and the EW (Edwards-Wilkinson) class are developed in lectures  \cite{sepp-10-ens}.  
Lecture notes \cite{come-16, denholl-polymer} give overviews of the state of the art  in  directed polymer models, which are the positive temperature counterpart of directed percolation.  
Articles \cite{corw-16-rev, damr-rass-sepp-16} are introductions to KPZ universality and growth models aimed at general mathematical  audiences.  

\medskip

 
 {\bf Some general notation and terminology.} 
 $\Z_{\ge0}=\{0,1,2,3, \dotsc\}$ and $\Z_{>0}=\{1,2,3,\dotsc\}$.  For $n\in\Z_{>0}$ we abbreviate  $[n]=\{1,2,\dotsc,n\}$.   A sequence of $n$   points  is denoted by $x_{0,n}=(x_k)_{k=0}^n=\{x_0,x_1,\dotsc,x_n\}$, and in case it is a path of length $n$  also by $x_{\bbullet}$.   $a\vee b=\max\{a,b\}$.     $C$ is a constant whose value can change from line to line.

 The standard basis vectors of $\R^2$ are  $e_1=(1,0)$ and $e_2=(0,1)$.  For a point  $x=(x_1,x_2)\in\R^2$  the $\ell^1$-norm  is   $\abs{x}_1=\abs{x_1} + \abs{x_2}$ and integer parts  are taken coordinatewise:  $\fl{x}=(\fl{x_1}, \fl{x_2})$.     We call the $x$-axis occasionally the $e_1$-axis, and similarly the $y$-axis and the $e_2$-axis are the same thing.  
 Inequalities on $\R^2$ are interpreted coordinatewise: for $x=(x_1,x_2)\in\R^2$ and $y=(y_1,y_2)\in\R^2$,   $x\le y$  means $x_1\le y_1$ and $x_2\le y_2$.    Notation $[x,y]$ represents both   the line segment $[x,y]=\{tx+(1-t)y: 0\le t\le 1\}$ and the rectangle  $[x,y]=\{(z_1,z_2)\in\R^2:   x_i\le z_i\le  y_i \text{ for }i=1,2\}$. The context will make clear which case is used.   $0$ denotes the origin of both $\R$ and $\R^2$.  
 
 $X\sim$ Exp($\lambda$) for $0<\lambda<\infty$ means that random variable $X$ has exponential distribution with rate $\lambda$.   This $X$ is a positive random variable  whose probability distribution satisfies $P(X>t)=e^{-\lambda t}$ for $t\ge 0$.  It has mean $E(X)=\lambda^{-1}$ and variance $\Var(X)=\lambda^{-2}$.

 We write $\w_x$ and $\w(x)$ interchangeably for the weight attached to lattice point $x$.    In general,  $\overline X=X-EX$ denotes a random variable $X$ centered at its mean.  
 
 \newpage    



\section{Variational formulas for last-passage percolation shapes}  \label{sec:var} 

\subsection{Directed last-passage percolation on $\Z^d$} \label{s:LPPgen} 
We consider here a general setting before specializing   to the two-dimensional corner growth model.    Let $\OSP$   be a Polish product probability space $\Omega=\Gamma^{\Z^d}$ of random environments $\w=(\w_x)_{x\in\Z^d}\in\Omega$, with Borel $\sigma$-algebra $\kS$,  and a product probability  measure $\P$ under which the coordinates are i.i.d.\ random variables: for any distinct lattice points $x_1,\dotsc, x_n\in\Z^d$ and any Borel sets  $B_1,\dotsc, B_n\subset\Gamma$, 
\be\label{v:P}  
\P\{\w:  \text{$\w_{x_i}\in B_i$ for $i=1,\dotsc,n$}\}=\prod_{i=1}^n \P\{\w:  \w_0\in B_i\}.  \ee
The group of translations or shifts $\{\theta_x\}_{x\in\Z^d}$  act on $\Omega$ by $(\theta_x\w)_y=\w_{x+y}$. 

Let $\range$ be a finite subset of $\Z^d$.    A lattice path $x_{0,n}=(x_k)_{k=0}^n\subset\Z^d$ is admissible if its steps satisfy  $z_k=x_k-x_{k-1}\in\range$.   Let $\Uset=\conv\range$ be the convex hull of $\range $ in $\R^d$, and $\ri\Uset$ the relative interior of $\Uset$.    We put ourselves in the directed setting by assuming that 
\be\label{v:dir}  0\notin\Uset. \ee
This implies the existence of a vector $\hat u\in\R^d$ and $\delta>0$ such that $z\cdot\hat u\ge\delta$ for all $z\in\range$.  

For convenience we also assume that $\Z^d$ is the smallest additive group that contains $\range$.  Without this assumption we would carry along the group generated by $\range$ in the development. 

The weights of admissible steps $z$  are  determined by a measurable function 
   $V:\Omega\times\range\to\R$ about which we assume the following:    
\be     \label{v:V1}  
\begin{aligned}   &\text{$\forall z\in\range$,  $V(\w,z)$ is a local function of $\w$ and for some $p>d$,   $V(\cdot\,,z)\in L^p(\P)$.}
\end{aligned}\ee
By definition, a local function of $\w$ is one that depends on only finitely many coordinates of $\w$. 

\begin{example}
The basic example to think about is the two-dimensional corner growth model, with real weights on the vertices:  $\w=(\w_x)_{x\in\Z^2}\in\Omega=\R^{\Z^2}$. The set of  admissible steps is $\range=\{e_1,e_2\}$, and potential given by the weight at the origin: $V(\w,z)=\w_0$.    The set of possible limiting velocities of paths is the  closed line segment $\cU=[e_1,e_2]$, and its relative interior is the open line segment $\ri\cU=(e_1,e_2)$. 
\qedex\end{example}

\begin{example}
The formulation covers also weights on directed  edges.  Let  $\range=\{e_1,e_2,\dotsc, e_d\}$ and   let  $\vec\cE_d=\{(x,y)\in\Z^d\times\Z^d:  y-x\in\range\}$ be the set of directed nearest-neighbor edges on $\Z^d$.   Let 
 $\w=(\w(e))_{e\in\vec\cE_d}$ be a configuration of weights on directed nearest-neighbor edges.  
  The potential picks out the edge weight: $V(\theta_x\w,z)=\w(x,x+z)$ for $x\in\Z^d$ and $z\in\range$.    
\qedex\end{example}

The point-to-level last-passage percolation with  {\it external field} or {\it tilt}  $h\in\R^d$  is defined by 
 \be
\Gpl_n(h)=  
 \max_{x_{0,n}: \,x_0=0} \Bigl\{ \;\sum_{k=0}^{n-1}V(\theta_{x_k}\w, z_{k+1}) + h\cdot x_n \Bigr\},  \qquad    h\in\R^d.  
 \label{v:p2lh}\ee
 The maximum is   over admissible $n$-step paths $x_{0,n}=(x_k)_{k=0}^n$   that start at the origin $x_0=0$ and whose steps are denoted by  $z_k=x_k-x_{k-1}$.

The  point-to-point last-passage percolation with  restricted path length 
 is defined by 
\be\Gpp_{x,(n),y}=\max_{x_{0,n}:\, x_0=x,\,x_n=y}\sum_{k=0}^{n-1}V(\theta_{x_k}\w,\, z_{k+1}) ,    \qquad x\in\Z^d.  \label{v:gpp1}\ee
 The maximum is   over admissible $n$-step paths $x_{0,n}=(x_k)_{k=0}^n$   that start at $x$  and end at $y$.   
If $y$ cannot be reached from $x$ with an admissible $n$-step path then set $\Gpp_{x,(n),y}=-\infty$.   Our convention is $G_{x,(0),x}=0$. 

Any path $x_{0,n}=(x_k)_{k=0}^n$  with steps $z_{k+1}=x_{k+1}-x_k$  is a {\it geodesic} or a {\it maximizing path}  if 
\be\label{v:geod.01}    
\sum_{k=0}^{n-1}V(\theta_{x_k}\w,\, z_{k+1}) = \Gpp_{x_0,(n),\,x_n}. 
\ee

\begin{remark}    The number of steps in an admissible path from $x$ to $y$  is determined uniquely by $x$ and $y$ for all pairs $x,y$  iff $0$ does not lie in the affine hull of $\range$. 
 This is true for natural  directed examples such as 
$\range=\{e_1,e_2,\dotsc, e_d\}$. 
  Then we can  write  $\Gpp_{x,y}=\Gpp_{x,(n),y}$ where $n$ is the unique number of admissible steps from $x$ to $y$. 
\qedex\end{remark}

We take the existence of the limiting shape functions for granted, as stated in the next theorem. 

\begin{theorem}\label{v:t-lln2}  Let $\P$ be an i.i.d.\ product probability measure and assume \eqref{v:dir} and \eqref{v:V1}.  

{\rm (i)}
There exists a  nonrandom finite, convex, Lipschitz   function $\gpl:\R^d\to\R$ such that 
\be\label{v:p2line} 
\gpl(h)=  \lim_{n\to\infty} n^{-1}\Gpl_{n}(h)   \quad  \text{$\P$-a.s.}
\ee 

{\rm (ii)} There exists a nonrandom  finite, concave, continuous   function  $\gpp:\Uset\to\R$  such that 
\begin{align}\label{v:p2p}
	\gpp(\xi)=\lim_{n\to\infty}n^{-1}\Gpp_{0,(n),[n\xi]}, \qquad  \xi\in\Uset  
	\end{align}
where $[n\xi]$ is a point reachable in $n$ steps and approximately $n\xi$.   The limits satisfy 
\be\label{v:pl}    \gpl(h)=\sup_{\xi\in\Uset} \{ \gpp(\xi)+h\cdot\xi\}.  
\ee
\end{theorem}
The theorem above is a part of Theorem 2.4 in \cite{geor-rass-sepp-16}. 

 Here is a sketch of the argument for the  duality \eqref{v:pl}  between the  point-to-point and point-to-line shape functions. 
\begin{align*}
 \gpl(h) &=\lim_{n\to\infty} \frac1n \Gpl_n(h) =  
\lim_{n\to\infty} \;\max_{x_{0,n}: \,x_0=0}  \frac1n \Bigl\{ \;\sum_{k=0}^{n-1}V(\theta_{x_k}\w, z_{k+1}) + h\cdot x_n \Bigr\}  \\
 &=\lim_{n\to\infty} \,\max_x    \frac1n \bigl\{ \Gpp_{0,(n),x} + h\cdot x \bigr\} 
 =\lim_{n\to\infty} \;\sup_{\xi\in\Uset} \, \Bigl\{ \, \frac1n  \Gpp_{0,(n),[n\xi]}  +    h\cdot  \frac{[n\xi]}n  \Bigr\}   \\
 &= \sup_{\xi\in\Uset}\, \lim_{n\to\infty} \, \Bigl\{ \, \frac1n  \Gpp_{0,(n),[n\xi]}  +    h\cdot  \frac{[n\xi]}n  \Bigr\} =  \sup_{\xi\in\Uset}  \; \{ \,\gpp(\xi) + h\cdot\xi  \,\}  . 
\end{align*}
The interchange of the limit and the supremum is justified by a discretization argument and estimates of the kind used for positive temperature polymer models in Lemma 2.9 of \cite{rass-sepp-p2p}. 

We can extend equation \eqref{v:pl} to a convex duality by setting $\gpp\equiv-\infty$ on $\Uset^c$ so that then the supremum extends over all $\xi\in\R^d$.   Then   
by convex duality,  
\be\label{v:pl2}     \gpp(\xi)=\inf_{h\in\R^d} \{  \gpl(h)-h\cdot\xi\}, \qquad \xi\in\Uset.  
\ee
Let us say that $\xi\in\Uset$ and $h\in\R^d$ are {\it dual}  if 
\[  \gpl(h)=  \gpp(\xi)+h\cdot\xi. \]

\begin{lemma} \label{v:lm-xih}  Every $\xi\in\ri\Uset$ has a dual $h\in\R^d$. 
\end{lemma}

The lemma is proved by arguing that the infimum in \eqref{v:pl2} can be restricted to a compact set.  See Lemma 4.3 in \cite{geor-rass-sepp-16}. 

In order to develop  variational formulas for the limits $\gpp$ and $\gpl$,   we introduce a class of stationary processes we call cocycles and state an ergodic theorem for them. 

\medskip 

\subsection{Stationary cocycles}  

\begin{definition}[Cocycles]
\label{def:cK}
A measurable function $B:\Omega\times\Z^d\times\Z^d\to\R$ is   a {\rm stationary  cocycle}
if it satisfies these two   conditions for $\P$-a.e.\ $\w$ and  all $x,y,z\in\Z^d$:
\begin{align*}
B(\w, x+z,y+z)&=B(\theta_z\w, x,y)  \qquad \text{{\rm(}stationarity{\rm)}}  \\
B(\w,x,y)+B(\w,y,z)&=B(\w, x,z) \qquad  \;\  \ \text{{\rm(}additivity{\rm)}.}
\end{align*}
$\cK$ denotes the space of stationary   cocycles $B$ such that  $\E\abs{B(x,y)} <\infty$   $\forall x,y\in\Z^d$.  $\cK_0$ denotes  the subspace of  $F\in\cK$ such that   $\E[F(x,y)]=0$  $\forall x,y\in\Z^d$. 
\end{definition}

An equivalent way to define a stationary cocycle is as a shift-invariant gradient of a function on $\Omega\times\Z^d$.   For such a function $f$, define $\nabla_uf(\w,x)=f(\w,x+u)-f(\w,x)$.  Say that $f$ has a {\it shift-invariant    gradient}  if  $\nabla_uf(\w,x)=\nabla_uf(\theta_x\w,0)$  for all $u,x\in\Z^d$. 

\begin{lemma}    Let  $B$ be a real function on $\Omega\times\Z^d\times\Z^d$.  Then $B$ is a stationary cocycle if and only if there exists   a real function  $f$ on $\Omega\times\Z^d$ with a  shift-invariant    gradient  and such that  $B(\w,x,y)=\nabla_{y-x}f(\w,x)=f(\w,y)-f(\w,x)$ for all $x,y\in\Z^d$. 
\end{lemma}

\begin{proof} Suppose $f$ exists.  Additivity of $B$ is immediate.  For stationarity, 
\begin{align*}
B(\w, u+x,u+y) =  \nabla_{y-x}f(\w,u+x)=   \nabla_{y-x}f(\theta_u\w, x)=B(\theta_u\w, x,y). 
\end{align*} 

Conversely, suppose $B$ is a stationary cocycle.  Set $f(\w,x)=B(\w, 0,x)$.   Then $f(\w,0)=B(\w,0,0)=0$.   
\begin{align*}
\nabla_uf(\w,x)&=f(\w,x+u)-f(\w,x)=B(\w,0, x+u)-B(\w,0,x)= B(\w,x, x+u)\\
&= B(\theta_x\w,0, u)=f(\theta_x\w, u)=f(\theta_x\w, u)-f(\theta_x\w, 0) =\nabla_uf(\theta_x\w, 0). 
\end{align*} 
Thus $f$ has a shift-invariant gradient.  
\end{proof} 

The second part of the proof above illustrates that the  first lattice  variable in our definition of a cocycle is superfluous.    If we put $\wt B(\w,y)=B(\w, 0,y)$ then $B(\w,x,y)=B(\theta_x\w, 0, y-x)=\wt B(\theta_x\w, y-x)$. 

A special class of cocycles is given by gradients   
$B(\w,x,y)=\varphi(\theta_y\w)-\varphi(\theta_x\w)$ of functions $\varphi$ on $\Omega$.  The next lemma shows that   $\cK_0$ is the $L^1(\P)$-closure of such gradients.   It can also be found as Proposition 3 in \cite{boiv-derr}. 

\begin{lemma}  Given $F\in\cK_0$, there exist $\varphi_n\in L^1(\P)$ such that for each fixed $z\in\Z^d$, 
$\varphi_n(\theta_z\w)-\varphi_n(\w)\to F(\w,0,z)$ as $n\to\infty$,  both almost surely and in $L^1(\P)$.    If  $1<p<\infty$ and $F(x,y)\in L^p(\P)$ for all $x,y\in\Z^d$, then the convergence   holds  also  in $L^p(\P)$. 
\end{lemma}
\begin{proof}   Let $\Lambda_n=[-n,n]^d\cap\Z^d$ and  
\[  \varphi_n(\w)= -\,\frac{1}{\abs{\Lambda_n}}\sum_{x\in\Lambda_n}  F(\w,0,x). \]
Then 
\begin{align*}
&\varphi_n(\theta_z\w)-\varphi_n(\w) =  \frac1{\abs{\Lambda_n}}\sum_{x\in\Lambda_n}\bigl[ -F(\w, z,x+z)+  F(\w,0,x)\bigr] \\
&\qquad =  \frac1{\abs{\Lambda_n}}\sum_{x\in\Lambda_n}\bigl[  F(\w,0,z)-F(\w, x,x+z)\bigr] 
=  F(\w,0,z) -  \frac1{\abs{\Lambda_n}}\sum_{x\in\Lambda_n} F(\theta_x\w, 0,z). 
\end{align*}
The second equality is true term by term by additivity.   As $n\to\infty$, the last term converges to $\E F(0,z)=0$  by the multidimensional ergodic theorem (for example, almost sure convergence by the multidimensional  pointwise ergodic theorem given in Theorem 14.A8 in \cite{geor} and then $L^p$ convergence by truncation).  
\end{proof}

Our convention for centering non-mean-zero cocycles is the following.   For $B\in\cK$     there exists a vector $h(\B)\in\R^d$ such that   
 \be \E[\B(0,x)]=-h(\B)\cdot x  \qquad \forall x\in\Z^d.    \label{EB}\ee
   Existence of $h(B)$ follows  because $c(x)=\E[\B(0,x)]$ is an additive function  on the 
  group $\Z^d$.  
Then 
\be F(\w, x,y)=  -h(\B)\cdot (y-x)-\B(\w, x,y), \qquad x,y\in\Z^d \label{FF}\ee
 is a centered  stationary $L^1(\P)$  cocycle.  

 Consider this  assumption on a given $F\in\cK_0$.
 \be \label{cc-ass1}  \begin{aligned}  
 &\text{$\exists$ $\overline F:\Omega\times\range\to\R$ such that the following properties hold for $z\in \range\setminus\{0\}$ and $\P$-a.s.: }   \\
 &  \text{\phantom{such that the following }}   F(\w,0,z)\le \overline F(\w,z)    \\
&\text{and}\\ 
& \text{\phantom{such that the  }}  \varlimsup_{\delta\searrow0 } \varlimsup_{n\to\infty} \max_{\abs{x}_1\le n} \frac1n \sum_{0\le i\le n\delta}  \abs{\overline F(\theta_{x+iz}\w, z)} =0 . 
\end{aligned}\ee 

A sufficient condition for the limit above is that $\E\abs{\overline F(\w, z)}^{d+\e} <\infty$ for some $\e>0$ and the shifts of $\overline F$ have  finite range of dependence.  The latter condition means that there exists  $r_0<\infty $ such that the random vectors 
$\{(\overline F(\theta_{x_i}\w, z))_{z\in\range}: 1\le i\le m\}$ are independent whenever $\abs{x_i-x_j}\ge r_0$ for each pair $i\ne j$.  

\begin{theorem}  \label{cc-ergthm} Let $F\in\cK_0$.  Under  assumption \eqref{cc-ass1}   we have the following uniform ergodic theorem:
\[ \lim_{n\to\infty}  \max_{\abs{x}_1\le n} \frac{\abs{F(\w,0,x)}}n \;=\; 0 \qquad\P\text{-a.s.} \]

\end{theorem} 
For a proof see Appendix A.3 of \cite{geor-rass-sepp-yilm-15}.  

\medskip 

\subsection{Variational formulas}  

In this section we derive variational formulas for the restricted path length  point-to-level and point-to-point last-passage values. 

\begin{theorem}\label{th:K-var}
 	\begin{align}
	\gpl(h)=\inf_{F\in\cK_0} \,\P\text{-}\esssup_\w\;  \max_{z\in\range} \{V(\w,z)+ h\cdot z+F(\w, 0,z)\}.\label{eq:g:K-var}
	\end{align}
A minimizing 	$F\in\cK_0$ exists for each $h\in\R^d$.  
 \end{theorem}
 
Abbreviate 
\be\label{cc-KF}   K(F)=  \P\text{-}\esssup_\w\;  \max_{z\in\range} \{V(\w,z)+ h\cdot z+F(\w, 0,z)\} .\ee

\begin{proof}  

{\bf Upper bound.}    Let $F\in\cK_0$.   Assume  $K(F)<\infty$.  Then 
\[   F(\w, 0,z)\le -V(\w,z) -h\cdot z+K(F)  \]
together with  assumption \eqref{v:V1}  on  $V$ imply that $F$ satisfies assumption \eqref{cc-ass1} and therefore the uniform ergodic theorem  (Theorem \ref{cc-ergthm})  applies. 
 \begin{align*}
\gpl(h)&=\lim_{n\to\infty} \max_{x_{0,n}} \frac1n\Bigl\{ \;\sum_{k=0}^{n-1} V(\theta_{x_k}\w, z_{k+1}) +h\cdot x_n\Bigr\}  \\
&=\lim_{n\to\infty} \max_{x_{0,n}} \frac1n\Bigl\{ \;\sum_{k=0}^{n-1} V(\theta_{x_k}\w, z_{k+1}) +h\cdot x_n + F(\w, 0, x_n) \Bigr\}  \\
&=\lim_{n\to\infty} \max_{x_{0,n}} \frac1n\sum_{k=0}^{n-1} \bigl[ V(\theta_{x_k}\w, z_{k+1}) +h\cdot z_{k+1}  + F(\theta_{x_k} \w, 0, z_{k+1}) \bigr]  \\
&\le K(F)
\end{align*}  
because  the last upper  bound is valid $\P$-a.s. for each term. 
We have shown that 
\[  \ddd   \gpl(h)\le \inf_{F\in\cK_0}  K(F).   \]

{\bf Lower bound.}    Let $\lambda>\gpl(h)$.  
Set $  u_n(\w) = e^{G_n(h)-n\lambda} $ with the interpretation that $u_0=1$.  Since  $n^{-1}G_n(h)\to\gpl(h)$ almost surely,   $  u_n(\w) < e^{-n\e} $  for large $n$ and a fixed $\e>0$.   
Hence   $f$  below is a  well-defined finite function: 
\begin{align*}    f(\w)\; & = \; \sum_{n=0}^\infty  u_n(\w) 
\; = \;  1+ \sum_{n=1}^\infty  \exp\biggl\{\max_{x_{0,n}}  \biggl[   V(\w, x_1)  +h\cdot x_1 - \lambda  \\
&\qquad \qquad\qquad\qquad\qquad
 +   \sum_{k=1}^{n-1} V(\theta_{x_k}\w, z_{k+1})  + h\cdot (x_n-x_1)  - (n-1)\lambda 
  \biggr]\biggr\} \\
  &\ge\;  1+   \max_z  e^{ V(\w, z)  + h\cdot z -  \lambda  }  \sum_{n=1}^\infty  u_{n-1}(\theta_z\w) 
  \; \ge \;   \max_z  e^{  V(\w, z)  + h\cdot z -  \lambda   }   f(\theta_z\w) \\
   &= \; e^{-\lambda }  \cdot  e^{   \max\limits_z  [ V(\w, z)  + h\cdot z  +  \log   f(\theta_z\w)]}. 
\end{align*}
In the first inequality above we take the sum inside the maximum over the first step $z$. 
Rearrange this to 
\begin{align*} \lambda  &\ge    \max_z   \{ V(\w, z)  + h\cdot z  + \log   f(\theta_z\w) - \log   f(\w)\}
\qquad \text{a.s.} \end{align*} 
Utilizing the notation $\nabla\log f(\w,x,y)= \log   f(\theta_y\w) - \log   f(\theta_x\w)$,  we can now write
\[  \lambda  \ge  K(\nabla\log f)   \ge \inf_{F\in\cK_0} K(F)  
\]
where the second inequality holds provided   $\nabla\log f\in\cK_0$.   This point   is implied by the next lemma.   Hence we can let $\lambda\searrow\gpl(h)$ to get 
\[   \gpl(h)   \ge \inf_{F\in\cK_0} K(F) . \]  
The  existence of a minimizer is proved by a weak convergence argument that we skip.   It is given for positive temperature polymer models in Theorem 2.3 of \cite{rass-sepp-yilm-17}. 
\end{proof}

\begin{lemma}
For  a measurable function $\varphi:\Omega\to\R$ define $\nabla\varphi(\w,x,y)=\varphi(\theta_y\w)-\varphi(\theta_x\w)$.     Then  $K(\nabla\varphi)<\infty$ implies $\nabla\varphi(\cdot\,,z)\in L^1(\P)$ and  $\E[\nabla\varphi(\cdot\,,z)]=0$,  for all  $z\in\range$.  
\end{lemma}

\begin{proof}  For each $z\in\range$,   
\[  V(\w, z)+ h\cdot z+ \nabla \varphi(\w, 0,z) \le  K(\nabla\varphi)<\infty   \qquad \text{a.s.}   \]
and so $(\nabla\varphi(\cdot\,,z))^+\in L^1(\P)$.  
Suppose  $(\nabla\varphi(\cdot\,,z))^-\notin L^1$.    Then a  contradiction arises as follows, where the first equality comes by the pointwise ergodic theorem: 
\begin{align*}
-\infty &=  \lim_{n\to\infty} \frac1n \sum_{k=0}^{n-1}   \nabla\varphi(\theta_{kz}\w,z) 
=  \lim_{n\to\infty} \frac1n \sum_{k=0}^{n-1}   [ \varphi(\theta_{(k+1) z}\w)  -  \varphi(\theta_{kz}\w) ]   \\
& =  \lim_{n\to\infty} \frac1n    [ \varphi(\theta_{n z}\w)  -  \varphi(\w) ] 
= 0 \qquad \text{in probability.} 
\qedhere  \end{align*}
Once we know  $\nabla\varphi(\cdot\,,z)\in L^1(\P)$,    the ergodic theorem and the calculation above give $\E[\nabla\varphi(\cdot\,,z)]=0$. 
\end{proof}


\begin{theorem}\label{v:th-var-pp}
For each $\xi\in\ri\Uset$  we have this variational formula. 
	\begin{align}
	\gpp(\xi)&\;=\;\inf_{B\in\cK}\, \P\text{-}\esssup_\w\;  \max_{z\in\range}\, \{V(\w,z)-   B(\w, 0,z)-  h(B)\cdot \xi\}.  \label{v:var-pp}
	\end{align}
For each $\xi\in\ri\Uset$   there exists   a minimizing $B\in\cK$  such that $h(B)$ is dual to $\xi$.   
 \end{theorem}

\begin{proof}  From  duality \eqref{v:pl2},
\begin{align*}
\gpp(\xi) &=\inf_{h\in\R^d}  \{   \gpl(h)-h\cdot\xi\} \\
& =\inf_{h\in\R^d}     \inf_{F\in\cK_0} \,\P\text{-}\esssup_\w\;  \max_{z\in\range} \{V(\w,z)+ h\cdot z+F(\w, 0,z)-h\cdot\xi\} \\
& =    \inf_{B\in\cK} \,\P\text{-}\esssup_\w\;  \max_{z\in\range} \{V(\w,z)-B(\w, 0,z)-h(B)\cdot\xi\}
\end{align*}
where we let  \[ B(\w,x,y)=-h\cdot(y-x)-F(\w,x,y)\]  define  an element of $\cK$ with 
\[  h(B)\cdot z=-\E[B(0,z)]\cdot z=h\cdot z 
\quad\Longrightarrow \quad h(B)=h. \]

The first infimum above is achieved at some $h$ dual to $\xi$ whose existence is given in Lemma \ref{v:lm-xih}, and for this $h$  a minimizing $F$ exists for $\gpl(h)$.  Thus the $B$ defined above has $h(B)$ dual to $\xi$.  
\end{proof}

\subsection{Cocycles adapted to the potential} 

  The cocycles and the potential $V$ that defines the percolation both live on a general product space with a product probability measure $\P$.   It is not evident how these two structures are connected.   Next we identify a local  condition that characterizes  those cocycles that are relevant to the percolation problem in various ways.  
  
 
\begin{definition}
 A cocycle $B\in\cK$ is adapted to the potential $V$ if 
\be\label{v:VB}   \max_{z\in\cR} \, [ V(\w,z)-B(\w,0,z)]=0 \quad \P\text{-a.s.} \ee
 \end{definition} 
 
 This condition is linked to (i) minimizing cocycles, (ii) geodesics, (iii) Busemann functions and (iv) stationary percolation.  We discuss briefly these four issues  before moving to an exactly solvable model. 

\medskip 

{\bf Minimizing cocycles.}  
Suppose $B\in\cK$ satisfies \eqref{v:VB}.  Define a mean-zero cocycle 
  $F\in\cK_0$ by 
\[   F(\w, x,y)=  -h(\B)\cdot (y-x)-\B(\w, x,y)  .    \]
Then \eqref{v:VB} becomes 
\[   0=  \max_{z\in\cR} [ V(\w,z)+h(B)\cdot z + F(\w,0,z)]  \quad \P\text{-a.s.}   \]
The one-sided bound $B(\w,0,z)\ge V(\w,z)$ is enough for the uniform ergodic theorem to work for $F$.  
Thus we can iterate the identity above and take a limit. 
\begin{align*}
  0&=  \max_{x_{0,n}} \Bigl\{ \; \frac1n\sum_{k=0}^{n-1} V(\theta_{x_k}\w, z_{k+1})   +\frac1n h(B)\cdot x_n + \frac1n F(\w,0,x_n) \Bigr\}  \\
  &=  \frac1n G_n(h(B)) + o(1)  
  \ \longrightarrow\  \gpl(h(B)) \quad\text{as $n\to\infty$.} 
\end{align*}  
We can conclude that 
\[   \gpl(h(B))=0=  \max_{z\in\cR} \bigl[ V(\w,z)+h(B)\cdot z + F(\w,0,z) \bigr]   \quad \P\text{-a.s.}   \]
The equality above shows  that $F$ minimizes in the variational formula 
\[   \gpl(h)=\inf_{F\in\cK_0} \,\P\text{-}\esssup_\w\;  \max_{z\in\range} \{V(\w,z)+ h\cdot z+F(\w, 0,z)\}  \]
 for $h=h(B)$,  even without the essential supremum over $\w$.    
 
 Furthermore: suppose $h(B)$ and $\xi$ are dual.   Then  from above, almost surely,   
 \begin{align*}
 \gpp(\xi)=\gpl(h(B))-h(B)\cdot\xi
 =-h(B)\cdot\xi =   \max_{z\in\cR} [ V(\w,z)-B(\w,0,z)] -h(B)\cdot\xi. 
 \end{align*}
 Thus $B$ is a minimizer for $\gpp(\xi)$.  
 
 In summary, we see that \eqref{v:VB} is a criterion that finds  cocycles that serve as minimizers in the variational formulas.    For future use we record the outcome of the calculation above in the next lemma. 
 
 \begin{lemma}\label{lm-h(B)}   Let $B\in\cK$ be an integrable stationary cocycle adapted to the potential $V$ as required by  \eqref{v:VB} and let $h(B)$ be the negative of the mean vector of $B$ as defined in \eqref{EB}.  Then $\gpl(h(B))=0$.  
 
 \end{lemma}

\medskip 
 
{\bf Geodesics.}  Turns out that a cocycle $B$ satisfying  \eqref{v:VB} is involved  not only in optimization on the macroscopic  level but also on the   pathwise level.  
 Suppose the path $(x_k)_{k=0}^n$  with steps $z_{k+1}=x_{k+1}-x_{k}$ follows the maximal increments specified in \eqref{v:VB}, in other words,  satisfies 
\be\label{v:geod} V(\theta_{x_k}\w,z_{k+1})   -B(\w, x_k, x_{k+1})=0\qquad  \forall k=0,1,\dotsc,n-1. \ee
  Then this path is a geodesic from $x_0$ to $x_n$.  Here is the simple argument.  Consider  any path $y_{\bbullet}$ from $y_0=x_0$ to $y_n=x_n$.  Then, by \eqref{v:VB},  the stationarity and additivity of $B$, and \eqref{v:geod},   
\begin{align*}
\sum_{k=0}^{n-1} V(\theta_{y_k}\w, y_{k+1}-y_k) &\le  \sum_{k=0}^{n-1} B(\theta_{y_k}\w, 0,   y_{k+1}-y_k)  =  \sum_{k=0}^{n-1} B(\w, y_k , y_{k+1}) = B(\w, x_0, x_n) \\
&=  \sum_{k=0}^{n-1} B(\w, x_k , x_{k+1})  =  \sum_{k=0}^{n-1} V(\theta_{x_k}\w,z_{k+1}). 
\end{align*}

 \medskip 
 
{\bf Busemann functions.}    Having seen the usefulness of condition \eqref{v:VB}, we must ask  how  cocycles that satisfy \eqref{v:VB} arise.   One way of obtaining such cocycles is through limits of local gradients of passage times,  called {\it Busemann functions}.   

 Suppose  that we are in a setting where admissible paths that connect  two given points have a uniquely determined number of steps.   Let $\Gpp_{x,y}=\Gpp_{x,(m),y}$ denote the point-to-point last-passage value  where $m$ is the unique number of steps from $x$ to $y$.   Fix a direction $\xi\in \ri\Uset$.   Assume that  we have a process $B$ that satisfies the following:  for all sequences $\{v_n\}\subset\Z^d$ such that  $v_n/\abs{v_n}\to\xi$    this almost sure limit holds:  
 \be\label{v:B1}    B(\w,x,y)=\lim_{n\to\infty} [ \Gpp_{x,\,v_n}-\Gpp_{y,\,v_n}]  \quad \forall x,y\in\Z^d  \quad \P\text{-a.s.}  \ee
$B$ is called a Busemann function in direction $\xi$.  It  is a stationary cocycle.  Additivity is immediate from the limit \eqref{v:B1}.  Stationarity comes from the fact that shifting $v_n$ by a fixed amount does not alter its limiting direction $\xi$.  Under additional assumptions (see for example  Theorem 5.1 in \cite{geor-rass-sepp-16}) this cocycle is integrable. 

At this time we simply wish to observe that $B$ satisfies \eqref{v:VB}: 
 \begin{align*}
  \max_{z\in\cR} [ V(\w,z)-B(\w,0,z)] &=  \lim_{n\to\infty}   \max_{z\in\cR} [   V(\w,z)-\Gpp_{0,\,v_n}+\Gpp_{z,\,v_n}] =0 
 \end{align*}
 where the last equality follows from 
 \be\label{v:199}  \Gpp_{0,\,v_n} = \max_{z\in\cR} [   V(\w,z) +\Gpp_{z,\,v_n}] . \ee 
  
 Proof of the limit in \eqref{v:B1} is highly nontrivial.   The route that we will take to finding cocycles that satisfy \eqref{v:VB} will involve the next item below. 
  
   
 \medskip   
   
 {\bf Stationary percolation.}  
 Once again we assume  cocycle $B$ satisfies  \eqref{v:VB}   and develop this identity in a different direction.   Let   $v\in\Z^d$ be fixed.   
\begin{align*}
0&=  \max_{z\in\cR} [ V(\theta_x\w,z)-B(\theta_x\w,0,z)] 
=   \max_{z\in\cR} [ V(\theta_x\w,z)-B(\w,x,x+z)]  \\
&=   \max_{z\in\cR} [ V(\theta_x\w,z)-B(\w,x,v) +B(\w, x+z, v)]  \\
\end{align*} 
from which we write 
\be\label{v:B2}  B(\w,x,v) = \max_{z\in\cR} \,[ V(\theta_x\w,z)  +B(\w, x+z, v)]
\qquad \forall x,v\in\Z^d.  \ee
Iterate   this.   Fix also $u\in\Z^d$ and  write  $x_n=u+ z_1+\dotsm+z_n$ for an admissible path from $u$ with steps $z_i$.  
\be\label{v:B7} \begin{aligned}
B(\w,u,v) &= \max_{z_1\in\cR} [ V(\theta_{x_0}\w,z_1)  +B(\w, x_1, v)] \\
&= \max_{z_1, z_2\in\cR} [ V(\theta_{x_0}\w,z_1) + V(\theta_{x_1}\w,z_2)  +B(\w, x_2, v)] \\
&=\dotsm=
 \max_{z_1,\dotsc,  z_n\in\cR} \Bigl[\, \sum_{k=0}^{n-1} V(\theta_{x_k}\w,z_{k+1})   +B(\w, x_n, v)\Bigr]\\
 &=\max_x \,[  \,G_{u,(n),x}  +B(\w, x, v)].  
\end{aligned}\ee

We can turn this into a boundary value problem.  Assume again that the number of steps on an admissible path is determined uniquely by the endpoints so that we can write 
\[ \Gpp_{x,y}=\begin{cases}  \Gpp_{x,(n),y}  &\text{if $y$ is reachable from $x$ along an admissible path of $n$ steps}\\ -\infty &\text{if $y$ is not reachable from $x$ along any  admissible path.} 
\end{cases} \] 
  Let  $\cH$ and $\partial\cH$ be finite  subsets of $\Z^d$  with the property that any infinite admissible path from $\cH$ intersects $\partial\cH$.   So in a sense $\partial\cH$ is a ``boundary'' of $\cH$.  

For example, suppose we are in the directed case $\range=\{e_1,\dotsc,e_d\}$.   If   $\cH$ is the rectangle   $\cH=\prod_{i=1}^d \{0,1,\dotsc, N_i\}$, then $\partial\cH$ could be its ``northeast'' boundary 
$  \partial\cH=\bigcup_{i=1}^d \{x\in\cH:  x_i=N_i\}. $

\begin{lemma}  Assume that the two endpoints of any admissible path determine uniquely the number of steps in the path.  Assume that the stationary cocycle $B$ satisfies  \eqref{v:VB}.  Fix $v\in\Z^d$ and finite subsets  $\cH$ and $\partial\cH$ of $\Z^d$ such that every long enough admissible path from $\cH$ intersects $\partial\cH$.   Then 
\be\label{v:H1} 
B(\w,u,v) = \max_{x\in\partial\cH}  \,[  \,G_{u,x}   +B(\w, x, v)] 
\qquad\text{for all $u\in\cH$.}    
\ee

\end{lemma}

\begin{proof}   Fix $u\in\cH$.   Equation  \eqref{v:B7}  gives 
\[  B(\w, u, v) \ge  G_{u,x}   +B(\w, x, v)   \]
whenever there is an admissible path from $u$ to $x$.    This tells us that $\ge$ holds in \eqref{v:H1}.  It also shows that if $u\in\cH\cap\partial\cH$ then the maximum in \eqref{v:H1} is assumed at $x=u$ and the equality holds.   (For this last statement we use the convention $G_{u,u}=0$.) 

We prove inequality $\le$ in \eqref{v:H1}.  
  By the directedness assumption \eqref{v:dir}  we can take $n$ in \eqref{v:B7} large enough so that every $n$-path from $u$ intersects $\partial\cH$.    Fix  a maximizing path $u=x_0,x_1,\dotsc,  x_n$ on the second last line of \eqref{v:B7}.  Let $x_m\in\partial\cH$.  Then 
\begin{align*}
B(\w,u,v) &=\sum_{k=0}^{m-1} V(\theta_{x_k}\w,z_{k+1})  + \sum_{k=m}^{n-1} V(\theta_{x_k}\w,z_{k+1})   +B(\w, x_n, v) \\
&\le G_{u,x_m} +  G_{x_m,x_n}  +B(\w, x_n, v)  \le G_{u,x_m} +   B(\w, x_m, v)  \\[2pt] 
&\le \text{right-hand side of \eqref{v:H1}} . 
\end{align*}
In the second-last inequality above we applied \eqref{v:B7} with $x_m$ in place of $u$ and $n-m$ in place of $n$. 
\end{proof}

Our interpretation is that  equation \eqref{v:H1} determines the values  $\{ B(u,v):u\in\cH\}$  from the last-passage percolation  $G_{x,y}$  and  given boundary values  $\{ B(x,v):x\in\partial\cH\}$.  Then we can search for cases where the solution is tractable.  In particular,  we can look for distributional invariance.    We can make this program work in exactly solvable cases. 

 The analogy the reader should have in mind is finding the invariant distribution of a Markov process such as an interacting particle system.    The analogue of boundary values  are the state variables at time zero.  The analogue of the weights $V( \theta_x\w,z)$    are the Poisson clocks or other random variables that govern the evolution of the particles. 



\section{Exponential corner growth model in two dimensions} 
\label{v:s-stat-cgm} 

For the remainder of these notes we restrict  the discussion   to the   two-dimensional corner growth model (CGM), with real weights on the vertices:  $\w=(\w_x)_{x\in\Z^2}\in\Omega=\R^{\Z^2}$. The set of  admissible steps is $\range=\{e_1,e_2\}$, and the  potential is given by the weight at the origin: $V(\w,z)=\w_0$.    The set of possible limiting velocities of paths is the  closed line segment $\cU=[e_1,e_2]=\{(s,1-s): 0\le s\le 1\}$, and its relative interior is the open line segment $\ri\cU=(e_1,e_2)=\{(s,1-s): 0< s< 1\}$. 

In this setting we alter slightly the earlier definition \eqref{v:gpp1}  of the point-to-point last-passage time to include both endpoints of the path.  This makes no difference to large-scale properties.  
 Given an environment $\w$ and two points $x,y\in\Z^2$ with $x\le y$ coordinatewise,
define  
 	\be\label{v:G}  
	  \Gpp_{x,y}=\max_{x_{\brbullet}\,\in\,\Pi_{x,y}}\sum_{k=0}^{\abs{y-x}_1}\w_{x_k}.
	\ee
	$\Pi_{x,y}$ is the set of   paths $x_{\bbullet}=(x_k)_{k=0}^n$  that start at  $x_0=x$,  end at $x_n=y$ with $n=\abs{y-x}_1$,  and have increments $x_{k+1}-x_k\in\{e_1,e_2\}$.   Call such paths {\it admissible} or {\it up-right}.  See Figure \ref{fig:cgm}  for an illustration.   The zero-length path case  is $G_{x,x}=\w_x$.    Our convention is that 
\be\label{v:Gfail} 	 \Gpp_{x,y}=-\infty\qquad   \text{if $x\le y$ fails. } \ee

\begin{figure}[h] 
	\begin{center}
		\begin{tikzpicture}[>=latex, scale=0.65]  

			\definecolor{sussexg}{gray}{0.7}
			\definecolor{sussexp}{gray}{0.4} 

			\draw[<->](0,4.8)--(0,0)--(5.8,0);
			
			\draw[line width = 3.2 pt, color=sussexp](0,0)--(0,1)--(2,1)--(2,3)--(3,3)--(3,4)--(5,4);
			
			\foreach \x in { 0,...,5}{
				\foreach \y in {0,...,4}{
					\fill[color=white] (\x,\y)circle(1.8mm); 
					
					\draw[ fill=sussexg!65](\x,\y)circle(1.2mm);
				}
			}	
			
			\foreach \x in {0,...,5}{\draw(\x, -0.3)node[below]{\small{$\x$}};}
			\foreach \x in {0,...,4}{\draw(-0.3, \x)node[left]{\small{$\x$}};}	
			
		\end{tikzpicture}
	\end{center}
	\caption{ \small  An up-right path from $(0,0)$ to $(5,4)$ on the lattice $\Z^2$. }\label{fig:cgm} 
  \end{figure}

  We work with the exponentially distributed weights, and so make the following assumption: 
\be\label{ass6} \begin{aligned}   
\text{the weights $\w_x$ are independent  rate  1 exponentially distributed random variables.}    
\end{aligned}\ee 
This means that  $\P\{\w_x>t\}=e^{-t}$ for $t\ge 0$.    This is abbreviated as $\w_x\sim$ Exp(1).   

In this section   we construct a stationary version of the exponential CGM  and then calculate a closed-form expression of  the  limiting shape function   
 \be\label{lln9}
\gpp(\xi)= \lim_{N\to\infty} N^{-1} \Gpp_{0,\fl{N\xi}}  
\qquad\text{almost surely for $\xi\in\R_{\ge0}^2$.}  
 \ee
 Existence of the limit  is a special case of the previous  Theorem \ref{v:t-lln2}.  Note though that in contrast with \eqref{v:p2p} the limit is now defined not only for $\xi\in\Uset$ but for all $\xi$ in the closed first quadrant.   (Theorem 2.3 of \cite{mart-04} gives \eqref{lln9}  under a slightly weaker hypothesis than the $2+\e$ moment assumption \eqref{v:V1} assumed for Theorem \ref{v:t-lln2}.)   
General soft arguments show that   $\gpp$ is concave, continuous and homogeneous (which means that $\gpp(c\xi)=c\gpp(\xi)$ for $c\ge 0$).   In the exponential case  the explicit formula for $\gpp$ reveals that this function is strictly concave, except on rays from the origin.

In the remainder of this section we prove two complements to the limit \eqref{lln9}. 
    We prove a version of \eqref{lln9} with uniformity over the lattice.   Based on this we obtain  the    {\it shape theorem},   which gives the limit  of the randomly growing cluster $A_t=\{\zeta\in\R_{\ge0}^2: \Gpp_{0,\,\fl{\zeta}}\le t\}$ as $t\to\infty$ and the cluster is scaled to $t^{-1}A_t$.     As the last theorem of the section we show  that the maximizing path   in \eqref{v:G}  stays within $o(\abs{y-x}_1 )$ of the straight line segment $[x,y]$,  with high probability as $\abs{y-x}_1\to\infty$.

%

\subsection{Stationary exponential corner growth model}
Our first task is to construct a coupling of i.i.d.\ rate 1 exponential weights and  a  stationary integrable cocycle,  for a given value of a parameter $0<\rho<1$, that together satisfy \eqref{v:VB},   essentially by solving the boundary value problem in \eqref{v:H1}.    This construction will be performed on quadrants $u+\Z_{\ge0}^2$ with a specified origin $u\in\Z^2$.  

Fix a parameter $0<\rho<1$ and  an origin $u\in\Z^2$.   Assume  given a    collection  of  mutually independent random variables  
\be\label{IJw}  \{   \w_{x} ,\, I_{u+ie_1}, \, J_{u+je_2}  :  \,x\in u+\Z_{>0}^2, \, i,j\in \Z_{>0}\}  \ee
    with these marginal distributions:
\be\label{w7}\begin{aligned}
\w_{x}\sim \text{ Exp}(1),  \quad I_{u+ie_1}\sim \text{ Exp}(1-\rho)  , \quad \text{ and }  \quad  
J_{u+je_2}\sim \text{ Exp}(\rho)  .  
\end{aligned}\ee 
The interpretation is that $I_x$ is a weight for the edge $(x-e_1,x)$, $J_x$ is a weight for edge $(x-e_2,x)$, and $\w_x$ is a vertex weight.   The edge weights are on the boundary of the quadrant $u+\Z_{\ge0}^2$  and the vertex weights in the bulk.

From the 
given  variables \eqref{IJw} 
we define further   variables as follows, proceeding  inductively to the north and east from the origin $u$: for all   $x\in u+\Z_{> 0}^2$, 
\begin{align}
\label{IJw5.1}  \wc\w_{x-e_1-e_2}&=I_{x-e_2}\wedge J_{x-e_1} 
\\
\label{IJw5.2}  I_x&=\w_x+(I_{x-e_2}-J_{x-e_1})^+ 
\\
\label{IJw5.3} J_x&=\w_x+(I_{x-e_2}-J_{x-e_1})^-. 
\end{align}  
The mapping above from $(\w_x, I_{x-e_2},J_{x-e_1})$ to $(\wc\w_{x-e_1-e_2}, I_x, J_x)$ is illustrated in Figure \ref{fig:ne}.  Note that \eqref{IJw5.2}--\eqref{IJw5.3} imply the symmetric counterpart of \eqref{IJw5.1}
\be\label{IJw5.4} \w_{x}=I_{x}\wedge J_{x} \ee
and the additivity around the unit square: 
\be\label{IJw5.5}   I_{x-e_2}+J_x=J_{x-e_1}+I_x. \ee

 \begin{figure}[h]
 \begin{center}
 
 \begin{tikzpicture}[>=latex, scale=0.7]
\draw(0,0) rectangle (3,3);
\draw(6,0) rectangle (9,3); 
\draw[|->](3.7,1.5)--(5.3, 1.5);
\draw[color=black] (4.5, 1.5)node[above]{\small{\eqref{IJw5.1}--\eqref{IJw5.3}}};
\shade[ball color=black](3,3) circle(2mm) node[above right]{\large$\w_x$};
\draw (1.5, 0)node[below]{\large$I_{x-e_2}$};
\draw(0,1.5)node[left]{\large$J_{x-e_1}$}
;
\draw [line width= 4pt, color=black](0.2,0)--(2.8,0);
\draw [line width= 4pt, color=black](0,.2)--(0,2.8);

\shade[ball color=black](6,0) circle(2mm); 
\draw (6.2,-.3)node{\large$\wc\w_{x-e_1-e_2}$};
\draw [line width= 4pt, color=black](6.2,3)--(8.8,3);
\draw (7.5,3)node[above]{\large$ I_{x}$};
\draw [line width= 4pt, color= black](9,2.8)--(9,.2); 
\draw (9, 1.5)node[ right]{\large$J_{x}$};

\end{tikzpicture}
 
 \end{center}
 \caption{\small Mapping \eqref{IJw5.1}--\eqref{IJw5.3}  on a single  lattice square. The figure illustrates how southwest corners are flipped into northeast corners in the  inductive construction of the increment variables.}
 \label{fig:ne}
 \end{figure}

 Utilizing  \eqref{IJw5.1}--\eqref{IJw5.3}  we  extend  \eqref{IJw} to the larger collection 
 \be\label{IJw9} 
  \{ \w_{x} ,\,   I_{x-e_2}, \, J_{x-e_1},  \,\wc\w_{x-e_1-e_2}  :  x\in u+\Z_{>0}^2\} . 
   \ee
This larger collection  has an $\wc\w_x$ variable for each vertex in the quadrant $u+\Z_{\ge0}^2$, an $I_x$   variable for each horizontal nearest-neighbor edge in the quadrant $u+\Z_{\ge0}^2$,  a $J_x$  variable for each vertical nearest-neighbor  edge in the quadrant $u+\Z_{\ge0}^2$, and the originally given $\w_x$ variables for vertices in the bulk $u+\Z_{>0}^2$.  

As the next theorem will show, process   \eqref{IJw9} is related to another   last-passage process $G^\rho_{u,x}$ whose  origin is fixed  at $u$ and  that utilizes both edge weights on the boundary and bulk weights in the interior of the quadrant.  First put     $G^\rho_{u,u}=0$ and on the boundaries 
  \be\label{Gr1} G^\rho_{u,\,u+me_1}=\sum_{i=1}^m I_{u+ie_1} 
  \quad\text{and}\quad
G^\rho_{u,\,u+ne_2}= \sum_{j=1}^n  J_{u+je_2} .   \ee 
 Then in the bulk 
  for $x=u+(m,n)\in u+\Z_{>0}^2$, 
\be\label{Gr2}
G^\rho_{u,\,x}= \max_{1\le k\le m} \;  \Bigl\{  \;\sum_{i=1}^k I_{u+ie_1}  + G_{u+ke_1+e_2, \,x} \Bigr\}  
\bigvee
 \max_{1\le \ell\le n}\; \Bigl\{  \;\sum_{j=1}^\ell  J_{u+je_2}  + G_{u+\ell e_2+e_1, \,x} \Bigr\} .
\ee
 $G_{a,x}$ inside the braces is the last-passage value defined in  \eqref{v:G}.  
  The superscript $\rho$ in $G^\rho_{u,x}$  distinguishes this last-passage value  from the   one in  \eqref{v:G} that uses only  i.i.d.\ bulk weights.  The first subscript $u$ in $G^\rho_{u,x}$ specifies that the $I$ and $J$ edge weights are placed on the axes $u+\Z_{>0}e_k$, $k=1,2$.   
  
  An equivalent definition of $G^\rho_{u,x}$ would be to give the  boundary conditions \eqref{Gr1} and the inductive equation 
  \be\label{Gr2.2} 
  G^\rho_{u,x} = \w_x+ G^\rho_{u,\,x-e_1}\vee G^\rho_{u,\,x-e_2}\,, 
  \qquad    x \in u+\Z_{>0}^2. 
 \ee 
  

 The next theorem summarizes the properties of process  \eqref{IJw9} and the connection with the process $G^\rho_{u,x}$.    Point (i) of the theorem  uses the following definitions.     A  bi-infinite sequence $\cY=(y_k)_{k\in\Z}$  in   $u+\Z_{\ge0}^2$ is a   {\it down-right path}  if    $y_k-y_{k-1}\in\{e_1,-e_2\}$ for all $k\in\Z$.  $\cY$ decomposes  the vertices of the quadrant into a disjoint union  $u+\Z_{\ge0}^2=\cG_-\cup\cY\cup\cG_+$  where  
 \[  \cG_-=\{x\in u+\Z_{\ge0}^2:  \exists j\in\Z_{>0}  \text{ such that  } x+(j,j)\in \cY\} \]
 is the region  strictly to the south and west of $\cY$  and 
 \[  \cG_+=\{x\in u+\Z_{\ge0}^2:  \exists j\in\Z_{>0}  \text{ such that } x-(j,j)\in \cY\} \]
 is the region  strictly to the north and east of $\cY$.  Note that $\cG_+$ is necessarily unbounded but $\cG_-$ is finite iff all but  finitely many of the points $y_k$ lie on the   axes  $u+\Z_{>0}e_k$, $k=1,2$.    In the  extreme case 
$\cY=\{u+ie_1, u+je_2: 0\le i,j<\infty\}$ consists of the axes, $\cG_-=\varnothing$ and $\cG_+=u+\Z_{>0}^2$. 
 
 For an undirected nearest-neighbor edge $e$ on $u+\Z_{\ge0}^2$,   we denote the weight by 
 \be\label{v:tincr}   \tincr(e)= \begin{cases}  I_x &\text{if $e=\{x-e_1,x\}$}\\ J_x &\text{if $e=\{x-e_2,x\}$.} \end{cases}  \ee

  
 \begin{theorem} \label{v:tIJw1}   Fix $u\in\Z^2$ and $0<\rho<1$ and assume given  independent variables \eqref{IJw} with marginal distributions \eqref{w7}.   Then the variables in \eqref{IJw9} have the following properties. 
 \begin{enumerate}
 \item[{\rm(i)}]   For any down-right path $\cY$  in   $u+\Z_{\ge0}^2$, the random variables 
 \be\label{v:Y3}    \{ \, \wc\w_z,  \, \tincr(\{y_{k-1},y_k\}),  \, \w_x:  z\in\cG_-, \, k\in\Z, \, x\in\cG_+\} 
 \ee
 are mutually independent with  marginal distributions  
\be\label{w9}\begin{aligned}
\w_{x},\, \wc\w_x \sim \text{\rm Exp}(1),  \quad I_{x}\sim \text{\rm Exp}(1-\rho)  , \quad \text{ and }  \quad  J_{x}\sim \text{\rm Exp}(\rho)  .  
\end{aligned}\ee 
 \item[{\rm(ii)}]   The $I_x$ and $J_x$ variables are the increments of the $G^\rho_{u,x}$ last-passage process: 
 \be\label{IJG6}\begin{aligned} 
 I_x=G^\rho_{u,x}-G^\rho_{u,x-e_1} \qquad \text{for $x\in u+(\Z_{>0})\times(\Z_{\ge0})$} , \\
  J_x=G^\rho_{u,x}-G^\rho_{u,x-e_2} \qquad \text{for $x\in u+(\Z_{\ge0})\times(\Z_{>0})$} . \\
\end{aligned}  \ee
 \end{enumerate}
 \end{theorem}

 Theorem \ref{v:tIJw1} rests on an inductive argument based on the next lemma,  which describes the joint distribution preserved  by the mapping in Figure \ref{fig:ne}. 
 
 \begin{lemma}\label{v:lmIJw}  Let $0<\rho<1$.  
 Assume given independent variables $W\sim$ {\rm Exp$(1)$},  $I\sim$ {\rm Exp$(1-\rho)$}, and 
 $J\sim$ {\rm Exp$(\rho)$}.   Define 
 \be\label{IJw19}\begin{aligned}
W'&=I\wedge J   \\
I'&=W+(I-J)^+ \\
J'&=W+(I-J)^- .  
\end{aligned}\ee 
Then the triple  $(W', I', J')$  has the same distribution as  $(W, I,J)$. 
 \end{lemma} 
 
 This lemma is proved by calculating a joint transform such as the Laplace transform or characteristic function,  or by transforming the joint density.    The mapping $(W, I,J)\mapsto(W', I', J')$  is an involution, that is, its own inverse.  
 
 \begin{proof}[Proof of Theorem \ref{v:tIJw1}]
 Part (i).    This is proved inductively on $\cY$.  The base case is $\cY=\{u+ie_1, u+je_2: 0\le i,j<\infty\}$, in which case the claim simply amounts to the initial condition in \eqref{w7}.  
 
 Now assume given $\cY=\{y_k\}$ for which the claim in part (i) holds.  We show that this claim  continues to hold for any $\cY'$ obtained from $\cY$ by ``flipping a southwest corner into a northeast corner''.   So pick any $x\in u+\Z_{>0}^2$ and $m\in\Z$  such that $(y_{m-1}, y_m, y_{m+1})=(x-e_1, x-e_1-e_2 ,x-e_2)$ are points along $\cY$.  Define $\cY'=\{y_k'\}$ by setting 
 \begin{align*}    y_k' = y_k  \qquad \text{for $k\ne m$, \quad and }  \qquad 
 y_m'=y_m+e_1+e_2=x. 
 \end{align*}
 In other words, $\cY$ has a southwest corner at $x-e_1-e_2$,  and $\cY'$ has a northeast corner at $x$. 
 
   Transforming $\cY$ into $\cY'$ changes $\cG_-$ to $\cG_-'=\cG_-\cup\{x-e_1-e_2\}$ and $\cG_+$ to $\cG_+'=\cG_+\setminus\{x\}$.   Thus constructing the variables  \eqref{v:Y3} for $\cY'$ involves transforming the triple $(\w_{x}, I_{x-e_2}, J_{x-e_1})$  into   $(\wc\w_{x-e_1-e_2},  I_x, J_x)$ according to equations  \eqref{IJw5.1}--\eqref{IJw5.3}, and copying the remaining variables from \eqref{v:Y3} for $\cY$.  The claim now follows  for $\cY'$ by  the induction assumption and Lemma \ref{v:lmIJw}.  
 By the induction assumption,   variables $(\w_{x}, I_{x-e_2}, J_{x-e_1})$ have the independent exponential distributions required for the hypothesis 
 of Lemma \ref{v:lmIJw}, and so by the lemma the triple  $(\wc\w_{x-e_1-e_2},  I_x, J_x)$ also has the independent exponential distributions required for \eqref{w9}.  

\medskip 

 Part (ii).    The claim  is true by construction for variables $I_{u+ie_1}$ and $J_{u+je_2}$ on the axes.   Here is the inductive argument for $I_x$, assuming that the claim holds for $I_{x-e_2}$ and $J_{x-e_1}$ and utilizing \eqref{Gr2.2}: 
 \begin{align*}
 I_x&=\w_x+(I_{x-e_2}-J_{x-e_1})^+= \w_x+(G^\rho_{u,\,x-e_2}-G^\rho_{u,\,x-e_1})^+ \\
& =\w_x+ G^\rho_{u,\,x-e_1}\vee G^\rho_{u,\,x-e_2}-G^\rho_{u,\,x-e_1} \\
&= G^\rho_{u,\,x} -G^\rho_{u,\,x-e_1}.  
 \end{align*} 
A similar argument works for $J_x$ under the same inductive assumption.    
 \end{proof}

 Let us observe some immediate and valuable consequences of Theorem \ref{v:tIJw1}. 
 
 By taking   $\cY=\{y_k\}$ as the axes at a new origin $v\in u+\Z_{\ge0}^2$, given by $y_k=v+ke_1$  and $y_{-k}=v+ke_2$ for $k\ge 0$, part (i) of the theorem implies that the process 
 $\{ \w_{v+x}, I_{v+x-e_2}, J_{v+x-e_1}:  x\in \Z_{>0}^2\}$ has the same distribution for all $v\in u+\Z_{\ge0}^2$.   Thus $B(x,y)=G^\rho_{u,y}-G^\rho_{u,x}$ is a stationary cocycle, restricted to the quadrant $x,y\in u+\Z_{\ge0}^2$.   
 
 The variables $\{\wc\w_x: x\in u+\Z_{\ge0}^2\}$ are i.i.d.\ Exp$(1)$ distributed.  
 
 \subsection{Explicit shape functions}  
 
We compute the limit shape functions  for both last-passage percolation processes, the stationary one and the one with i.i.d.\  weights. 
For the stationary process define the function 
\be\label{g-rho}  g^\rho(s,t)=   \frac{s}{1-\rho}+\frac{t}{\rho}.  \ee

\begin{proposition}   Fix $0<\rho<1$.  The stationary corner growth model satisfies these properties: 
$    \E[G^\rho_{0,(m,n)}]  =  g^\rho(m,n)$  for all $m,n\in\Z_{\ge0}$ and the law of large numbers 
\be\label{g-rho-3}   \lim_{N\to\infty} N^{-1}  G^\rho_{0,(\fl{Ns}, \fl{Nt})}  = g^\rho(s,t)  
\qquad \text{ almost surely for each $(s,t)\in\R_{\ge0}^2$. } \ee
\end{proposition} 

\begin{proof}
Rewrite in terms of nearest-neighbor increments: 
\be\label{G9}\begin{aligned}
G^\rho_{0,(m,n)}=   \sum_{i=1}^{m} I_{(i,0)} + \sum_{j=1}^{n} J_{(m,j)}. 
\end{aligned}\ee
Then use the translation invariance of the distributions which says that each nearest-neighbor increment has the exponential distribution imposed on the boundary variables in \eqref{w7}:
\be\label{G12}\begin{aligned}
\E[G^\rho_{0,(m,n)}]=   \sum_{i=1}^{m} \E I_{ie_1} + \sum_{j=1}^{n} \E J_{(m,j)} = \frac{m}{1-\rho}+\frac{n}{\rho}. 
\end{aligned}\ee
The limit of the stationary  last-passage process is an application of the classical law of large numbers and some large deviation estimates, applied separately to the two sums: the limit below holds almost surely for any given $(s,t)\in\R_{\ge0}^2$.  
\be\label{G17}\begin{aligned} 
\lim_{N\to\infty} N^{-1}  G^\rho_{0,(\fl{Ns}, \fl{Nt})} &=
\lim_{N\to\infty}  \biggl\{ N^{-1} \sum_{i=1}^{\fl{Ns}} I_{(i,0)} + N^{-1}  \sum_{j=1}^{\fl{Nt}} J_{(\fl{Ns},\,j)}\biggr\} \\
&=  s \E (I_{e_1}) + t \E (J_{e_2}) = \frac{s}{1-\rho}+\frac{t}{\rho}. 
\qedhere\end{aligned}\ee
\end{proof} 

Next we take a limit in  the coupling between the last-passage processes $G_{x,y}$ and $G^\rho_{0,x}$.   Fix $s,t>0$ and use   \eqref{Gr2} for $x= (\fl{Ns}, \fl{Nt})$ to write 
\be\label{Gr7} \begin{aligned}
G^\rho_{0,(\fl{Ns}, \fl{Nt})} &=\sup_{0\le a\le s}  \Bigl\{  \;  \sum_{i=1}^{\fl{Na}} I_{(i,0)} +  G_{(\fl{Na},1),(\fl{Ns},\fl{Nt})}  \Bigr\}   \\
& \qquad \qquad \qquad  \bigvee  
\sup_{0\le b\le t}  \Bigl\{  \;  \sum_{j=1}^{\fl{Nb}} J_{(0,j)} +  G_{(1,\fl{Nb}),(\fl{Ns},\fl{Nt})}  \Bigr\}. 
\end{aligned} \ee

After letting $N\to\infty$, with some estimation on the right-hand side,  utilizing limits \eqref{lln9} and \eqref{g-rho-3},  we have 
\be\label{Gr9} \begin{aligned}
\frac{s}{1-\rho}+\frac{t}{\rho}= \sup_{0\le a\le s}  \Bigl\{  \frac{a}{1-\rho}+ \gpp(s-a,t)  \Bigr\}   
\vee  
\sup_{0\le b\le t}   \Bigl\{  \frac{b}{\rho}+ \gpp(s,t-b)  \Bigr\} . 
\end{aligned}\ee

In the  next theorem we take advantage of  the connection  above to find the   shape function $\gpp$ for the LPP process \eqref{v:G} with i.i.d.\ Exp$(1)$ weights.

 \begin{theorem}\label{thg}   Assume \eqref{ass6}.  Then we have the following law of large numbers.  For each $\xi\in\R_{\ge0}^2$ the limit below holds with probability 1, with  the  shape function $\gpp$ as given.  
 \be\label{lln}
 \lim_{N\to\infty} N^{-1} \Gpp_{0,\fl{N\xi}} =\gpp(\xi)\equiv \bigl(\sqrt \xi_1+\sqrt \xi_2\,\bigr)^2.  
 \ee
  \end{theorem}

\begin{proof} 

By the general law of large numbers Theorem \ref{v:t-lln2}   for last-passage percolation, we know that the limit in \eqref{lln} exists and that $\gpp$ is finite, concave and continuous. 
Begin with \eqref{Gr9} for $s=t$:  
\begin{align*}
\frac{t}{1-\rho}+\frac{t}{\rho}= \sup_{0\le a\le t}  \Bigl\{  \frac{a}{1-\rho}+ \gpp(t-a,t)  \Bigr\}   
\vee  
\sup_{0\le b\le t}   \Bigl\{  \frac{b}{\rho}+ \gpp(t,t-b)  \Bigr\} . 
\end{align*}
Use the symmetry of $\gpp$ and assume that $0<\rho\le 1/2$:
\begin{align*}
\frac{t}{1-\rho}+\frac{t}{\rho}&= \sup_{0\le a\le t}  \Bigl\{  \frac{a}{1-\rho}+ \gpp(t-a,t)  \Bigr\}   
\vee  
\sup_{0\le b\le t}   \Bigl\{  \frac{b}{\rho}+ \gpp(t-b,t)  \Bigr\} \\
&=\sup_{0\le b\le t}   \Bigl\{  \frac{b}{\rho}+ \gpp(t-b,t)  \Bigr\}.  
\end{align*}
Let 
\be\label{fgpp} f(b)=\begin{cases}  -\gpp(t-b,t), &0\le b\le t\\ \infty,  &b<0\text{ or } b>t. \end{cases} \ee
Then $f$ is convex and lower semicontinuous.    After a change of variable $x=1/\rho\in[2,\infty)$,  the equation above becomes 
\begin{align*}
t\Bigl( x+1+\frac1{x-1}\Bigr)=\sup_{b\in\R} \{ bx-f(b)\}, \qquad x\ge 2.  
\end{align*}
This is an instance of  convex duality,  so the convex conjugate $f^*$ of $f$ satisfies  
\[  f^*(x)= t\Bigl( x+1+\frac1{x-1}\Bigr) \qquad \text{ for $x\ge 2$. } \]   
The derivatives $(f^*)'(2+)=0$ and $(f^*)'(\infty-)=t$ tell us that we can restrict the supremum in the double convex duality as in the second equality below, for $0\le b\le t$.  Then find the supremum by calculus: 
\[  f(b)=f^{**}(b)=\sup_{x\ge 2} \{ xb- f^*(x)\}= 
-\,\frac{b\sqrt{t(t-b)}+bt-t(\sqrt t+\sqrt{t-b}\,)^2}{\sqrt{t(t-b)}}. \]
Taking $b=t-s$ for $s\in[0,t]$ in the definition \eqref{fgpp}  of $f$     gives 
\[   \gpp(s,t)=-f(t-s)=(\sqrt s+\sqrt t\,)^2 \qquad \text{for $0\le s\le t$.}  \]
Symmetry of $\gpp$ completes the proof. 
\end{proof}

Next we use \eqref{Gr9} to identify the {\it characteristic direction}  $\xi(\rho)\in\Uset$  associated with parameter value $\rho\in(0,1)$.   The edge weights $I$ and $J$ are stochastically larger than the bulk weights $\w$.  Hence both boundaries attract the path.  There is a unique direction in which the pulls of the boundaries balance out.   By definition,  $\xi(\rho)$  is the value of $\xi=(s,1-s)$    for which the right-hand side of \eqref{Gr9} with $(s,t)=(s,1-s)$  is maximized at $a=b=0$.  This direction $\xi(\rho)\in\Uset$  is uniquely given  by
\be\label{v:ch7} 
\xi(\rho)=\biggl( \frac{(1-\rho)^2} {(1-\rho)^2+\rho^2}\,, \frac{\rho^2}{(1-\rho)^2+\rho^2}\biggr). 
\ee
It then follows that in this direction  $\xi(\rho)$  the optimal path for $G^\rho_{0, \fl{N\xi(\rho)}}$ in  \eqref{Gr9} takes $o(N)$ steps on the coordinate axes,  as $N\to\infty$. 

An alternative characterization of the characteristic direction $\xi(\rho)$  is by comparison of the stationary and i.i.d.\ limit shapes.  In general   $\gpp(s,t)\le g^\rho(s,t)$ for all $s,t\ge 0$, and 
\be\label{v:ch8}  \gpp(s,t)= g^\rho(s,t) \ \text{ if and  only if } \    (s,t)=c\bigl( (1-\rho)^2, \rho^2 \bigr)\ \text{ for some $c\ge 0$. }
\ee

We can also record the limit of the point-to-line LPP process, defined  
for  $h\in\R^2$   by 
 \[ 
\Gpl_n(h)=  
 \max_{x_{0,n}: \,x_0=0} \Bigl\{ \;\sum_{k=0}^{n-1}\w_{x_k} + h\cdot x_n \Bigr\}.
 \]

 \begin{theorem}\label{thgp2l}   Assume \eqref{ass6}.  Then for every $h\in\R^2$ the limit below holds with probability 1, with  the limit function $\gpl$ as given.  
 \be\label{lln-p2l}
 \lim_{n\to\infty} n^{-1} \Gpl_n(h) =\gpl(h)\equiv  1+\frac{h_1+h_2}2 + \frac 12\sqrt{(h_1-h_2)^2 + 4\,} .  
 \ee
  \end{theorem}
  
\begin{proof}  The limit is given in Theorem \ref{v:t-lln2}.  The formula for $\gpl$ comes from  the duality \eqref{v:pl} with $\gpp$ given in \eqref{lln}. 
\end{proof} 

\subsection{Shape theorem and the geodesic}  \label{s:sh-g}

Begin with this  strengthening of the law of large numbers  \eqref{lln9}. 

\begin{theorem}\label{t:sh} 
Given $\e>0$, there exists an almost surely finite  random variable  $K$ such that 
\be\label{sh89} \abs{ \Gpp_{0,x}-\gpp(x)}  \le \e\abs{x}_1 
\quad \text{ for all $x\in\Z_{\ge 0}^2$ such that $\abs{x_1}\ge K$.}  
\ee
\end{theorem}

We postpone the proof of Theorem \ref{t:sh} towards the end of this section  but apply it presently  to prove the shape theorem.  
For real $t\ge 0$  define the randomly growing cluster in $\R_{\ge0}^2$ as 
\[  A_t=\{\zeta\in\R_{\ge0}^2: \Gpp_{0,\,\fl{\zeta}}\le t\}. \]
In other words, to get a continuum object, we attach a unit square to each integer point $x$ such that $\Gpp_{0,x}\le t$. 
The limit shape is defined by  
\[  \cA=\{\zeta\in\R_{\ge0}^2:  \gpp(\zeta)\le 1\}. \]
Since $\gpp(\zeta)\ge \abs{\zeta}_1$, we have  $\cA\subset \{ \zeta\in\R_{\ge0}^2: \abs{\zeta}_1\le 1\}$. 

\begin{theorem}[Shape theorem] \label{t:sh0} 
Given $\e>0$,  there exists an almost surely  finite random variable $T$ such that 
\be\label{sh15}  t(1-\e)\cA\, \subset\,  A_t  \, \subset \, t(1+\e)\cA \qquad\text{for all $t\ge T$.} \ee
\end{theorem}

\begin{proof} 
With the random variable $K$ given in Theorem \ref{t:sh}, set 
\[  T=\max_{x\in\Z_{\ge0}^2:  \,\abs{x}_1\le K} \Gpp_{0,x} \vee \gpp(x) . \]
Consider  $t\ge T$.  We prove the two inclusions in turn. 

 Suppose $\zeta\in\R_{\ge0}^2$ satisfies $\gpp(\zeta)\le t(1-\e)$. Let $x=\fl{\zeta}$, the closest integer point to the left and below $\zeta$.   Then also $\abs{x}_1\le\gpp(x)\le t(1-\e)$. To have $\zeta\in A_t$ we need to show that  $x\in A_t$.  If $\abs{x}_1\ge K$, Theorem \ref{t:sh} gives 
\begin{align*}
\Gpp_{0,x}\le \gpp(x)+\e\abs{x}_1 \le t(1-\e) + \e t(1-\e) \le t. 
\end{align*}
While if $\abs{x}_1\le K$, then the definition of $T$ gives 
\begin{align*}
\Gpp_{0,x}\le T  \le t. 
\end{align*}
We have proved the first inclusion in \eqref{sh15}. 

\smallskip

For the second inclusion, suppose $\Gpp_{0,x}\le t$.  Then if $\abs{x}_1\ge K$, 
\begin{align*}
\Gpp_{0,x}\ge \gpp(x)-\e\abs{x}_1  
\ge (1-\e)\abs{x}_1,
\end{align*}
and we conclude that  $\Gpp_{0,x}\le t$ implies 
\[  \abs{x}_1 \le K \vee \frac{t}{1-\e} = \frac{t}{1-\e}.  \] 
($T\ge K$ follows from the definition of $T$.) Then for $\abs{x}_1\ge K$, 
\[ \gpp(x)\le \Gpp_{0,x} + \e\abs{x}_1 \le t+  \frac{\e t}{1-\e}  \]
while for $\abs{x}_1\le K$, $\gpp(x)\le T  \le t$ by the definition of $T$.   Thus integer points $x\in A_t$ satisfy 
\[  \gpp(x) \le t\bigl( 1+  \tfrac{\e }{1-\e}\bigr). \]

Now, given $\zeta\in A_t$, let $x=\fl{\zeta}$.  Then, by the explicit formula  \eqref{lln} for $\gpp$, 
\begin{align*}
\gpp(\zeta)&\le \gpp(x+e_1+e_2)\le \gpp(x)+ 4+2\sqrt{x_1}+2\sqrt{x_2} \\
&\le  t\bigl( 1+  \tfrac{\e }{1-\e}\bigr)+ 4+  4\sqrt{\tfrac{t }{1-\e}} \\
&=t\Bigl(1+  \tfrac{\e }{1-\e} +\tfrac4t + \tfrac4{\sqrt{t(1-\e)}}\,\Bigr)
\le t(1+\e'), 
\end{align*}
when $\e'>0$ is given, and then $\e>0$ is chosen small enough and $T$ large enough. Thus  we can conclude that $A_t\subset  t(1+\e')\cA$ for $t\ge T$. 
\end{proof}

Let $\pi^{x,y}_\dbullet \in\Pi_{x,y}$ denote the (almost surely unique)  maximizing path for $\Gpp_{x,y}$ defined by \eqref{v:G}. That is, the last-passage value satisfies 
\be\label{v:Gpi}  
	  \Gpp_{x,y}= \sum_{i=0}^{\abs{y-x}_1}\w_{\pi^{x,y}_i}.
	\ee
Even though last-passage percolation seeks to maximize rather than minimize path weight, $\pi^{x,y}_\bbullet$ is often called the {\it geodesic} from $x$ to $y$.   We measure the distance 
   of the geodesic from the straight line segment between  $x$ and  $y$ with the quantity 
\be\label{D6}  D_{x,y}=\max_{0\le i\le \abs{y-x}_1} \Bigl\lvert \, \pi^{x,y}_i-  \frac{\abs{y-x}_1-i}{\abs{y-x}_1} \, x \,-\, \frac{i}{\abs{y-x}_1} \,y\,\Bigr\rvert_1 . \ee
 $D_{x,y}$ is larger than $\max_{0\le i\le \abs{y-x}_1} \dist(\pi^{x,y}_i, [x,y])$, the maximal Euclidean distance of a point $\pi^{x,y}_i$ of the geodesic from the line segment $[x,y]$. The reason is that the $\ell^1$ norm  dominates the Euclidean norm and the point $\frac{\abs{y-x}_1-i}{\abs{y-x}_1} \, x + \frac{i}{\abs{y-x}_1} \,y$ is not necessarily the orthogonal  projection of $\pi^{x,y}_i$ onto  the line segment $[x,y]$. 

One can check that the shape function $\gpp(\xi)=\xi_1+\xi_2+2(\xi_1\xi_2)^{1/2}$ has the following property:  if $\xi, \zeta\ne 0$, then 
\[ \gpp(\xi+\zeta)> \gpp(\xi)+\gpp(\zeta)  \]
unless $\xi$ and $\zeta$ lie on the same ray from the origin, that is, $\zeta=c\xi$ for some $c>0$.  This tells us that, on the law of large numbers scale, the optimal path is always the straight line segment.  
In the next theorem we make this rigorous, by showing  that  the  random  geodesic to a point at distance $N$ is within distance $o(N)$ of the straight line segment,  with high probability. 

 \begin{theorem}\label{t:sh2}  We have the limit 
 \be\label{sh88}  \lim_{M\to\infty}\; \max_{x\in\Z_{\ge 0}^2:\, \abs{x}_1=M} \,\frac{D_{0,x}}{\abs{x}_1} =0 \qquad\text{in probability.}\ee
\end{theorem}

\bigskip

We turn to the proofs of Theorems \ref{t:sh} and \ref{t:sh2}.  Both proofs use an approximation in terms of a grid of points on the line segment  $\Uset=[e_1,e_2]$, given in the next lemma. 

\begin{lemma} \label{sh-lm44} 
Fix  $\delta>0$ and $k\in\Z_{>0}$ such that $\tfrac1k<\delta$. 
Set $\xi_i=(\frac{i}k, \frac{k-i}k)\in\Uset$ for  $0\le i\le k$.   Then for each  $z=(z_1,z_2)\in\R_{\ge0}^2$ such that  \[ 1+\delta\le \abs{z}_1\le 1+2\delta \] there exists $i\in\{0,\dotsc,k\}$ such that  $\xi_i\le z\le \xi_i+2\delta(e_1+e_2)$. 
\end{lemma} 

\begin{proof}
{\it Case 1.}  Suppose $1\le z_2\le 1+2\delta$. Then for $i=0$ we have 
\[   \xi_{0,2}=1\le z_2\le 1+2\delta = \xi_{0,2}+2\delta  \]
 and 
\[   \xi_{0,1}=0\le z_1\le 1+2\delta -z_2 \le 2\delta = \xi_{0,1}+2\delta. \]
 
 {\it Case 2.} 
For  $0\le z_2 <1 $    choose $i\in\{1,\dotsc,k\}$ such that 
\[  \xi_{i,2}= \tfrac{k-i}k\le z_2< \tfrac{k-i+1}k. \]  Then 
\begin{align*}
z_1\ge 1+\delta-z_2 > 1+\delta -1+\tfrac{i}k -\tfrac1k > \tfrac{i}k = \xi_{i,1}.  
\end{align*}
From the opposite side 
\begin{align*}
z_1\le 1+2\delta-z_2 \le  2\delta  +\tfrac{i}k =   \xi_{i,1} + 2\delta 
\end{align*}
and $z_2< \tfrac{k-i}k +\tfrac1k \le  \tfrac{k-i}k+2\delta=\xi_{i,2} +2\delta$.
\end{proof} 

\begin{proof}[Proof of Theorem \ref{t:sh}]   Given $\e\in(0,1)$, fix   $\e_1>0$ small enough so that $  17\sqrt{\e_1} <\e$.    Fix an integer  $k>\e_1^{-1}$.  Let $\xi_i=(\frac{i}k, \frac{k-i}k)\in\Uset$  as in Lemma \ref{sh-lm44} and set $\xi_i'=\xi_i+2\e_1(e_1+e_2)$. 
An explicit calculation shows that 
\[ 0<  \gpp(\xi_i')-\gpp(\xi_i)\le 16\sqrt{\e_1} . \]

  Utilizing limit \eqref{lln}, let $N_0$ be an almost surely finite random variable such that, for all  $\zeta\in\{\xi_1,\dotsc,\xi_k, \xi_1',\dotsc,\xi_k'\}$, 
\[  \abs{\,\Gpp_{0, \fl{N\zeta}}-  N\gpp(\zeta)\, } \le N\e_1  \qquad \text{for $N\ge N_0$.} 
 \]  
  Increase $N_0$ if necessary to guarantee that  
  \[  (N+1)(1+\e_1)<N(1+2\e_1) \quad  \text{ for all $N\ge N_0$. } \] 
  
Let $K=(1+\e_1)N_0$.  Given $x\in\Z_{\ge0}^2$ such that $\abs{x}_1\ge K$,  let $N \ge N_0$  be the integer such that \[ N(1+\e_1)\le \abs{x}_1< (N+1)(1+\e_1) <N(1+2\e_1). \]   Then by Lemma \ref{sh-lm44} we can fix $j$ such that 
\[ \fl{N\xi_j} \le N\xi_j \le x\le \fl{N\xi_j'} \le N\xi_j '. 
  \]

By monotonicity of $\Gpp$  and monotonicity and  homogeneity   of $\gpp$, 
\begin{align*}
\Gpp_{0,x}-\gpp(x) &\ge \Gpp_{0, \fl{N\xi_j}}- N\gpp(\xi_j') \\
&\ge \Gpp_{0, \fl{N\xi_j}}- N\gpp(\xi_j)- 16N\sqrt{\e_1} \\ 
&\ge  -N\e_1- 16N\sqrt{\e_1}  \; \ge \; -N\e  \; \ge \; -\e\abs{x}_1. 
\end{align*}
From the other side, 
\begin{align*}
\Gpp_{0,x}-\gpp(x) &\le \Gpp_{0, \fl{N\xi_j' }}- N\gpp(\xi_j) \\
&\le \Gpp_{0, \fl{N\xi_j'}}- N\gpp(\xi_j') + 16N\sqrt{\e_1} \\ 
&\le  N\e_1+ 16N\sqrt{\e_1}  \; \le \; N\e  \; \le \; \e\abs{x}_1. 
\end{align*}
This completes the proof of Theorem \ref{t:sh}.
%
\end{proof}

\begin{proof}[Proof of Theorem \ref{t:sh2}]    Fix $0<\e<1$.  We show that 
\[   \lim_{M\to\infty} \P\biggl\{ \; \max_{x\in\Z_{\ge 0}^2:\, \abs{x}_1=M} \,\frac{D_{0,x}}{\abs{x}_1} \ge \e\biggr\}  =0. \]  
Consider $M$ fixed but large until the end of the argument where we take it to $\infty$.  Throughout the proof, $x$ denotes a point  in $\Z_{\ge 0}^2$  such that $\abs{x}_1=M$ and $z$ denotes a point on  the geodesic  $\pi^{0,x}_\dbullet$ that violates the condition 
$\abs{z-(\abs{z}_1/\abs{x}_1)x}_1< \e$. 

With $0<\e<1$ given, we pick three small positive quantities $\delta_1, \e_1$, and  $\e_2$.  Let $\e_1>0$ satisfy 
\be\label{shpi-e1}  0< \e_1< \e/8. \ee 
Then fix $\delta_1>0$ such that 
\be\label{shpi-d} 0<\delta_1 <  \e-6\e_1.   \ee 
Given this $\delta_1>0$,  let  $\e_2>0$ be such that  
\be\label{shpi98}  \begin{aligned}  
 \gpp(\eta)\ge \gpp(\zeta)+\gpp(\eta-\zeta)+\e_2 \quad &\text{for all $\,0\le\zeta\le \eta$ such that } \\
 &1\le\abs{\eta}_1\le 2 \ \text{ and } \ 
   \Bigl\lvert \, \zeta-\frac{\abs{\zeta}_1}{\abs{\eta}_1}\eta\, \Bigr\rvert_1   \ge \delta_1.  
 \end{aligned}  \ee
Now shrink $\e_1>0$ further to ensure that 
\be\label{shpi-e3}  \e_2 > 3\e_1+16\sqrt{\e_1}.  \ee 
Note that  we can shrink $\e_1$ without violating \eqref{shpi-e1} or \eqref{shpi-d}, while keeping $\e, \delta_1$ and $\e_2$ fixed.

  As in Lemma \ref{sh-lm44} with $\delta=\e_1$,  define $\xi_i=(\frac{i}k, \frac{k-i}k)$   and   $\xi_i'=\xi_i+2\e_1(e_1+e_2)$ for an integer  $k>\e_1^{-1}$.  Recall also 
\be\label{gpp-xi7}   \gpp(\xi_i')-\gpp(\xi_i)\le 16\sqrt{\e_1} . \ee 
 
Let $N$  be the integer such that 
\be\label{shpi111} N(1+\e_1)\le  M=\abs{x}_1< (N+1)(1+\e_1). \ee
  Increase $M$ if necessary to guarantee that  $(N+1)(1+\e_1)<N(1+2\e_1)$.  Then by Lemma \ref{sh-lm44} for each $x$ with $\abs{x}_1=M$ we can fix $j$ such that 
\be\label{shpi113} \fl{N\xi_j} \le N\xi_j \le x\le \fl{N\xi_j'} \le N\xi_j '. 
  \ee

We isolate two lemmas from the main body of the proof.  

\smallskip 

\begin{lemma}\label{shpi-lm3} 
Let $z$ be  a point that satisfies $0\le z\le x$ and 
\be\label{shpi115}    \Bigl\lvert \,z-\frac{\abs{z}_1}{\abs{x}_1}x\, \Bigr\rvert_1  \ge\e\abs{x}_1 . 
\ee
Then 
\be\label{shpi118}
\abs{z}_1\ge \frac\e2\,\abs{x}_1
\quad\text{and}\quad 
\abs{z-x}_1\ge \frac\e2\,\abs{x}_1. 
\ee
For $j$ as in \eqref{shpi113}, 
\be\label{shpi120}
 \Bigl\lvert \,\frac{z}N \, - \,\frac{\abs{z}_1}{N\abs{\xi_j'}_1}  {\xi_j'}\, \Bigr\rvert_1  \ge\delta_1 . 
\ee
\end{lemma}

\begin{proof} [Proof of Lemma \ref{shpi-lm3}]  
By the triangle inequality and $\abs{z}_1\le  \abs{x}_1$, 
\begin{align*}
 \Bigl\lvert \,z-\frac{\abs{z}_1}{\abs{x}_1}x\, \Bigr\rvert_1 \; \le \;  
 2\abs{z}_1 
\end{align*}
and 
\begin{align*}
 \Bigl\lvert \,z-\frac{\abs{z}_1}{\abs{x}_1}x\, \Bigr\rvert_1 \; &\le \;  
\Bigl\lvert \,\frac{\abs{z}_1}{\abs{x}_1}(z-x) +\frac{\abs{x}_1-\abs{z}_1}{\abs{x}_1}z\, \Bigr\rvert_1
\;\le\;  \abs{z-x}_1 + \abs{x}_1-\abs{z}_1 \\
&\le\;  2\abs{z-x}_1. 
\end{align*}
Inequalities   \eqref{shpi118} have been verified.

For   \eqref{shpi120}  we   bound  the distance $\abs{{x}/{\abs{x}_1}- {\xi_{j}'}/{\abs{\xi_{j}'}}}$.   Let $m\in\{1,2\}$ be the  coordinate  index.   First an upper bound on $\frac{x_m}{\abs{x}_1}$.  
\begin{align*}
\frac{x_m}{\abs{x}_1} \le \frac{N\xi_{j,m}'}{N(1+\e_1)} = \frac{\xi_{j,m}'}{\abs{\xi_{j}'}_1}\cdot\frac{1+4\e_1}{1+\e_1}  \le   \frac{\xi_{j,m}'}{\abs{\xi_{j}'}} +3\e_1. 
\end{align*}
Then a lower bound. 
\begin{align*}
\frac{x_m}{\abs{x}_1} \ge \frac{N(\xi_{j,m}'-2\e_1)}{N(1+2\e_1)} \ge  \frac{\xi_{j,m}'}{\abs{\xi_{j}'}_1}- 2\e_1. 
\end{align*}
From the last two displays
\be\label{shpi130}    \Bigl\lvert \,\frac{x}{\abs{x}_1} \, - \,\frac{\xi_{j}'}{\abs{\xi_{j}'}_1}\, \Bigr\rvert_1 
\le 6\e_1. \ee
On the other hand, from \eqref{shpi115} and  $\abs{x}_1\ge  N(1+\e_1)$ we obtain 
\[    \Bigl\lvert \,\frac{z}N -\frac{\abs{z}_1}N \cdot \frac{x}{\abs{x}_1}\, \Bigr\rvert_1  \ge\e(1+\e_1). \]
To bound $\abs{z}_1$ from above,  by  $0\le z\le x$ and \eqref{shpi118} ,
\begin{align*}
\abs{z}_1=\abs{x}_1 - \abs{x-z}_1 \le \Bigl(1-\frac\e2\Bigr) \abs{x}_1 \le N. 
\end{align*}

Finally we combine the bounds above with the triangle inequality: 
\begin{align*}
 \Bigl\lvert \,\frac{z}N \, - \,\frac{\abs{z}_1}{N\abs{\xi_j'}_1}  {\xi_j'}\, \Bigr\rvert_1  
&\ge
  \Bigl\lvert \,\frac{z}N -\frac{\abs{z}_1}N \cdot \frac{x}{\abs{x}_1}\, \Bigr\rvert_1 
\;  - \;  \frac{\abs{z}_1}N \cdot   \Bigl\lvert \,\frac{x}{\abs{x}_1} \, - \,\frac{\xi_{j}'}{\abs{\xi_{j}'}_1}\, \Bigr\rvert_1  \\
  &\ge  \e(1+\e_1) - 6\e_1 \ge \delta_1.  
\end{align*} 
Claim \eqref{shpi120} has been verified and Lemma \ref{shpi-lm3} proved. 
\end{proof} 

For $N\in\Z_{>0}$ let $B_{N} $ be the event on which the following statements hold:
\begin{align} 
\label{BM1}  \Gpp_{0, \fl{N\xi_i}} \ge N\gpp(\xi_i) -N\e_1
 \qquad\text{for $i\in[k]$, } 
\end{align}  
and whenever lattice points $0\le z\le x$ and index  $j$  satisfy 
\eqref{shpi111},  \eqref{shpi113},  and \eqref{shpi115},   
  we have 
\be \label{BM2}   \Gpp_{0,z}\le \gpp(z)+N\e_1  \ee
and 
\be \label{BM3}  \Gpp_{z, \,\fl{N\xi_j'}} \le \gpp(N\xi_j'-z) +N\e_1. \ee

\begin{lemma}\label{shpi-lm5} 
$\P(B_N)\to 1$ as $N\to\infty$. 
\end{lemma}

\begin{proof} [Proof of Lemma \ref{shpi-lm5}]  
Both \eqref{BM1} and \eqref{BM2} are  satisfed with high probability by Theorem \ref{t:sh}.     For \eqref{BM1}  this is immediate.   For \eqref{BM2} we need a few  additional observations.  
By the assumptions on $z$ and $x$ and by  \eqref{shpi118},  
\[   \abs{z}_1\ge \frac\e2\,\abs{x}_1 \ge  \frac{N}2 \e(1+\e_1),  \]
and so $\abs{z}_1$ grows large as $N$ grows. Consequently   \eqref{sh89} applies and  gives 
\[   \Gpp_{0,z}\le \gpp(z)+\e'\abs{z}_1 \] 
for large enough $N$,  for all $z$ that satisfy \eqref{shpi115}, for any fixed $\e'>0$. 
Then take $\e'<\e_1/(1+2\e_1)$ so that 
$\e'\abs{z}_1\le  \e'\abs{x}_1\le N\e'(1+2\e_1)\le N\e_1$. 

To satisfy \eqref{BM3}, consider   $N$ and $j$ fixed, and let $x$ and  $z$ vary.  Translate and reflect the lattice to transform the process $\{\Gpp_{z, \,\fl{N\xi_j'}}\}_{z\le \fl{N\xi_j'}}$ into the distributionally identical process $\{\Gpp_{0, \,\fl{N\xi_j'}-z}\}_{z\le \fl{N\xi_j'}}$.  We bound the norm of the endpoint above and below. 
\begin{align*}
\abs{\,\fl{N\xi_j'}-z\,}_1 \le \abs{x-z}_1\,+\, \abs{\,\fl{N\xi_j'}-x\,}_1
\le  \abs{x}_1 + 4N\e_1 \le N (1+6\e_1) .
\end{align*} 
 By the definition of $B_N$,  the relevant points $z$  satisfy condition \eqref{shpi115}, and hence    by \eqref{shpi118}, 
\begin{align*}
\abs{\,\fl{N\xi_j'}-z\,}_1 \ge \abs{x-z}_1\,-\, \abs{\,\fl{N\xi_j'}-x\,}_1
\ge \frac\e2\,\abs{x}_1 - 4N\e_1 \ge N (\tfrac12\e(1+\e_1)-4\e_1) .
\end{align*} 
As $N$ grows this last lower bound grows large, uniformly for the points $z$ that satisfy   \eqref{shpi115}.  Hence for large enough $N$  \eqref{sh89} applies and for these points $z$ we have 
\begin{align*}
\Gpp_{0, \,\fl{N\xi_j'}-z} \le \gpp(N\xi_j'-z)  +\e' \abs{\,\fl{N\xi_j'}-z\,}_1 \; \le \; \gpp(N\xi_j'-z)   +N\e_1
\end{align*}
once $\e'$ is chosen small enough.   Since we transformed the original process, we do not get \eqref{BM3} almost surely, but only with high probability. 
\end{proof} 

\smallskip

We conclude the proof of   Theorem \ref{t:sh2}.
Suppose that event $B_N$ of \eqref{BM1}--\eqref{BM3} occurs and  the geodesic $\pi^{0,x}_\bbullet$ goes through  a point  $z$   that satisfies \eqref{shpi115}.   This gives  $\Gpp_{0,x} \le  \Gpp_{0,z}+ \Gpp_{z,x}$ (an inequality because the weight at $z$ is counted twice).  
The following string of inequalities  uses in turn \eqref{gpp-xi7}, \eqref{BM1}, monotonicity, \eqref{BM2}--\eqref{BM3}, and finally   \eqref{shpi98} applied to $\eta=\xi_j'$ and $\zeta=z/N$:  
\begin{align*}
N\gpp(\xi_j')-N(\e_1+16\sqrt{\e_1}\,) &\le  N\gpp(\xi_j)-N\e_1 \; \le \; \Gpp_{0,\fl{N\xi_j}} \le \Gpp_{0,x}  \\[4pt]
&\le  \Gpp_{0,z}+ \Gpp_{z,x} \; \le \; \Gpp_{0,z}+ \Gpp_{z,\fl{N\xi_j'}}  \\[3pt] 
&\le N\Bigl[ \,\gpp\Bigl(\frac{z}N\Bigr) + \gpp\Bigl(\xi_j'-\frac{z}N\Bigr)\Bigr] +2N\e_1\\[3pt] 
&\le  N\gpp(\xi_j')-N\e_2 +2N\e_1. 
\end{align*}
By \eqref{shpi-e3} we have a contradiction.  Hence on the event $B_N$ no geodesic to a point $x$ that satisfies \eqref{shpi111} can go through a point  $z$   that satisfies \eqref{shpi115}.   We  conclude that 
\[  \P\biggl\{   \frac{D_{0,x}}{\abs{x}_1} \ge \e \  \text{ for some $x$ that satisfies \eqref{shpi111}}\biggr\}   \le \P(B_N^c)\to 0.  \]
 Theorem \ref{t:sh2} is proved.
\end{proof}

\medskip 
 
 \section{Busemann functions for the exponential corner growth model} \label{sec:bus} 
 
  In this section we prove the existence of the Busemann functions and show that they provide minimizers for the variational formulas.   Section \ref{sec:bus1}  states the results and  Section \ref{sec:bus2}  contains the proof of the main existence theorem.   

 \subsection{Existence and properties of Busemann functions}
 \label{sec:bus1}

We extend  the constructions discussed in Section \ref{v:s-stat-cgm}  to the full lattice $\Z^2$.   As before a down-right path is a bi-infinite sequence $\cY=(y_k)_{k\in\Z}$  in   $\Z^2$ such that    $y_k-y_{k-1}\in\{e_1,-e_2\}$ for all $k\in\Z$.  The  lattice decomposes  into a disjoint union  $\Z^2=\cG_-\cup\cY\cup\cG_+$  where   the two regions are 
 \[  \cG_-=\{x\in \Z^2:  \exists j\in\Z_{>0}  \text{ such that  } x+j(e_1+e_2)\in \cY\} \]
   and 
 \[  \cG_+=\{x\in \Z^2:  \exists j\in\Z_{>0}  \text{ such that } x-j(e_1+e_2)\in \cY\} . \]

It will be convenient to formalize the properties  identified earlier in Theorem \ref{v:tIJw1} in the following definition.
 
 \begin{definition} \label{v:d-exp-a} 
Let $0<\alpha<1$.   Let us say that a process 
  \be\label{IJw800} 
  \{ \eta_{x} ,\,   I_{x}, \, J_{x},  \,\wc\eta_{x}  :  x\in \Z^2\} 
   \ee
 is an {\it exponential-$\alpha$  last-passage percolation system}  if the following properties  {\rm (a)--(b)} hold. \\[-7pt] 
 \begin{enumerate}[{\rm(a)}]\itemsep=7pt 
 \item  The process is stationary with marginal distributions  
 \be\label{w19}\begin{aligned}
\eta_{x},\, \wc\eta_x \sim \text{\rm Exp}(1),  \quad I_{x}\sim \text{\rm Exp}(1-\alpha)  , \quad \text{ and }  \quad  J_{x}\sim \text{\rm Exp}(\alpha)  .  
\end{aligned}\ee 
  For any down-right path $\cY=(y_k)_{k\in\Z}$  in   $\Z^2$, the random variables 
 \be\label{v:Y4}   
 \{\wc\eta_z:  z\in\cG_-\}, \quad  \{  \tincr(\{y_{k-1},y_k\}): k\in\Z\}, \quad\text{and}\quad 
 \{  \eta_x: x\in\cG_+\} 
 \ee
 are all mutually independent, where the undirected edge variables $\tincr(e)$ are defined as before in \eqref{v:tincr} as 
  \be\label{v:tincr6}   \tincr(e)= \begin{cases}  I_x &\text{if $e=\{x-e_1,x\}$}\\ J_x &\text{if $e=\{x-e_2,x\}$.} \end{cases}  \ee

 \item  The following equations  are in force at all $x\in\Z^2$: 
 \begin{align}
\label{IJw5.1.7}  \wc\eta_{x-e_1-e_2}&=I_{x-e_2}\wedge J_{x-e_1} 
\\
\label{IJw5.2.7}  I_x&=\eta_x+(I_{x-e_2}-J_{x-e_1})^+ 
\\
\label{IJw5.3.7} J_x&=\eta_x+(I_{x-e_2}-J_{x-e_1})^-. 
\end{align}

 \end{enumerate}
\end{definition}

Equations  \eqref{IJw5.1.7}--\eqref{IJw5.3.7} above  are the same as those in  \eqref{IJw5.1}--\eqref{IJw5.3}.  The indexing in \eqref{IJw5.1.7}--\eqref{IJw5.3.7}  is chosen to highlight the relationship of the variables on the edges and diagonal corners of a unit square with vertices  $\{ x-e_1-e_2, x-e_1, x, x-e_2\}$.  

\medskip

With the definition of the exponential-$\alpha$  LPP system we can state fairly succinctly the existence of Busemann functions in the next theorem.  The setting is  the following.  $\OSP$ is a probability space whose generic sample point is denoted by $\w$.  There is a group of measurable bijections $(\theta_x)_{x\in\Z^2}$ on $\Omega$ that preserve $\P$.   $\Yw=(\Yw_x)_{x\in\Z^2}$  are  i.i.d.\ Exp(1) weights  defined on $\Omega$ that  satisfy $Y_x(\w)=Y_0(\theta_x\w)$.    The LPP process is defined in terms of the  weights $\Yw$:
 	\be\label{v:GY}  
	  \Gpp_{x,y}=\max_{x_{\brbullet}\,\in\,\Pi_{x,y}}\sum_{k=0}^{\abs{y-x}_1}\Yw_{x_k}.
	\ee
The canonical choice $\Omega=\R_+^{\Z^2}$ with translations $(\theta_x\w)_y=\w_{x+y}$, an i.i.d.\ product measure $\P$ and the coordinate process $\Yw_x(\w)=\w_x$ is one possibility.  The additional generality is meaningful for the uniqueness   in Theorem \ref{t:unique}. 

 \begin{theorem}\label{t:buse}  
 For each  $0<\alpha<1$  there exist a stationary   cocycle $B^\alpha$ and a family of random weights  $\{\Xw^\alpha_x\}_{x\in\Z^2}$  on $\OSP$   with the following properties.  
 \begin{enumerate}\itemsep=3pt
 \item[{\rm(i)}]   For each $0<\alpha<1$,  process 
  \[ \{  \Xw^\alpha_x, \, B^\alpha_{x-e_1,x}, \, B^\alpha_{x-e_2,x}, \,\Yw_x:  x\in\Z^2\}  \] 
 is an exponential-$\alpha$ last-passage system as described in Definition \ref{v:d-exp-a}. 

 \item[{\rm(ii)}]     There exists a single event $\Omega_2$ of full probability such that for all $\w\in\Omega_2$,   all $x\in\Z^2$ and all $\lambda<\rho$ in $(0,1)$ we have the inequalities 
 \be\label{v:853.1} 
B^\lambda_{x,x+e_1}(\w) \le  B^\rho_{x,x+e_1}(\w) \quad\text{and}\quad 
B^\lambda_{x,x+e_2}(\w) \ge  B^\rho_{x,x+e_2}(\w).   
\ee
Furthermore,  for all $\w\in\Omega_2$ and $x,y\in\Z^2$,  the function $\lambda\mapsto B^\lambda_{x,y}(\w)$ is right-continuous with left limits. 

 \item[{\rm(iii)}]     For each fixed $0<\alpha<1$ there exists an event $\Omega^{(\alpha)}_2$ of full probability such that the following holds: for each $\w\in\Omega^{(\alpha)}_2$  and any  sequence $v_n\in\Z^2$ such that $\abs{v_n}_1\to\infty$ and 
 \be\label{v:853.3}     \lim_{n\to\infty}  \frac{v_n}{\abs{v_n}_1} \, = \,\xi(\alpha)\,=\, 
\biggl(  \frac{(1-\alpha)^2}{(1-\alpha)^2+\alpha^2} \,,   \frac{\alpha^2}{(1-\alpha)^2+\alpha^2}\biggr) ,  
 \ee
we have the limits  
 \be\label{v:855}   
 B^\alpha_{x,y}(\w)  =\lim_{n\to\infty} [ G_{x, v_n}(\w)-G_{y,v_n}(\w)]  \qquad\forall x,y\in\Z^2. \ee
 The LPP process $G_{x,y}$ is now defined by \eqref{v:GY}.  
 Furthermore,  for all $\w\in\Omega^{(\alpha)}_2$ and $x,y\in\Z^2$,  
 \be\label{v:855.7}   \lim_{\lambda\to\alpha}B^\lambda_{x,y}(\w)=B^\alpha_{x,y}(\w). \ee 
 
  \end{enumerate} 
 
 \end{theorem}

 \medskip 


Note that the process $\lambda\mapsto B^\lambda_{x,y}(\w)$ is globally cadlag (part (ii)) and at each fixed $\lambda$  almost surely continuous \eqref{v:855.7}. 

Part (i) of  Theorem \ref{t:buse}   together with \eqref{IJw5.1.7} implies 
\be\label{914}     \Yw_x=B^\alpha_{x,\,x+e_1} \wedge B^\alpha_{x,\,x+e_2} . 
\ee
In other words, cocycle $B^\alpha$ is adapted to the potential of the corner growth model in the sense of \eqref{v:VB}.   In the opposite lattice direction we have the identity
\be\label{915}     \Xw^\alpha_x=B^\alpha_{x-e_1,\,x} \wedge B^\alpha_{x-e_2,\,x} , 
\ee
which follows from \eqref{IJw5.2.7}--\eqref{IJw5.3.7}.  
    ({\it Note that in this construction the given $\Yw$-weights are now playing the role of the $\wc\w$-weights of Theorem \ref{v:tIJw1}, and the constructed weights $\Xw^\alpha$ play the role of the $\w$-weights of Theorem \ref{v:tIJw1}.}) 

 From the explicit formula for the shape function (Theorem \ref{thg}),
\[  \gpp(s,t)=s+t+2\sqrt{st}\]  and from this we derive 
\[  \nabla\gpp(s,t)=\bigl( \,1+\sqrt{t/s}\,,  1+\sqrt{s/t}\,\bigr) . \]
By the   exponential distributions of $B^\alpha$-increments given in part (i) of Theorem \ref{t:buse}  we obtain   the relation 
\be\label{EB8}     \bigl( \, \E[B^\alpha_{0,e_1}]\,, \E[B^\alpha_{0,e_2}]\,\bigr)
=\Bigl(\, \frac1{1-\alpha}\,, \frac1\alpha\,\Bigr) = \nabla\gpp(\xi(\alpha)). \ee  
This is natural since the limit \eqref{v:855} represents $B^\alpha$ as a microscopic gradient of passage time. 

The next theorem gives a uniqueness statement. 

  \begin{theorem}\label{t:unique}  Assume the setting of Theorem \ref{t:buse} with given i.i.d.\ {\rm Exp$(1)$} weights $\{Y_x\}$ on a probability space $\OSP$.  Fix $0<\rho<1$. 
  
  {\rm (i)}      Suppose that, indexed by the nearest-neighbor edges  in the  quadrant with lower left corner at  $u=(u_1,u_2)\in\Z^2$,   there are random variables $(A_{x,\, x+e_1}, A_{x,\,x+e_2})_{x\,\in\, u+\Z_{\ge0}^2}$ on $\OSP$ with  properties {\rm(a)} and {\rm(b)}:  
   \begin{enumerate}\itemsep=3pt
 \item[{\rm(a)}]   $A_{x,\, x+e_1}\sim$ {\rm Exp}$(1-\rho)$ and  $A_{x,\,x+e_2}\sim$ {\rm Exp}$(\rho)$ for all $x\in u+\Z_{\ge0}^2$.   For any $v=(v_1,v_2)\in  u+\Z_{>0}^2$,  the random variables $(A_{v-ie_1,\,v-(i-1)e_1}, A_{v-je_2,\,v-(j-1)e_2})_{1\le i\le v_1-u_1,\, 1\le j\le v_2-u_2 }$ are independent.  
 
 \item[{\rm(b)}]    For each  $x\in u+\Z_{\ge0}^2$ we have $\P$-almost surely 
additivity around the  unit square  with lower left corner at $x$: 
 \[    A_{x,\, x+e_1}+  A_{x+e_1,\,x+e_1+e_2}= A_{x,\,x+e_2}+A_{x+e_2,\, x+e_1+e_2}, 
 \]
 and recovery: 
 \[  \Yw_x=A_{x,\, x+e_1}\wedge  A_{x,\,x+e_2}.\\[4pt]  \] 
 
   \end{enumerate} 
 
 \noindent   
 Then   $A_{x,\, x+e_r}=B^\rho_{x,\, x+e_r}$  for all $x\in u+\Z_{\ge0}^2$ and $r\in\{1,2\}$, $\P$-almost surely. 
   
   \smallskip
   
{\rm (ii)}  In particular, suppose that on $\OSP$ there are random variables $(U_x, A_{x-e_1,x}, A_{x-e_2,x})_{x\in\Z^2}$ such that $\{U_x,  A_{x-e_1,x},  A_{x-e_2,x}, \Yw_x:  x\in\Z^2\} $ 
 is an exponential-$\rho$ last-passage system as described in Definition \ref{v:d-exp-a}. 
 Then  $U_x= \Xw^\rho_x$, $A_{x-e_1,x}=B^\rho_{x-e_1,x}$ and $A_{x-e_2,x}=B^\rho_{x-e_2,x}$ for all $x$, $\P$-almost surely. 
 
  \end{theorem} 
  
The second part of assumption (a) above says that the edge variables $A$ are independent on the north and east boundaries of any rectangle $[u,v]$. Part (ii)  about the uniqueness of  an exponential-$\rho$ last-passage system  that includes the given weights $\Yw$ is a consequence of   part (i)  of the theorem because assumptions  (a) and (b) are true for   an exponential-$\rho$ last-passage system.  
 
Before turning to the   proofs of   Theorems \ref{t:buse} and \ref{t:unique} in Section \ref{sec:bus2}, we revisit  previous themes.

\medskip

{\bf Busemann function as a stationary last-passage percolation process.}   
Fix $0<\rho<1$.   The Busemann function  $B^\rho$ can be viewed as an increment-stationary LPP process in two ways, depending on whether we use the weights $\Yw$ or $\Xw^\rho$.  Fix a reference point   $v\in\Z^2$.    We    construct two LPP processes with boundary  given by $B^\rho$.   $\{G^{NE}_{u,v}\}_{u\le v}$ has  boundary conditions on the north and east and admissible paths from   $v$  proceed down-left using  steps $\{-e_1, -e_2\}$,   while  $\{G^{SW}_{v,w}\}_{w\ge v}$ has  boundary conditions on the south and west and admissible paths are of the same up-right kind as before.  The definitions are 
 \be\label{NE:101}     G^{NE}_{v-ke_r,v}=B^\rho_{v-ke_r,v}  
 \qquad\text{for  $r\in\{1,2\}$ and $k\in\Z_{\ge0}$ }  \ee
 and in the bulk 
 \be\label{NE:102}   G^{NE}_{u,v}=\Yw_u+G^{NE}_{u+e_1,v}\vee G^{NE}_{u+e_2,v}
  \qquad\text{for $u\le v-e_1-e_2$}
  \ee
  and then 
   \[     G^{SW}_{v,\,v+ke_r}=B^\rho_{v,\,v+ke_r}  
 \qquad\text{for  $r\in\{1,2\}$ and $k\in\Z_{\ge0}$ }  \]
 and in the bulk 
  \[  G^{SW}_{v, w}=\Xw^\rho_w+G^{SW}_{v,\,w-e_1}\vee G^{SW}_{v,\,w-e_2}
  \qquad\text{for $w\ge v+e_1+e_2$.}
  \]
Inductive arguments using \eqref{914} and \eqref{915}  and the additivity of $B^\rho$  prove the next  theorem.

\begin{theorem}\label{t:GB} 
For all $u\le v$ and all $w\ge v$,   $G^{NE}_{u,v}=B^\rho_{u,v}$ and $G^{SW}_{v,w}=B^\rho_{v,w}$. 
\end{theorem}

The LPP processes constructed here have the same distribution as $\Gpp^\rho$ constructed in \eqref{Gr1}--\eqref{Gr2} (modulo a lattice reversal for $G^{NE}$):   with $v$ fixed,  
\be\label{GB45} 
 \{B^\rho_{v-x,\,v}\}_{x\in\Z_{\ge0}^2}   \deq \{\Gpp^\rho_{v,\, v+x}\}_{x\in\Z_{\ge0}^2} 
\quad\text{and}\quad 
\{B^\rho_{v,\, v+x}\}_{x\in\Z_{\ge0}^2}   \deq \{\Gpp^\rho_{v, \,v+x}\}_{x\in\Z_{\ge0}^2} .
\ee
This follows because the processes $\{\Gpp^{NE}_{v-x,\,v}\}_{x\in\Z_{\ge0}^2}$,    $\{\Gpp^{SW}_{v,\, v+x}\}_{x\in\Z_{\ge0}^2} $ and 
$ \{\Gpp^\rho_{v, \,v+x}\}_{x\in\Z_{\ge0}^2}$ are constructed with identical procedures from distributionally identical inputs.

\medskip

{\bf Busemann functions as extremals in the variational formulas.} 
Recall the variational formula for the point-to-point limit shape function, specialized to the two-dimensional corner growth model.  
	\begin{align}  \label{v:901}
	\gpp(\xi)&\;=\;\inf_{B\in\cK}\, \P\text{-}\esssup\;  \max_{i=1,2}\, \{\Yw_0-   B(0, e_i)-  h(B)\cdot \xi\}. 
	\end{align}
Above $h(B)$ is the negative of the mean vector:
\be\label{903} 
h(B)=-\bigl( \, \E[B(0,e_1)], \E[B(0,e_2)]\,\bigr). 
\ee
The next theorem shows that cocycle $B^\alpha$ from Theorem \ref{t:buse} is a minimizer in \eqref{v:901} for the characteristic direction $\xi(\alpha)$.   

\begin{theorem}  \label{t:exp-var}  Continue with the setting of Theorem \ref{t:buse}.  The following hold  for each $0<\alpha<1$. 

\begin{enumerate}
 \item[{\rm(i)}]   The  characteristic direction $\xi(\alpha)$ and the vector $h(B^\alpha)$ are dual to each other in the sense that 
 \be\label{904}   \gpl(h(B^\alpha)) = \gpp(\xi(\alpha))+ h(B^\alpha)\cdot\xi(\alpha) . \ee
 
 \item[{\rm(ii)}]   We have the identities  
 \be  \label{v:907}	\begin{aligned} 
		\gpp(\xi(\alpha))	&\;=\;     -  h(B^\alpha)\cdot \xi(\alpha) \\
		&\;=\;  \max_{i=1,2}\, \bigl[ \,\Yw_0(\w)-   B^\alpha_{0, e_i}(\w)  -  h(B^\alpha)\cdot \xi(\alpha)\,\bigr]  
	  \qquad\text{$\P$-almost surely.} 
	\end{aligned} \ee
  In other words,  $B^\alpha$ minimizes in the variational formula \eqref{v:901} in the characteristic direction $\xi(\alpha)$, without the essential supremum.
 \end{enumerate} 
 
\end{theorem} 

\begin{proof}  Part (i).  From \eqref{EB8}  we have 
\[  -h(B^\alpha)   = \nabla\gpp(\xi(\alpha)) \] 
and hence by the strict concavity of $\gpp$,  $\xi(\alpha)$ is the unique maximizer in the duality 
\[   \gpl(h(B^\alpha)) = \sup_{\xi\in\Uset} [ \,\gpp(\xi)+ h(B^\alpha)\cdot\xi \,]  . \] 

Part (ii).    By the duality \eqref{904},  by   $\gpl(h(B^\alpha)) =0$ from  Lemma \ref{lm-h(B)},  and by the  adaptedness \eqref{914}, 
\begin{align*}
\gpp(\xi(\alpha)) &= \gpl(h(B^\alpha))   - h(B^\alpha)\cdot\xi(\alpha) =  - h(B^\alpha)\cdot\xi(\alpha) \\
&= \Yw_0-B^\alpha_{0,e_1} \wedge B^\alpha_{0,e_2}  - h(B^\alpha)\cdot\xi(\alpha) \\
&=\max_{i=1,2}\, \bigl[ \,\Yw_0-   B^\alpha_{0, e_i}  -  h(B^\alpha)\cdot \xi(\alpha)\,\bigr] . 
\qedhere \end{align*}
\end{proof}

\medskip 
 
\subsection{Proof of the existence of Busemann functions} \label{sec:bus2}
  This section proves  the existence and properties of the limiting Busemann functions.    The proof is an adaptation of the one given in 
\cite{geor-rass-sepp-yilm-15}  for the exactly solvable log-gamma (or inverse gamma) polymer model.    

 Before specializing to the processes we study, 
  we state and prove   general inequalities  for planar last-passage increments.    Early appearances of these types of inequalities  in  first-passage percolation  can be found in \cite{alm-98, alm-wier-99}.  
Let  real weights $\{\wt Y_x\}_{x\in\Z^2}$ be given.  Define  last-passage times   
\be\label{Gtil7} \wt \Gpp_{x,y}=\max_{x_{0,n}}\sum_{k=0}^{n}\wt Y_{x_k}\ee
where the maximum is over up-right paths from $x_0=x$ to $x_n=y$ with $n=|y-x|_1$ and the admissible steps are as before  $x_k-x_{k-1}\in\{e_1,e_2\}$.  The convention is $\wt \Gpp_{v,v}=0$.  
For  $x\le v-e_1$ and $y\le v-e_2$  denote   increments  by 
	\[	\wt I_{x,v} = \wt \Gpp_{x,v} - \wt \Gpp_{x+e_1,v} \qquad  \text{ and } \qquad  \wt J_{y,v} = \wt \Gpp_{y,v} - \wt \Gpp_{y+e_2,v}\,. \]
For the precise statement of the lemma below it is important that the sum in \eqref{Gtil7} includes the first weight $\wt Y_x$, but  the increments  $\wt I$ and $\wt J$ are not sensitive to whether the last weight  $\wt Y_y$ is included or excluded. 
	
\begin{lemma}\label{lm-til}
	For  $x\le v-e_1$ and $y\le v-e_2$ 
		\be
		\label{til-8}
			\wt I_{x,v+e_2} \ge \wt I_{x,v} \ge \wt I_{x,v+e_1} 
			\qquad  \text{ and } \qquad 
			\wt J_{y, v+e_2} \le \wt J_{y,v} \le \wt J_{y, v+e_1}\,.
		\ee
\end{lemma}

\begin{proof}
	Let $v=(m,n)$. 
	The proof goes   by an induction argument that starts from the north and east boundaries. 
	On the north, for  $x = (k,n)$ for some $k < m$, 
		\begin{align*}
			\wt I_{(k,n), (m, n+1)} 
			&= \wt \Gpp_{(k,n), (m, n+1)} - \wt \Gpp_{(k +1,n), (m, n+1)} \\
			&= \wt Y_{k,n} + \wt \Gpp_{(k+1,n), (m, n+1)} \vee \wt \Gpp_{(k,n+1), (m, n+1)} - \wt \Gpp_{(k +1,n), (m, n+1)}\\
			&\ge \wt Y_{k,n}= \wt \Gpp_{(k,n), (m, n)} - \wt \Gpp_{(k+1,n), (m, n)} = \wt I_{(k,n), (m,n)}\,.
		\end{align*}
A similar argument (or the above inequality applied to transposed lattice points $(a',b')=(b,a)$) gives, for  $y= (m,\ell)$ for some $\ell < n$,  
\[  \wt J_{(m,\ell),(m+1,n)} \ge  \wt J_{(m,\ell),(m,n)} 	.\]  	
	
We also have the equalities, first for   $y= (m,\ell)$ for some $\ell < n$
		\begin{align*}
			\wt J_{(m,\ell),(m,n+1)}&=\wt \Gpp_{(m,\ell),(m,n+1)} - \wt \Gpp_{(m,\ell+1),(m,n+1)}\\
			&=\wt Y_{m,\ell}  = \wt \Gpp_{(m,\ell),(m,n)}-\wt \Gpp_{(m,\ell+1),(m,n)}  = \wt J_{(m,\ell),(m,n)}\, 
		\end{align*}
and similarly  also 
\[    \wt I_{(k,n), (m+1, n)} =  \wt I_{(k,n), (m, n)} . \]  
		
	These inequalities start the induction. Now let   $u\le v- e_1- e_2$.
	Assume by induction  that \eqref{til-8} holds for $x=u+e_2$ and $y=u+e_1$.   We prove the first inequalities of \eqref{til-8} for $x=u$. 
		\begin{align*}
			\wt I_{u,v+e_2}
			&=\wt \Gpp_{u,v+e_2}-\wt \Gpp_{u+e_1,v+e_2}
			= \wt Y_u + (\wt \Gpp_{u+e_2,v+e_2}-\wt \Gpp_{u+e_1,v+e_2})^+ \\
			&=\wt Y_u+(\wt I_{u+e_2,v+e_2} - \wt J_{u+e_1,v+e_2})^+  \\
			&\ge  \wt Y_u+(\wt I_{u+e_2,v} - \wt J_{u+e_1,v})^+ = \wt I_{u,v}\,.  
		\end{align*} 
		The last equality comes by repeating  the first three equalities with $v$ instead of $v+e_2$. 
		
	Replacing the pair $(v+e_2, v)$ with $(v, v+e_1)$ in the argument above gives  $\wt I_{u,v}\ge \wt I_{u,v+e_1}$.   A symmetric argument works for the $\wt J$ inequalities. 	
\end{proof}

We introduce the following general notational device which we illustrate in the context of \eqref{Gtil7}.  If  $\Lambda$ is a subset of admissible paths from $x$ to $y$, then 
\be\label{Gtil9}   \wt\Gpp_{x,y}(\Lambda)=\max_{x_{0,n}\in\Lambda} \sum_{k=0}^n \wt Y_{x_k}  \ee
is the last-passage value obtained when the maximum is restricted to paths $x_{0,n}$  in $\Lambda$.

\medskip 
 
 Return to the context of the CGM with exponential weights. 
Fix an origin $u\in\Z^2$ and a parameter $0<\rho<1$  for the moment.    
Assume given independent variables 
\be\label{IJw6.7}  \{   \sigma_{x} ,\, I_{u+ie_1}, \, J_{u+je_2}  :  \,x\in u+\Z_{>0}^2, \, i,j\in \Z_{>0}\}   \ee 
as in \eqref{IJw}, with marginal distributions \eqref{w7}.  (We use  $\sigma$  for these auxiliary Exp$(1)$ weights   instead of the $\w$  of  \eqref{IJw} to avoid confusion with the sample point $\w$ of  Theorem \ref{t:buse}.)  
Let the variables 
 \be\label{IJw9.3} 
  \{ \sigma_{x} ,\,   I_{x-e_2}, \, J_{x-e_1},  \,\sigmawc_{x-e_1-e_2}  :  x\in u+\Z_{>0}^2\}  
   \ee 
 be defined via equations \eqref{IJw5.1}--\eqref{IJw5.3} from   initial variables \eqref{IJw6.7}, so that   the properties   given in Theorem \ref{v:tIJw1} are satisfied, with $\sigma$ replacing $\w$. 
  
  Utilizing the edge weights $I_x$ and $J_x$ and the vertex weights $\sigmawc_x$ we define several last-passage percolation processes.  First a process with i.i.d.\ weights: 
 \be\label{Gch} 
  \wc \Gpp_{x,y}=\max_{x_{\brbullet}\,\in\,\Pi_{x,y}}\sum_{k=0}^{\abs{y-x}_1}\sigmawc_{x_k} \qquad\text{for $y\ge x\ge u$.} 
 \ee
Its increments are defined by  
 \[  \wc I_{x,v} =\wc\Gpp_{x,v}-\wc\Gpp_{x+e_1,v} \quad\text{for $u\le x\le v-e_1$} 
 \quad\text{and}\quad 
 \wc J_{y,v} =\wc\Gpp_{y,v}-\wc\Gpp_{y+e_2,v} \quad\text{for $u\le y\le v-e_2$} . \]
 
Introduce   an auxiliary increment-stationary  last-passage process  $ \wc\Gpp^{N\!E}_{x,v}$ for $u\le x\le v$,  with boundary edge weights on the north and east borders.  
\be\label{Gr26}\begin{aligned} 
&\text{On the northwest corner set   $ \wc\Gpp^{N\!E}_{v,v}=0$;} \\
&\text{on the north and east boundaries emanating from $v$ in the negative directions, }\\
 &\quad \wc\Gpp^{N\!E}_{v-ke_1,v}=\sum_{i=0}^{k-1} I_{v-ie_1} 
  \quad\text{and}\quad
\wc\Gpp^{N\!E}_{v-\ell e_2,v}= \sum_{j=0}^{\ell-1}  J_{v-je_2}\qquad \text{for $k,\ell\ge 1$} ;   \\
&\text{in the bulk for $u\le x\le v-e_1-e_2$ , }\\
&\quad \wc\Gpp^{N\!E}_{x,v} = \sigmawc_x + \wc\Gpp^{N\!E}_{x+e_1,v}\vee \wc\Gpp^{N\!E}_{x+e_2,v} \\
&\quad
= \max_{1\le k\le (v-x)\cdot e_1} \;  \Bigl\{    \wc\Gpp_{x, v-ke_1-e_2}+\sum_{i=0}^{k-1} I_{v-ie_1}  \Bigr\}  
\bigvee
 \max_{1\le \ell\le (v-x)\cdot e_2}\; \Bigl\{  \wc\Gpp_{x, v-e_1-\ell e_2}+ \sum_{j=0}^{\ell-1}  J_{v-je_2}    \Bigr\} . 
\end{aligned}  \ee

%

In the next lemma we check that the increments of the $\wc\Gpp^{N\!E}$ process are in fact the  $I$ and $J$ variables already given in \eqref{IJw9.3}.  

\begin{lemma} \label{v:lm86}  For $u\le x\le v-e_1$ and $u\le y\le v-e_2$, 
\be\label{INE3} \begin{aligned} 
I_{x+e_1}=\wc\Gpp^{N\!E}_{x,v}   -\wc\Gpp^{N\!E}_{x+e_1,v}   
\qquad\text{and}\qquad 
J_{y+e_2}=\wc\Gpp^{N\!E}_{y,v}   -\wc\Gpp^{N\!E}_{y+e_2,v}. 
\end{aligned} \ee
\end{lemma}

\begin{proof} The claim is true for $x=v-ke_1$ and $y=v-\ell e_2$ by definition \eqref{Gr26}.   Here is the induction step for the edge $(x,x+e_1)$, assuming that \eqref{INE3} has been proved for edges $(x+e_2, x+e_1+e_2)$ and $(x+e_1, x+e_1+e_2)$. 
\begin{align*}
\wc\Gpp^{N\!E}_{x,v}   -\wc\Gpp^{N\!E}_{x+e_1,v}   
&= \sigmawc_x+ \bigl( \wc\Gpp^{N\!E}_{x+e_2,v}   -\wc\Gpp^{N\!E}_{x+e_1,v} \bigr)^+
= \sigmawc_x+ \bigl( I_{x+e_1+e_2}   -J_{x+e_1+e_2} \bigr)^+\\
&= I_{x+e_1}   \wedge J_{x+e_2}+ \bigl( I_{x+e_1}   -J_{x+e_2} \bigr)^+ = I_{x+e_1}. 
\end{align*}
The second last equality used \eqref{IJw5.1} and the cocycle property 
$I_{x+e_1}+J_{x+e_1+e_2} = J_{x+e_2}+I_{x+e_1+e_2} $.  A similar argument extends \eqref{INE3} to the edge $(x,x+e_2)$. 
\end{proof} 

Next we consider restricted $\wc\Gpp^{N\!E}$ last-passage values utilizing the notation introduced in \eqref{Gtil9}.  In particular, we consider last-passage values of the kind 
\be\label{G-restr}  \wc\Gpp^{N\!E}_{x,\,v}(v-e_1\in x_{\bbullet})
= \max_{1\le k\le (v-x)\cdot e_1} \;  \Bigl\{    \wc\Gpp_{x, v-ke_1-e_2}+\sum_{i=0}^{k-1} I_{v-ie_1}  \Bigr\}  
\ee
where the condition $v-e_1\in x_{\bbullet}$ means that the path goes through the point $v-e_1$,  which is equivalent to saying that the last step of the path goes from $v-e_1$ to $v$. 
We show that the asymptotics of the restricted $\wc\Gpp^{N\!E}$ are the expected ones and calculate the limits.

 \begin{lemma}\label{v:lm-88}  
  Fix a point $a\in u+\Z_{\ge0}^2$ and reals $0<s,t<\infty$. Let $v_n\in u+\Z_{\ge0}^2$ be such that  $\abs{v_n}_1\to\infty$ and   $v_n/\abs{v_n}_1\to(s,t)/(s+t)$ as $n\to\infty$.  
  Then we have the following almost sure  limits. 
		\be\label{v:800}\begin{aligned} 
			\abs{v_n}_1^{-1}\,\wc\Gpp^{N\!E}_{a,\,v_n}(v_n-e_1\in x_{\bbullet}) \; \underset{n\to\infty}{ \longrightarrow }\;
	&(s+t)^{-1}\sup_{0 \le \tau \le s}\Bigl\{ \frac{\tau}{1-\rho}  + \gpp(s-\tau,t) \Bigr\}  \\[6pt]
			=\;& 
\begin{cases}     \dfrac{\gpp(s,t)}{s+t}  = \dfrac{\bigl(\sqrt s+\sqrt t\,\bigr)^2}{s+t}, &\dfrac{s}{t}\le \biggl(\dfrac{1-\rho}\rho\biggr)^2\\[10pt] 
\dfrac1{s+t}\biggl(\dfrac{s}{1-\rho}+\dfrac{t}{\rho}\biggr)\,, &\dfrac{s}{t}\ge \biggl(\dfrac{1-\rho}\rho\biggr)^2  
\end{cases} 
\end{aligned} 		\ee

and 
 		\be\label{v:801}\begin{aligned} 
			\abs{v_n}_1^{-1}\,\wc\Gpp^{N\!E}_{a,\,v_n}(v_n-e_2\in x_{\bbullet}) \; \underset{n\to\infty}{\longrightarrow} \;
			&(s+t)^{-1}\sup_{0 \le \tau \le t}\Bigl\{ \frac\tau\rho + \gpp(s, t-\tau) \Bigr\}\\[6pt] 
			=\;& 
\begin{cases}     \dfrac{\gpp(s,t)}{s+t}  = \dfrac{\bigl(\sqrt s+\sqrt t\,\bigr)^2}{s+t}, &\dfrac{s}{t}\ge \biggl(\dfrac{1-\rho}\rho\biggr)^2\\[10pt] 
\dfrac1{s+t}\biggl(\dfrac{s}{1-\rho}+\dfrac{t}{\rho}\biggr)\,, &\dfrac{s}{t}\le \biggl(\dfrac{1-\rho}\rho\biggr)^2. 
\end{cases} 
\end{aligned} 	 		\ee
 \end{lemma}

\begin{proof}
	We prove the limit in  \eqref{v:800}. Finding the supremum is calculus.    The proof of \eqref{v:801} being entirely analogous. Fix $\e>0$, 
	  let  $M=\fl{\e^{-1}}$,  and 
		\[
			q^n_j = j\Bigl\lfloor\frac{\e\abs{v_n}_1 s}{s+t}\Bigr\rfloor \, \text{ for }0\le j \le M-1, 
				\text{ and } q^n_{M}= (v_n-a)\cdot e_1.
		\]
		For large enough $n$ it is the case that  $q^n_{M} -C\e\abs{v_n}_1< q^n_{M-1}<q^n_{M}$.  

Suppose    a maximal path for $\wc\Gpp^{N\!E}_{a,\,v_n}(v_n-e_1\in x_{\bbullet})$ enters 
	 the north boundary from  the bulk at the point 
  $v_n-(\ell,0)$  with $q^n_{j} < \ell\le q^n_{j+1}$.  By nonnegativity of the weights, 
		\begin{align*}
			&\wc\Gpp^{N\!E}_{a,\,v_n}(v_n-e_1\in x_{\bbullet}) 
= \wc\Gpp_{a, \,v_n-(\ell,1)}   + \sum_{i=0}^{\ell-1}   I_{v_n-ie_1}
\\
			&\qquad 
	\le \wc\Gpp_{a, \,v_n-(q^n_j,1)}   +  \frac{q^n_j}{1-\rho} 
 + \sum_{i=0}^{q^n_{j+1}-1}  \bigl( I_{v_n-ie_1} - \tfrac1{1-\rho}\bigr)  	
 + \frac{ q^n_{j+1}-q^n_j}{1-\rho} . 
		\end{align*}
 Collect the bounds for all the intervals $(q^n_j, q^n_{j+1}]$:
	\be	\begin{aligned}
			\wc\Gpp^{N\!E}_{a,\,v_n}(v_n-e_1\in x_{\bbullet})  
				&\;\le\;\max_{0\le j \le M-1 }\!\!\Big\{ 		
		\wc\Gpp_{a, \,v_n-(q^n_j,1)}   +  \frac{q^n_j}{1-\rho} \\
&\qquad \qquad \qquad 
 + \sum_{i=0}^{q^n_{j+1}-1}  \bigl( I_{v_n-ie_1} - \tfrac1{1-\rho}\bigr)  	
 + \frac{ q^n_{j+1}-q^n_j}{1-\rho}  \Bigr\}. 
	\end{aligned}\label{G:bnd} \ee 		
Divide through by $\abs{v_n}_1$ and let $n\to\infty$.  
On the right-hand side above,  Theorem \ref{t:sh}  and the homogeneity of $\gpp$ give 
\begin{align*}
\frac{\wc\Gpp_{a, \,v_n-(q^n_j,1)}}{\abs{v_n}_1} = \gpp\biggl( \frac{v_n}{\abs{v_n}_1} 
-\frac{q^n_j}{\abs{v_n}_1} e_1 + \frac{O(1)}{\abs{v_n}_1}\biggr) +o(1) 
\; \longrightarrow \;  \frac{\gpp(s-sj\e, t)}{s+t}. 
\end{align*}
The mean zero i.i.d.\ sum satisfies 
\[ \frac{1}{\abs{v_n}_1} \sum_{i=0}^{q^n_{j+1}-1}  \bigl( I_{v_n-ie_1} - \tfrac1{1-\rho}\bigr) 
\; \longrightarrow \; 0.  \]
We get  the upper bound
\begin{align*}
\varlimsup_{n\to\infty}   \abs{v_n}_1^{-1} \wc\Gpp^{N\!E}_{0,\,v_n}(v_n-e_1\in x_{\bbullet})  
   & \le 
 (s+t)^{-1} \max_{0\le j \le M-1 }\Big[\gpp(s-sj\e, t) +\frac{sj\e}{1-\rho} + C\e  \Big]\\
 &\le (s+t)^{-1}\sup_{0 \le \tau \le s}\Bigl[ \,\frac{\tau}{1-\rho}  + \gpp(s-\tau,t) \Bigr]+ C\e.  
 \end{align*} 
 Let $\e\searrow 0$ to complete the proof of the upper bound.

  To get the  matching lower bound let  the supremum 
  \[  \ddd\sup_{\tau\in [0,s]} \{ \tfrac\tau{1-\rho} +  \gpp(s-\tau , t)\}\]  be attained at $\tau^* \in [0,s]$.  
	With $m_n=\abs{v_n}_1/(s+t)$ we have 
 		\begin{align*}
	\wc\Gpp^{N\!E}_{a,\,v_n}(v_n-e_1\in x_{\bbullet}) &\ge 	\wc\Gpp_{a,v_n-(\fl{m_n\tau^*}\vee1,1)}\; +\; \sum_{i=0}^{(\fl{m_n\tau^*}-1)^+}   I_{v_n-ie_1}. 
		\end{align*}
Let $n\to\infty$ to get 
		\begin{align*}
\varliminf_{n\to\infty}  		\abs{v_n}_1^{-1}\wc\Gpp^{N\!E}_{0,\,v_n}(v_n-e_1\in x_{\bbullet})
		\ge  (s+t)^{-1}\Bigl[\gpp(s-\tau^*,t) + \frac{\tau^*}{1-\rho} \,\Bigr] . 
		\end{align*}
This completes  the proof of  the limit in  \eqref{v:800}.
\end{proof}

As a consequence we record  the expected asymptotics for the unrestricted process with edge weights on the north and east:
\be\label{v:802}\begin{aligned}
	&\lim_{n\to\infty} \abs{v_n}_1^{-1}\,\wc\Gpp^{N\!E}_{a,\,v_n}	
	=	\lim_{n\to\infty} \abs{v_n}_1^{-1}\,\wc\Gpp^{N\!E}_{a,\,v_n}(v_n-e_1\in x_{\bbullet})  \bigvee  \wc\Gpp^{N\!E}_{a,\,v_n}(v_n-e_2\in x_{\bbullet})  \\
	&\qquad 
	=(s+t)^{-1}\sup_{0 \le \tau \le s}\Bigl\{\, \frac{\tau}{1-\rho}  + \gpp(s-\tau,t) \Bigr\}  
\bigvee	 \sup_{0 \le \tau \le t}\Bigl\{ \,\frac\tau\rho + \gpp(s, t-\tau) \Bigr\} \\
	&\qquad 
	=(s+t)^{-1}\sup_{0 \le \tau \le s}\Bigl\{ \,\frac{\tau}{1-\rho}  + \bigl(\sqrt{s-\tau} +\sqrt t\,\bigr)^2 \Bigr\}  
\bigvee	 \sup_{0 \le \tau \le t}\Bigl\{ \, \frac\tau\rho + \bigl(\sqrt{s} +\sqrt{t-\tau}\,\bigr)^2 \Bigr\} \\
&\qquad = \frac1{s+t}\biggl( \frac{s}{1-\rho}+\frac{t}{\rho}\biggr). 
\end{aligned} \ee
	
In the next lemma we derive bounds on the limiting local gradients of the last-passage values $\wc \Gpp$ defined in \eqref{Gch} in  terms of the i.i.d.\ $\sigmawc$ weights.  Recall   definition \eqref{v:ch7}  of the characteristic direction 
$\xi(\rho)=\bigl( \frac{(1-\rho)^2} {(1-\rho)^2+\rho^2}\,, \frac{\rho^2}{(1-\rho)^2+\rho^2}\bigr)$.  

 \begin{lemma}\label{v:lm-89}  
  Consider two sequences $\{v_n\}$ and $\{w_n\}$ in $a+\Z_{\ge 0}^2$  such that 
  $\abs{v_n}_1\wedge \abs{w_n}_1\to\infty$ and 
   \[     \lim_{n\to\infty}  \frac{v_n}{\abs{v_n}_1} = (s, 1-s)  
   \quad\text{and}\quad   \lim_{n\to\infty}  \frac{w_n}{\abs{w_n}_1} = (t, 1-t) .  \]
 Assume that   
 \[  s <  \xi_1(\rho)=\frac{(1-\rho)^2}{(1-\rho)^2+\rho^2}   < t.  \]
 Then we have the following statements.  
 
 \begin{enumerate}  
\item[{\rm(i)}] Almost surely 
 \be\label{v:805}   \wc\Gpp^{N\!E}_{a,v_n } (v_n-e_2\in x_{\bbullet}) = \wc\Gpp^{N\!E}_{a,v_n} 
  \qquad\text{and}\qquad  
 \wc\Gpp^{N\!E}_{a,w_n } (w_n-e_1\in x_{\bbullet}) = \wc\Gpp^{N\!E}_{a,w_n} \ee 
 for all large enough $n$.  \\[-2pt]
 
 \item[{\rm(ii)}]   The following inequalities hold almost surely:  
 \be\label{v:832} 
 \varlimsup_{n\to\infty}  [ \wc\Gpp_{a,w_n}-\wc\Gpp_{a+e_1,w_n}]   \le   I_{a+e_1}
 \le  \varliminf_{n\to\infty}  [ \wc\Gpp_{a,v_n}-\wc\Gpp_{a+e_1,v_n}] 
  \ee
  and 
  \be\label{v:833} 
 \varlimsup_{n\to\infty}  [ \wc\Gpp_{a,v_n}-\wc\Gpp_{a+e_2,v_n}]   \le   J_{a+e_2}
 \le  \varliminf_{n\to\infty}  [ \wc\Gpp_{a,w_n}-\wc\Gpp_{a+e_2,w_n}]  . 
  \ee
\end{enumerate}  
 \end{lemma} 
 
 \begin{proof}     Part (i).    We prove the second statement of \eqref{v:805}.   The maximizing path to $w_n$ comes through either $w_n-e_1$ or $w_n-e_2$.  So to get a contradiction we can assume that $\P(A)>0$ for the event $A$ on which   $\wc\Gpp^{N\!E}_{a,w_n } (w_n-e_2\in x_{\bbullet}) = \wc\Gpp^{N\!E}_{a,w_n}$  happens for infinitely many $n$.   On the event $A$ we can take limits \eqref{v:801} and \eqref{v:802}  to get 
 \begin{align*}
  \sup_{0 \le \tau \le 1-t}\Bigl\{ \,\frac\tau\rho + \bigl(\sqrt{t} +\sqrt{1-t-\tau}\,\bigr)^2 \Bigr\}
 = \frac{t}{1-\rho}+\frac{1-t}{\rho}. 
 \end{align*} 
But by \eqref{v:801},  $\frac{t}{1-t}> \bigl(\frac{1-\rho}{\rho}\bigr)^2$ implies that the supremum on the left equals $\bigl(\sqrt t+\sqrt{1-t}\,\bigr)^2$  which is strictly less than the right-hand side by \eqref{v:ch8}.    Thus $\P(A)>0$ is not possible.  
 
 \medskip 
 
  Part (ii).     We prove the statements for $w_n$.   By a combination of developments from above we derive the   sequence of inequalities and equalities in \eqref{v:836} below.  The steps are justified one by one after the inequalities.  In the first step,  $\wc\Gpp^N_{a, w_n+e_2}$ denotes a last-passage process in the rectangle $\{x:a\le x\le w_n+e_2\}$  that uses $\sigmawc$ weights on the horizontal lines below the top one, and the $I$ weights on the top horizontal line $x\cdot e_2=w_n\cdot e_2+1$ (north boundary of the rectangle), with an irrelevant  zero weight assigned at the top right corner $w_n+e_2$. 
  The third  step  below  is valid almost surely for large $n$.  This  proves the first inequality of \eqref{v:832}.    
\be\label{v:836}   \begin{aligned}
\wc\Gpp_{a,w_n}-\wc\Gpp_{a+e_1,w_n} &\le \wc\Gpp^N_{a,\,w_n+e_2}-\wc\Gpp^N_{a+e_1,\,w_n+e_2}  \\
&= \wc\Gpp^{N\!E}_{a,\,w_n+e_1+e_2} (w_n+e_2\in x_{\bbullet}) -\wc\Gpp^{N\!E}_{a+e_1,\,w_n+e_1+ e_2} (w_n+e_2\in x_{\bbullet}) \\
&= \wc\Gpp^{N\!E}_{a,\,w_n+e_1+e_2}   -\wc\Gpp^{N\!E}_{a+e_1,\,w_n+e_1+ e_2}   \\
&= I_{a+e_1}.  \end{aligned}\ee

 The first inequality in \eqref{v:836} above  is a special case  of  the first inequality of \eqref{til-8},  applied to the situation where the weights in \eqref{Gtil7}  are given by 
  \[  \wt Y_x =\begin{cases}  0&\text{if $x=w_n+e_2$,} \\  I_{x+e_1} &\text{if $x=w_n+e_2-ie_1$ for some $i\ge 1$,} \\   \sigmawc_x &\text{if } x \le w_n 
  \end{cases}   
  \]  

In the  first equality in \eqref{v:836} we move the upper right corner from $w_n+e_2$ one step to the right to $w_n+e_1+e_2$ so that we can include the boundary weights both on the north and east boundaries.   This is exactly the definition of $ \wc\Gpp^{N\!E}$ in \eqref{Gr26} and \eqref{Gr26}.    To preserve the equality we force the paths to go through $w_n+e_2$.  

The second equality in \eqref{v:836} is valid almost surely for large enough $n$, by the already proved \eqref{v:805}.     The last equality comes from  \eqref{INE3}. 

Similarly we reason for the $e_2$ increment:  
\be\label{v:838}   \begin{aligned}
\wc\Gpp_{a,w_n}-\wc\Gpp_{a+e_2,w_n} &\ge \wc\Gpp^N_{a,\,w_n+e_2}-\wc\Gpp^N_{a+e_2,\,w_n+e_2}  \\
&= \wc\Gpp^{N\!E}_{a,\,w_n+e_1+e_2} (w_n+e_2\in x_{\bbullet}) -\wc\Gpp^{N\!E}_{a+e_2,\,w_n+e_1+ e_2} (w_n+e_2\in x_{\bbullet}) \\
&= \wc\Gpp^{N\!E}_{a,\,w_n+e_1+e_2}   -\wc\Gpp^{N\!E}_{a+e_2,\,w_n+e_1+ e_2}   \\
&= J_{a+e_2}.  \end{aligned}\ee
This proves the last inequality of \eqref{v:833}. 
 \end{proof}   


Next we use the estimates above to build an exponential-$\alpha$ last-passage system from   limits of local gradients of last-passage values.  
  Denote the Exp$(\lambda)$ cumulative distribution function by 
\[  F_\lambda(s)=\begin{cases} 0,  &s<0\\ 1-e^{-\lambda s}, &s\ge 0. \end{cases} \]

 \begin{lemma}\label{v:lm-vn-a}  Let i.i.d.\ {\rm Exp(1)} weights $\Yw=(\Yw_x)_{x\in\Z^2}$ be given and define the point-to-point last-passage process $\{\Gpp_{x,y}\}$ by \eqref{v:GY}.  Fix  $0<\alpha<1$ and a sequence $v_n\in\Z^2$ such that $\abs{v_n}_1\to\infty$ and 
 \be\label{v:843}     \lim_{n\to\infty}  \frac{v_n}{\abs{v_n}_1} \, = \,\xi(\alpha)\,=\, 
\biggl(  \frac{(1-\alpha)^2}{(1-\alpha)^2+\alpha^2} \,,   \frac{\alpha^2}{(1-\alpha)^2+\alpha^2}\biggr) \, \in \,\ri\Uset  . 
 \ee
 
 \begin{enumerate}\itemsep=3pt
\item[{\rm(i)}]  The  limits 
 \be\label{v:845}  
 B^\alpha_{x,y} =\lim_{n\to\infty} [ G_{x, v_n}-G_{y,v_n}]  \ee
 exist  $\P$-almost surely for all $x,y\in\Z^2$ and satisfy additivity $B^\alpha_{x,y}+B^\alpha_{y,z}=B^\alpha_{x,z}$.  
 
 \item[{\rm(ii)}]    Define 
 \be\label{v:X}  
  \Xw^\alpha_x= B^\alpha_{x-e_1,x}\wedge  B^\alpha_{x-e_2,x}\qquad \text{for $x\in\Z^2$.}  
 \ee
 Then the process 
 \[ \{  \Xw^\alpha_x, \,B^\alpha_{x-e_1,x}, \,B^\alpha_{x-e_2,x}, \,\Yw_x:  x\in\Z^2\}  \] 
 is an exponential-$\alpha$ last-passage system as described in Definition \ref{v:d-exp-a}. 
 \end{enumerate}  
 \end{lemma} 
 
 \begin{proof}  Part (i).   Fix $a\in\Z^2$ and let 
 \be\label{845.81}   \wb B=\varlimsup_{n\to\infty} [ \Gpp_{a, v_n}-\Gpp_{a+e_1,v_n}]
   \quad\text{and}\quad 
 \underline B=\varliminf_{n\to\infty} [ \Gpp_{a, v_n}-\Gpp_{a+e_1,v_n}]. \ee
 
  To get control of the distributions of $\wb B$ and $\ul B$, we realize the process $\{\Gpp_{a,y}\}$ on another probability space as an instance of $\{\wc\Gpp_{a,y}\}$.  Then we can apply bounds \eqref{v:832}--\eqref{v:833}. 
  
 Let $0<\lambda<\alpha<\rho<1$.    This implies 
 \be\label{v:846.4}  \frac{(1-\rho)^2}{(1-\rho)^2+\rho^2}< \frac{(1-\alpha)^2}{(1-\alpha)^2+\alpha^2}<\frac{(1-\lambda)^2}{(1-\lambda)^2+\lambda^2}. \ee 
 
 Take  any lattice point  $u\le a$ as an origin. 
 Suppose on some arbitrary probability space we have mutually independent variables 
 $\sigma=(\sigma_x)_{x\in\Z^2}$,  $(I^\lambda_{u+ie_1})_{i\ge 1}$,  $(J^\lambda_{u+je_2})_{j\ge 1}$,  $(I^\rho_{u+ie_1})_{i\ge 1}$,  and $(J^\rho_{u+je_2})_{j\ge 1}$ with marginal distributions 
 \be\label{845.01}  \begin{aligned}  \sigma_x\sim \text{Exp}(1), \quad &I^\lambda_{u+ie_1}\sim \text{Exp}(1-\lambda), \quad   J^\lambda_{u+je_2}\sim \text{Exp}(\lambda ), \\
 &   I^\rho_{u+ie_1}\sim \text{Exp}(1-\rho), 
\quad \text{and}\quad  J^\rho_{u+je_2}\sim \text{Exp}(\rho). 
\end{aligned}  \ee
 In other words,  for parameters $\lambda$ and $\rho$ we have initial conditions of the kind described in \eqref{IJw} and \eqref{w7}.   Iterating with equations  \eqref{IJw5.1}--\eqref{IJw5.3}, on the quadrant $u+\Z_{\ge0}^2$,  construct two  processes  of the type \eqref{IJw9}: one denoted by 
 \be\label{v:847}    
  \{ \sigma_{x} ,\,   I^\lambda_{x-e_2}, \, J^\lambda_{x-e_1},  \,\wc\sigma^{[\lambda]}_{x-e_1-e_2}  :  x\in u+\Z_{>0}^2\}   
\ee
with parameter $\lambda$, and the other denoted by 
 \be\label{v:848}  
  \{ \sigma_{x} ,\,   I^\rho_{x-e_2}, \, J^\rho_{x-e_1},  \,\wc\sigma^{[\rho]}_{x-e_1-e_2}  :  x\in u+\Z_{>0}^2\}   
\ee
with parameter  $\rho$.  By construction these processes have the properties given in Theorem \ref{v:tIJw1}, process  \eqref{v:847} with parameter $\lambda$   and  with configuration $\wc\sigma^{[\lambda]}$ replacing $\wc\w$, and process \eqref{v:848} with parameter $\rho$ (as stated in Theorem \ref{v:tIJw1})   but  with  $\wc\sigma^{[\rho]}$ replacing $\wc\w$.   

We stipulated   that the initial edge weights in \eqref{845.01}  on the axes $\{u+ie_k: i\ge1, k=1,2\}$ for the $\lambda$ and $\rho$ systems were independent. In fact 
 the coupling between processes  \eqref{v:847} and \eqref{v:848} is immaterial because  the two processes will not be used jointly.

The key point is that  both $\{\wc\sigma^{[\lambda]}\}_{x\,\in\, u+\Z_{\ge0}^2}$ and $\{\wc\sigma^{[\rho]}\}_{x\,\in\, u+\Z_{\ge0}^2}$ are i.i.d.\ Exp(1) variables.  The superscripts $[\lambda]$ and $[\rho]$ simply remind us that these variables were constructed from edge variables with parameters $\lambda$ and $\rho$, respectively.   
   As  in \eqref{Gch}, define  last-passage processes $\wc\Gpp^{[\lambda]}$ and  $\wc\Gpp^{[\rho]}$ that use  i.i.d.\ Exp(1) weights $\wc\sigma^{[\lambda]}$ and $\wc\sigma^{[\rho]}$, respectively: 
 \be\label{Gch17} 
  \wc \Gpp^{[\lambda]}_{x,y}=\max_{x_{\brbullet}\,\in\,\Pi_{x,y}}\sum_{k=0}^{\abs{y-x}_1}\sigmawc^{[\lambda]}_{x_k} 
  \quad\text{and}\quad 
 \wc \Gpp^{[\rho]}_{x,y}=\max_{x_{\brbullet}\,\in\,\Pi_{x,y}}\sum_{k=0}^{\abs{y-x}_1}\sigmawc^{[\rho]}_{x_k}   
  \qquad\text{for $y\ge x\ge u$.} 
 \ee
The  distributional equality of last-passage  processes 
\be\label{845.38} 
\{ \wc\Gpp^{[\lambda]}_{x,y}: y\ge x\ge u\} \; \deq \; \{ \wc\Gpp^{[\rho]}_{x,y}: y\ge x\ge u\} \; \deq \; \{ \Gpp_{x,y}: y\ge x\ge u\}
\ee
follows   because they are defined the same way from i.i.d.\ Exp(1) weights. 


We derive bounds for the distribution functions of $\wb B$ and $\underline B$. 
\begin{align*}
\P\{ \wb B \le s\}&= \P\bigl\{ \, \varlimsup_{n\to\infty} ( G_{a, v_n}-G_{a+e_1,v_n}) \le s \bigr\}
= \P\bigl\{ \, \varlimsup_{n\to\infty} ( \wc\Gpp^{[\rho]}_{a, v_n}-\wc\Gpp^{[\rho]}_{a+e_1,v_n})  \le s \bigr\}  \\
& \ge \P\{ I^\rho_{a+e_1} \le s\} = F_{1-\rho}(s)  
\end{align*}
and 
\begin{align*}
\P\{ \underline B \le s\}&= \P\bigl\{ \, \varliminf_{n\to\infty} ( G_{a, v_n}-G_{a+e_1,v_n}) \le s \bigr\}
= \P\bigl\{ \, \varliminf_{n\to\infty} ( \wc\Gpp^{[\lambda]}_{a, v_n}-\wc\Gpp^{[\lambda]}_{a+e_1,v_n})  \le s \bigr\}  \\
& \le \P\{ I^\lambda_{a+e_1} \le s\} = F_{1-\lambda}(s). 
\end{align*}
Above we first replaced the weights $\w$ with $\wc\sigma^{[\lambda]}$ and $\wc\sigma^{[\rho]}$, respectively.   Then we applied \eqref{v:832},  as justified by \eqref{v:846.4}.  Last we used the  known distributions of the $I$ increment variables.     Since $\wb B\ge\underline B$ always, we have deduced that 
\[  F_{1-\rho}(s) \le  \P\{ \wb B \le s\} \le \P\{ \underline B \le s\} \le F_{1-\lambda}(s)
\qquad \text{for all $\lambda, \rho$ such that $\lambda<\alpha<\rho$. }   
\]
Letting $\lambda\nearrow\alpha$ and $\rho\searrow\alpha$ allows us to conclude that $\wb B=\underline B\sim$ Exp$(1-\alpha)$.    This proves the limit in \eqref{v:845} for $(x,y)=(a,a+e_1)$.    Proof of the limit for $(x,y)=(a,a+e_2)$  proceeds analogously.   
Since $a\in\Z^2$ was  arbitrary, we have the limit in \eqref{v:845} for all nearest-neighbor pairs $x,y$.  

An arbitrary increment $y-x$ can be decomposed into a sum of  nearest-neighbor increments, and then the limit   follows  for all pairs $x,y$ by the additivity on the right-hand side of  \eqref{v:845}.  Along the way one also derives the   additivity $B^\alpha_{x,y}+B^\alpha_{y,z}=B^\alpha_{x,z}$.  

\medskip 

Part (ii).   We need to verify the properties in Definition \ref{v:d-exp-a}.  We begin with the joint distribution of $B^\alpha$-increments along a down-right path.  

Consider   the joint distribution of $k+\ell$ nearest-neighbor increments  $B^\alpha_{x_i,x_i+e_1}$ and $B^\alpha_{y_j,y_j+e_2}$ for $1\le i\le k$ and $1\le j\le \ell$.  Fix an origin $u$ such that all $x_i$ and $y_j$ lie in the quadrant $u+\Z_{\ge0}^2$.   
By the distributional equality \eqref{845.38},    part (i)  gives the existence of the almost  sure limits 
\[  \wc B^{[\lambda]}_{x,y} =\lim_{n\to\infty} ( \Gpp^{[\lambda]}_{x, v_n}-\Gpp^{[\lambda]}_{y,v_n})  
\quad\text{and}\quad 
 \wc B^{[\rho]}_{x,y} =\lim_{n\to\infty} ( \Gpp^{[\rho]}_{x, v_n}-\Gpp^{[\rho]}_{y,v_n}) 
 \quad \forall x,y\in u+\Z_{\ge0}^2.  
 \]

  Then, first by distributional equality of processes,  
\begin{align*}
&\P\bigl\{  B^\alpha_{x_i,\,x_i+e_1}\le s_i, \, B^\alpha_{y_j,\,y_j+e_2}>t_j\ \ \forall i\in[k], j\in[\ell]\,\bigr\} \\
&=
\P\bigl\{  \wc B^{[\lambda]}_{x_i,\,x_i+e_1}\le s_i, \, \wc B^{[\lambda]}_{y_j,\,y_j+e_2}>t_j\ \ \forall i\in[k], j\in[\ell]\,\bigr\}  \\
&\le  \P\bigl\{  I^\lambda_{x_i+e_1}\le s_i, \, J^\lambda_{y_j+e_2}>t_j\ \ \forall i\in[k], j\in[\ell]\,\bigr\} 
\end{align*}
The last step came from the inequalities \eqref{v:832} and \eqref{v:833} (the case of $v_n$ of those inequalities is the right one to look at).   Similarly, using the remaining two inequalities of \eqref{v:832} and \eqref{v:833}, we deduce 
\begin{align*}
&\P\bigl\{  B^\alpha_{x_i,\,x_i+e_1}\le s_i, \, B^\alpha_{y_j,\,y_j+e_2}>t_j\ \ \forall i\in[k], j\in[\ell]\,\bigr\} \\
&=
\P\bigl\{  \wc B^{[\rho]}_{x_i,\,x_i+e_1}\le s_i, \, \wc B^{[\rho]}_{y_j,\,y_j+e_2}>t_j\ \ \forall i\in[k], j\in[\ell]\,\bigr\}  \\
&\ge  \P\bigl\{  I^\rho_{x_i+e_1}\le s_i, \, J^\rho_{y_j+e_2}>t_j\ \ \forall i\in[k], j\in[\ell]\,\bigr\} . 
\end{align*}
Assume in particular now that the edges $\{x_i,x_i+e_1\}$ and $\{y_j,y_j+e_2\}$ lie on a given down-right path.   Then,   part (i) of Theorem \ref{v:tIJw1} applied to the processes \eqref{v:847} and \eqref{v:848} turns the bounds above into   
\begin{align*}
\prod_{i\in[k]}  F_{1-\rho}(s_i) \cdot   \prod_{j\in[\ell]} (1- F_\rho(t_j)) 
& \; \le \; \P\bigl\{  B^\alpha_{x_i,\,x_i+e_1}\le s_i, \, B^\alpha_{y_j,\,y_j+e_2}>t_j\ \ \forall i\in[k], j\in[\ell]\,\bigr\} \\
&\qquad\qquad \le \prod_{i\in[k]}  F_{1-\lambda}(s_i) \cdot   \prod_{j\in[\ell]} (1- F_\lambda(t_j))  . 
\end{align*} 
Letting again $\lambda\nearrow\alpha$ and $\rho\searrow\alpha$ shows  that along a down-right path, the variables $B^\alpha_{x,\,x+e_1}$ and $B^\alpha_{y,\,y+e_2}$ are independent with distributions Exp$(1-\alpha)$ and Exp$(\alpha)$, respectively, as required by part (a) of Definition \ref{v:d-exp-a}. 

Fix a down-right path $\cY$ in $\Z^2$.  
We verify the distributional properties on $\cG_-$, $\cY$ and $\cG_+$ inside an arbitrarily large rectangle.  

Consider a large  rectangle $\cD=\{x\in\Z^2: (M_0, N_0)\le x\le (M_1, N_1)\}$ whose lower left  and upper right corners are $(M_0, N_0)$ and $(M_1, N_1)$.  The $e_1$-edge variables $B^\alpha_{ (i, N_1), (i+1, N_1)}$ for $M_0\le i\le M_1-1$ on the north boundary, the  $e_2$-edge variables $B^\alpha_{ (M_1, j), (M_1, j+1)}$ for $N_0\le j\le N_1-1$ on the east boundary, and the bulk variables  $\Yw_x$ for $(M_0, N_0)\le x\le (M_1-1, N_1-1)$ are mutually independent.   (The $\Yw$-variables are independent of the $B^\alpha$-variables to their north and east because limit \eqref{v:845} constructs a $B^\alpha$-variable in terms of $\Yw$-weights to its north and east.)  

 In other words,  the $B^\alpha$-increments on the north and east boundaries and the $\Yw$-weights in the bulk of the rectangle satisfy the properties of the $I,J$ boundary weights and $\wc\w$-bulk weights in Theorem \ref{v:tIJw1}.  Next we show by a south and westward  induction that the joint distribution of $(\Xw^\alpha, B, \Yw)$ is the  correct one.  


We claim that the variables satisfy the equations 
\be\label{v:850} \begin{aligned}
\Xw^\alpha_{x+e_1+e_2}&=B^\alpha_{x+e_2,\,x+e_1+e_2} \wedge B^\alpha_{x+e_1,\,x+e_1+e_2} \\
B^\alpha_{x,\,x+e_1}&=\Yw_x+(B^\alpha_{x+e_2,\,x+e_1+e_2} -  B^\alpha_{x+e_1,\,x+e_1+e_2})^+ \\
B^\alpha_{x,\,x+e_2}&=\Yw_x+(B^\alpha_{x+e_2,\,x+e_1+e_2} -  B^\alpha_{x+e_1,\,x+e_1+e_2})^-. 
\end{aligned} \ee
In contrast with the iteration \eqref{IJw5.1}--\eqref{IJw5.3}, the equations above proceed to the south and  west.

The first equation in \eqref{v:850}  is definition \eqref{v:X}. The second and third come from the limits \eqref{v:845}. For example, 
\begin{align*}
B^\alpha_{x,\,x+e_1} &=\lim_{n\to\infty} [ G_{x, v_n}-G_{x+e_1,\,v_n}] 
=\lim_{n\to\infty} [ \Yw_x + G_{x+e_1,\,v_n}\vee G_{x+e_2,\, v_n}-G_{x+e_1,\,v_n}] \\
&=\Yw_x + \lim_{n\to\infty} [G_{x+e_2,\, v_n}-G_{x+e_1,\,v_n}]^+  \\
&=\Yw_x + \lim_{n\to\infty} \bigl[ (G_{x+e_2,\, v_n}-G_{x+e_1+e_2,\, v_n}) - (G_{x+e_1,\,v_n}-G_{x+e_1+e_2,\, v_n}) \bigr]^+  \\
&=\Yw_x+[B^\alpha_{x+e_2,\,x+e_1+e_2} -  B^\alpha_{x+e_1,\,x+e_1+e_2}]^+. 
\end{align*}

By Lemma \ref{v:lmIJw} and induction,  for any down-right path from the upper left corner  $(M_0,N_1)$ to the lower right corner  $(M_1, N_0)$, inside the rectangle $\cD$, the   $B^\alpha$-increments on the path, the $\Yw$-weights below and to the left of the path, and the $\Xw^\alpha$ weights above and  to the right of the path,  are all independent with the correct marginal distributions stipulated in \eqref{w19}.      Part (a) of Definition \ref{v:d-exp-a} has been verified. 

It remains to check that equations \eqref{IJw5.1}--\eqref{IJw5.3} are satisfied. 
\begin{align*}
&B^\alpha_{x, x+e_1}\wedge B^\alpha_{x,x+e_2} = 
\lim_{n\to\infty} [ \Gpp_{x, v_n}-\Gpp_{x+e_1, v_n}] \wedge [ \Gpp_{x, v_n}-\Gpp_{x+e_2, v_n}]\\
&= \lim_{n\to\infty} [ \Yw_x+ \Gpp_{x+e_1, v_n}\vee\Gpp_{x+e_2, v_n} -\Gpp_{x+e_1, v_n}] \wedge [\Yw_x+ \Gpp_{x+e_1, v_n}\vee\Gpp_{x+e_2, v_n}-\Gpp_{x+e_2, v_n}]\\
&=\Yw_x + [\Gpp_{x+e_2, v_n} -\Gpp_{x+e_1, v_n}]^+ \wedge [\Gpp_{x+e_1, v_n} -\Gpp_{x+e_2, v_n}]^+ \\
&=\Yw_x. 
\end{align*}
This verifies \eqref{IJw5.1} (at $x$ instead of at  $x-e_1-e_2$). 

 By definition \eqref{v:X} and the additivity of $B^\alpha$, 
\begin{align*}
 B^\alpha_{x-e_1,\,x}&=  B^\alpha_{x-e_1,\,x}\wedge  B^\alpha_{x-e_2, \,x} +
(B^\alpha_{x-e_1, \,x}-B^\alpha_{x-e_2, \,x})^+\\
&= \Xw^\alpha_x+ (B^\alpha_{x-e_1-e_2, \, x-e_2}-B^\alpha_{x-e_1-e_2, \,x-e_1})^+ . 
\end{align*}
This is \eqref{IJw5.2}.  Equation \eqref{IJw5.3} is verified in a similar manner. 
 \end{proof}

 \begin{lemma}\label{v:lm-D}  Fix a countable dense  subset $D\subset(0,1)$. Then there exists an event $\Omega_0$ of full probability such that the following holds  for each  $\w\in\Omega_0$. 
 \begin{enumerate}
\item[{\rm(i)}]  For each    $\rho\in D$ the process $\{B^\rho_{x,y}(\w)\}_{x,y\in\Z^2}$ is well-defined by the limits in \eqref{v:845} for the specific sequence $v_n=\fl{n\xi(\rho)}$.     These processes satisfy the following inequalities: 
\be\label{v:843.1}  \begin{aligned} 
B^\lambda_{x,x+e_1}(\w) &\le  B^\rho_{x,x+e_1}(\w) \\
\text{and } \quad 
B^\lambda_{x,x+e_2}(\w) &\ge  B^\rho_{x,x+e_2}(\w) \qquad\text{for all $x\in\Z^2$ and $\lambda<\rho$ in $D$.} 
\end{aligned} \ee

\item[{\rm(ii)}]  For each    $\alpha\in(0,1)$,  and any sequence $u_n$ in $\Z^2$  such that $\abs{u_n}_1\to\infty$ and 
 \be\label{v:843.3}     \lim_{n\to\infty}  \frac{u_n}{\abs{u_n}_1} \, = \,  \xi(\alpha) \,=\, 
\biggl(  \frac{(1-\alpha)^2}{(1-\alpha)^2+\alpha^2} \,,   \frac{\alpha^2}{(1-\alpha)^2+\alpha^2}\biggr)  , 
 \ee
 we have these bounds: 
 \be\label{v:845.3}   \begin{aligned} 
 \sup_{\lambda\in D: \lambda<\alpha} B^\lambda_{x,x+e_1} &\le \varliminf_{n\to\infty} [ G_{x, u_n}-G_{x+e_1,u_n}]   \\
 &\le \varlimsup_{n\to\infty} [ G_{x, u_n}-G_{x+e_1,u_n}] 
 \le   \inf_{\rho\in D: \rho>\alpha} B^\rho_{x,x+e_1} \end{aligned}  \ee
 and 
 \be\label{v:845.5}   \begin{aligned} 
\sup_{\rho\in D: \rho>\alpha} B^\alpha_{x,x+e_2} &\le \varliminf_{n\to\infty} [ G_{x, u_n}-G_{x+e_2,u_n}]  \\
 &\le \varlimsup_{n\to\infty} [ G_{x, u_n}-G_{x+e_2,u_n}] 
 \le   \inf_{\lambda\in D: \lambda<\alpha} B^\alpha_{x,x+e_2} . \end{aligned}   \ee
 \end{enumerate} 
  \end{lemma}
  
\begin{proof}  Define $\Omega_0$ to be the event on which  the limits in \eqref{v:845}  hold for each $\rho\in D$,  for the specific sequence $v_n=\fl{n\xi(\rho)}$.    Then fix  $\w\in\Omega_0$.      Let $0<\lambda<\alpha<\rho<1$ be such that $\lambda, \rho\in D$.    Let $u_n$ be any sequence in $\Z^2$  such that $\abs{u_n}_1\to\infty$ and 
\eqref{v:843.3}   holds.    Let $b_n=\abs{u_n}_1$.      Then \eqref{v:846.4} implies that for large $n$,  
\be\label{850.8}  \begin{aligned}
\fl{b_n\xi(\rho)}\cdot e_1 < u_n\cdot e_1<   \fl{b_n\xi(\lambda)}\cdot e_1 
\quad\text{and}\quad
 \fl{b_n\xi(\lambda)}\cdot e_2< u_n\cdot e_2<   \fl{b_n\xi(\rho)}\cdot e_2 .
\end{aligned}\ee
Then Lemma \ref{lm-til} gives  the bounds 
 \begin{align*}
 G_{x, \fl{b_n\xi(\lambda)} }-G_{x+e_1, \fl{b_n\xi(\lambda)}}\le G_{x, u_n}-G_{x+e_1,u_n} 
 \le G_{x, \fl{b_n\xi(\rho)}}-G_{x+e_1, \fl{b_n\xi(\rho)}}
 \end{align*} 
 and 
  \begin{align*}
G_{x, \fl{b_n\xi(\rho)}}-G_{x+e_2, \fl{b_n\xi(\rho)}} \le G_{x, u_n}-G_{x+e_2,u_n} 
 \le   G_{x, \fl{b_n\xi(\lambda)} }-G_{x+e_2, \fl{b_n\xi(\lambda)}}.  
 \end{align*} 
 All the inequalities claimed  follow by taking $n\to\infty$.  
 \end{proof}

\medskip 

%
%
%
%
%

\begin{proof}[Proof of Theorem \ref{t:buse}]  We start  with a countable dense  subset $D\subset(0,1)$, the processes $B^\lambda$ for $\lambda\in D$ defined by the limits 
\be\label{860.01}  B^\lambda_{x,y}=  \lim_{n\to\infty}[ G_{x, \fl{n\xi(\lambda)} }-G_{y, \fl{n\xi(\lambda)}} \,] 
\ee
on  the event $\Omega_0$ of full probability given in Lemma \ref{v:lm-D}.    For each $\lambda\in D$ we have  the additivity 
\be\label{860.1}   B^\lambda_{x,y}+B^\lambda_{y,z}=B^\lambda_{x,z}
\ee
and  from Lemma \ref{v:lm-vn-a} we know that,  with 
 $ \Xw^\lambda_x= B^\lambda_{x-e_1,x}\wedge  B^\lambda_{x-e_2,x}$,   
   the process 
 $ \{  \Xw^\lambda_x, B^\lambda_{x-e_1,x}, B^\lambda_{x-e_2,x}, \Yw_x:  x\in\Z^2\}  $ 
 is an exponential-$\lambda$ last-passage system as described in Definition \ref{v:d-exp-a}.

  Let $\Omega_1$ be the subset of $\Omega_0$ on which 
\be\label{v:860}    \begin{aligned} 
 \sup_{\lambda\in D: \lambda<\gamma} B^\lambda_{x,x+e_1} &= B^\gamma_{x,x+e_1}  = \inf_{\rho\in D: \rho>\gamma} B^\rho_{x,x+e_1}   \\
\quad\text{and}\quad  
\sup_{\rho\in D: \rho>\gamma} B^\rho_{x,x+e_2} &= B^\gamma_{x,x+e_2}  =  \inf_{\lambda\in D: \lambda<\gamma} B^\lambda_{x,x+e_2}\qquad\text{for all $\gamma\in D$} . 
\end{aligned}   \ee
Event $\Omega_1$  has full probability because of the monotonicity and control of distributions: for example, for the first equality in \eqref{v:860} reason as follows:     by \eqref{v:843.1}
\[  \sup_{\lambda\in D: \lambda<\gamma} B^\lambda_{x,x+e_1} = \lim_{D\ni\lambda\nearrow\gamma}  B^\lambda_{x,x+e_1} \le  B^\gamma_{x,x+e_1},  \] 
but  by Lemma \ref{v:lm-vn-a}(b) and the convergence of distributions,  both $\lim_{D\ni\lambda\nearrow\gamma}  B^\lambda_{x,x+e_1}$ and  $B^\gamma_{x,x+e_1}$ have Exp$(1-\gamma)$ distribution.  Hence they agree almost surely.  From \eqref{v:860},    the monotonicity in \eqref{v:843.1}, and additivity to extend from nearest-neighbor increments, we get continuity of the process on $D$: 
\be\label{v:860.16}
\lim_{D\ni\lambda\,\to\,\gamma}  B^\lambda_{x,y} =  B^\gamma_{x,y}
\qquad \text{ for all $\gamma\in D$, on the event    $\Omega_1$.} 
\ee
  By discarding another zero probability event we can assume that $\Omega_1$ is invariant under translations.  

In order to prove that $B^\gamma$ is a stationary cocycle (Definition \ref{def:cK}) for $\gamma\in D$, it remains to check the stationarity $B^\gamma_{x,y}(\theta_z\w)=B^\gamma_{z+x,\,z+y}(\w)$.    We apply bounds \eqref{v:845.3}--\eqref{v:845.5} to the sequence $u_n=\fl{n\xi(\gamma)}+z$ that satisfies \eqref{v:843.3} with limit  $\xi(\gamma)$. Together with \eqref{v:860} these bounds give (with an extension by additivity) almost sure equalities 
\begin{align*}
B^\gamma_{x,y}=  \lim_{n\to\infty}[ G_{x, \,\fl{n\xi(\gamma)}+z }-G_{y,\, \fl{n\xi(\gamma)}+z} ] 
\end{align*}
for any fixed $z$ and all $x,y$.  Consequently 
\begin{align*}
B^\gamma_{x+z,y+z}&=  \lim_{n\to\infty}[ G_{x+z, \,\fl{n\xi(\gamma)}+z }-G_{y+z, \,\fl{n\xi(\gamma)}+z} ] 
=\lim_{n\to\infty}[ G_{x, \,\fl{n\xi(\gamma)} }-G_{y, \,\fl{n\xi(\gamma)}} ] \circ\theta_z  \\
&=B^\gamma_{x,y}\circ\theta_z.  
\end{align*}

We have now checked that $B^\gamma$ is a cocycle for each $\gamma\in D$. 

\medskip 

We take the step to general $\alpha\in(0,1)$.  
For each $\alpha\in(0,1)$ define  processes $\Xw^\alpha$ and  $B^\alpha$ by taking right limits from values in $D$:  set for each $\w\in\Omega_1$ and $x,y\in\Z^2$
\be\label{861}  (\Xw^\alpha_x (\w), B^\alpha_{x,y}(\w)) =\lim_{D\ni\lambda\searrow\alpha}  ( \Xw^\lambda_x(\w), B^\lambda_{x,y}(\w)). 
\ee
These limits exist for nearest-neighbor pairs $x,y$ by the monotonicity in \eqref{v:843.1}, and extend to all pairs $x,y$ by additivity on the right.   The limit for   $\Xw^\alpha_x (\w)$ comes along as a function by virtue of the definition  $\Xw^\lambda_x (\w) =   B^\lambda_{x-e_1,x}(\w) \wedge B^\lambda_{x-e_2,x}(\w)$ which is then also preserved to the limit. 
 
Extend these functions in some arbitrary way outside $\Omega_1$.  By the arguments given above, we have not altered these functions on $\Omega_1$  if $\alpha$ happens to lie in $D$. 
The properties of both Definitions \ref{def:cK} and \ref{v:d-exp-a} are preserved by the limits:  $B^\alpha$ is a stationary cocycle and  $ \{  \Xw^\alpha_x, B^\alpha_{x-e_1,x}, B^\alpha_{x-e_2,x}, \Yw_x:  x\in\Z^2\}  $ 
 is an exponential-$\alpha$ last-passage system.   We have verified part (i) of the theorem. 
 
 \medskip 
 
 Inequalities \eqref{v:853.1} are valid on the   event $\Omega_1$ simultaneously for all $\lambda, \rho\in D$ and preserved by the limit in \eqref{861}.  \eqref{v:860.16}, \eqref{861} and the monotonicity and additivity together give the cadlag paths $\lambda\mapsto B^\lambda_{x,y}(\w)$ for $\w\in\Omega_1$.   Part (ii) is proved. 
 
 \medskip 
 
 For part (iii), fix $0<\alpha<1$ and  let $\Omega^{(\alpha)}_2$ be the intersection of the event $\Omega_1$ above (which is contained in the event $\Omega_0$ of Lemma \ref{v:lm-D}) with the event on which 
 \be\label{864} 
 \sup_{\lambda\in D: \lambda<\alpha} B^\lambda_{x,x+e_1} =  \inf_{\rho\in D: \rho>\alpha} B^\rho_{x,x+e_1}  
\quad\text{ and }\quad 
\sup_{\rho\in D: \rho>\alpha} B^\alpha_{x,x+e_2} =  \inf_{\lambda\in D: \lambda<\alpha} B^\alpha_{x,x+e_2} . 
 \ee
The equalities above hold with probability $1$ by the argument used already above.  First, by monotonicity inequality  $\le$  holds in both  equalities above.   Then the  suprema and infima are limits, and the left- and right-hand sides of the equalities above are  equal in distribution.   Hence the left- and right-hand sides agree almost surely.  As a consequence we get limit \eqref{v:855.7}.

The coincidence of the lower and upper bounds in \eqref{v:845.3}--\eqref{v:845.5} imply that the claimed limit in \eqref{v:855}  holds for nearest-neighbor pairs $x,y$.   Extend to all $x,y$ by additivity. 

\medskip

This completes the proof of Theorem \ref{t:buse}. 
\end{proof}   


\begin{proof}[Proof of Theorem \ref{t:unique}]   As already observed, part (ii) of the theorem follows from part (i).    Assumptions (a) and (b) of part (i) of Theorem \ref{t:unique} enable us to  repeat the development of  Lemmas  \ref{v:lm86}--\ref{v:lm-89} for  the LPP processes constructed  as in \eqref{Gr26}--\eqref{Gr26} but this time with the inputs  $I_x=A_{x-e_1,x}$, $J_x=A_{x-e_2,x}$ and $\wc\sigma_x=\Yw_x$.   



The conclusion of part (i) comes from an application of Lemma \ref{v:lm-89}.  
 Given $\rho$, let $0<\lambda_1<\rho<\lambda_2<1$. Then the characteristic directions satisfy $\xi_1(\lambda_1)>\xi_1(\rho)>\xi_1(\lambda_2)$.  Let $\{v_n\}$ and $\{w_n\}$ be sequences in $a+\Z_{\ge 0}^2$  such that 
  $\abs{v_n}_1\wedge \abs{w_n}_1\to\infty$ and 
   \[     \lim_{n\to\infty}  \frac{v_n}{\abs{v_n}_1} = \xi(\lambda_2) 
   \quad\text{and}\quad   \lim_{n\to\infty}  \frac{w_n}{\abs{w_n}_1} = \xi(\lambda_1) .  \]
 The hypotheses of Lemma \ref{v:lm-89}  are satisfied.  Since   we have the Busemann limits   in \eqref{v:855} for the weights $\wc\sigma=\Yw$,  conclusions \eqref{v:832}--\eqref{v:833} give
 \[  B^{\lambda_1}_{a, \,a+e_1} \le A_{a,\,a+e_1} \le B^{\lambda_2}_{a, \, a+e_1}
 \quad\text{$\P$-a.s.} \]
 with the opposite inequalities for the edge $(a,a+e_2)$.  Let $\lambda_1, \lambda_2\to\rho$ and apply limit \eqref{v:855.7} to conclude that $A_{a,a+e_i} = B^{\rho}_{a, a+e_i}$ a.s.  Then from the properties of an exponential system,  
 $U_x=A_{x-e_1,x}\wedge A_{x-e_2,x}=B^\rho_{x-e_1,x}\wedge B^\rho_{x-e_2,x}=\Xw^\rho_x$. 
\end{proof}

\subsection{Midpoint problem} 

Let again $\pi^{x,y}_\dbullet$ denote the (almost surely unique)  maximizing path or geodesic  for $\Gpp_{x,y}$ defined by \eqref{v:GY}.   The {\it midpoint problem} asks whether the probability that the geodesic goes through some  particular point roughly midway between $x$ and $y$ converges to zero,  as the distance  between $x$ and $y$ grows.    We prove this in the next theorem for the case where $y-x$ has a limiting direction.   Since Theorem \ref{t:sh2} already showed that the geodesic stays within distance $o(\abs{x}_1)$ of the ray from $x$ to $y$, we consider   intermediate points in this narrow range. 

In undirected first-passage percolation the midpoint problem was solved under a differentiability assumption on the limit shape by \cite{damr-hans-17} and without any such unverifiable assumptions by \cite{ahlb-hoff}. 

%

\begin{theorem} \label{th:midpt}    Let $u_n\le z_n\le v_n$ be three  sequences on $\Z^2$ that satisfy the following conditions:  $u_n$ and $v_n$ can be random but   $z_n$ is not  {\rm(}that is, $u_n$ and $v_n$ can be measurable  functions of $\w$ but $z_n$ does not depend on $\w${\rm)},   $\abs{v_n-z_n}_1\to\infty$ and  $\abs{z_n-u_n}_1\to\infty$,  and 
\[ \lim_{n\to\infty} \frac{v_n-z_n}{\abs{v_n-z_n}_1} = \xi(\alpha) 
\quad\text{and}\quad 
\lim_{n\to\infty} \frac{z_n-u_n}{\abs{z_n-u_n}_1} = \xi(\alpha). \]
Then  $\ddd\lim_{n\to\infty} \P\{ z_n\in \pi^{u_n,v_n} \}=0$. 

\end{theorem} 

 \begin{proof} 
 Let $^o\Gpp_{x,y}$ denote the last-passage value without the first weight, that is, 
 	\be\label{v:GYo}  \nn 
	  \Gppo_{x,y}=\max_{x_{\brbullet}\,\in\,\Pi_{x,y}}\sum_{k=1}^{\abs{y-x}_1}\Yw_{x_k}
	\ee
where a path from $x$ to $y$ is given by $x_\brbullet=\{x=x_0,x_1,\dotsc, x_{\abs{y-x}_1}=y\}$.  

Let $R$ denote reflection of the lattice across the origin that acts on weights by  $(R\Yw)_x=\Yw_{-x}$.  Then define $\wt B^\alpha_{x,y}=B^\alpha_{-x,-y}\circ R$.  By Theorem \ref{t:buse},  there  exists an event $\Omega^{(\alpha)}_3$ of full probability on which  the following holds: for  any  sequence $a_n\in\Z^2$ (random or not)  such that $\abs{a_n}_1\to\infty$ and 
 \be\label{v:853.37}   \nn  \lim_{n\to\infty}  \frac{a_n}{\abs{a_n}_1} \, = \,-\xi(\alpha)\,=\, 
-\biggl(  \frac{(1-\alpha)^2}{(1-\alpha)^2+\alpha^2} \,,   \frac{\alpha^2}{(1-\alpha)^2+\alpha^2}\biggr) ,  
  \ee
we have the limits  
 \be\label{v:855.77}    \nn
 \wt B^\alpha_{x,y}  =\lim_{n\to\infty} ( G_{a_n, \,x}-G_{a_n, \,y} \,) \qquad\forall x,y\in\Z^2. \ee
In other words, $\wt B^\alpha$ is a Busemann function in the southwest direction.  For fixed $x,y$ the distributional relationship is $\wt B^\alpha_{x,y}\deq B^\alpha_{y,x}$.

The point $z_n$ lies on the geodesic $\pi^{u_n,v_n}$ if and only if for all $k\in\Z$, 
\be\label{43.99}\nn \begin{aligned}
\Gpp_{u_n, z_n}+\Gppo_{z_n,v_n}&\ge \Gpp_{u_n, z_n+(k,-k)}+\Gppo_{z_n+(k,-k),v_n}\\
\text{if and only if} \quad \Gpp_{u_n, z_n}-\Gpp_{u_n, z_n+(k,-k)} &\ge  \Gppo_{z_n+(k,-k),v_n}-\Gppo_{z_n,v_n}.
\end{aligned}
\ee
In order to replace $\Gppo$ with $\Gpp$ so that we can directly apply Theorem \ref{t:buse}, we shift the origin to $z_n$ and augment the probability space with an independent  collection of i.i.d.\ Exp(1) weights $\{\Yw'_{(k,-k)}\}_{k\in\Z}$.  Let $\Gpp'$ denote last-passage values that replace the original weights $\{\Yw_{(k,-k)}\}_{k\in\Z}$ on the antidiagonal through the origin with the new  independent weights $\{\Yw'_{(k,-k)}\}_{k\in\Z}$.    

In the second step below,  shifting the lattice index  by $z_n$ does not change the joint distribution of the weights.  This is the reason for the assumption that  $z_n$ not be random.    Fix a positive integer $m$.  
 \begin{align*}
&\P\{ z_n\in \pi^{u_n,v_n}\} \\
&\le \P\bigl\{ \Gpp_{u_n, z_n}-\Gpp_{u_n, z_n+(k,-k)} \ge  \Gppo_{z_n+(k,-k),v_n}-\Gppo_{z_n,v_n} \ \ \forall k\in[m] \bigr\}    \\
&= \P\bigl\{ \Gpp_{u_n-z_n,0}-\Gpp_{u_n- z_n, (k,-k)} \ge  \Gppo_{(k,-k),v_n-z_n}-\Gppo_{0,v_n-z_n} \ \ \forall k\in[m] \bigr\}    \\
&= \P\bigl\{ \Gpp_{u_n-z_n,0}-\Gpp_{u_n- z_n, (k,-k)} \ge  \Gpp'_{(k,-k),v_n-z_n}-\Gpp'_{0,v_n-z_n} 
 -\Yw'_{(k,-k)}+\Yw'_{0} \ \ \forall k\in[m] \bigr\}    \\
&=\P\biggl\{ \; \sum_{i=1}^k [\Gpp_{u_n-z_n,(i-1, -i+1)}-\Gpp_{u_n-z_n,(i,-i)}]  \\
&\qquad\qquad \qquad\qquad
\ge   \sum_{i=1}^k [ \Gpp'_{(i,-i),v_n-z_n}-\Gpp'_{(i-1, -i+1),v_n-z_n} ] -\Yw'_{(k,-k)}+\Yw'_{0}  \ \ \forall k\in[m]  \biggr\} \\
&\underset{n\to\infty}\longrightarrow \P\biggl\{ \; \sum_{i=1}^k \wt B^\rho_{(i-1, -i+1),(i,-i)} 
\ge   \sum_{i=1}^k  B^\rho_{(i,-i),(i-1, -i+1)}  -\Yw'_{(k,-k)}+\Yw'_{0}  \ \ \forall k\in[m]  \biggr\}\\
&= \P\bigl\{ S_k 
\ge     -\Yw'_{(k,-k)}+\Yw'_{0}  \ \ \forall k\in[m]  \bigr\}\\
&\le \P\bigl\{ S_k  \ge     -\Yw'_{(k,-k)}  \ \ \forall k\in[m]  \bigr\}. 
\end{align*}

After the limit  above, $B^\rho$ is the Busemann function for the weights $\{\Yw'_{(k,-k)}, \Yw_{(k,-k)+x}: k\in\Z, x\in\Z_{>0}^2\}$  (on and to the right of the antidiagonal), while  $\wt B^\rho$ is the southwest Busemann function for the weights $\{\Yw_{(k,-k)-x}: k\in\Z, x\in\Z_{\ge0}^2\}$.  These two Busemann functions are independent because they are functions of distinct collections of i.i.d.\ weights.   The convergence of the probability above follows from the almost sure  limits of the Busemann functions because the limiting random variables have continuous  distributions.

  Above we defined the symmetric random walk 
\be\label{Sk8} S_k= \sum_{i=1}^k [ \,\wt B^\rho_{(i-1, -i+1),(i,-i)} - B^\rho_{(i,-i),(i-1, -i+1)}\,] . \ee
Random walk  $S_\brbullet$ is not independent of the weights  $ \Yw'_{(k,-k)}$ on the antidiagonal. 
We continue to bound the last probability as follows. 
\begin{align*}
&\P\bigl\{ S_k  \ge     -\Yw'_{(k,-k)}  \ \ \forall k\in[m]  \bigr\}\\
&\le   \P\bigl\{ \;\max_{1\le k\le m} \Yw'_{(k,-k)}   \ge     m^{1/4}    \bigr\}  +  \P\bigl\{ S_k  \ge     - m^{1/4}  \ \ \forall k\in[m]  \bigr\}\\
&\le m e^{-m^{1/4}}     +  \P\biggl\{ \;\inf_{0\le s\le 1}  \frac{S_{\fl{sm}}}{\sqrt m}   \ge     -m^{-1/4}     \biggr\}. 
\end{align*}
The last line converges to zero as $m\to\infty$ because $\inf_{0\le s\le 1} m^{-1/2}S_{\fl{sm}}$ converges in distribution to   $\inf_{0\le s\le 1} \sigma B_s$ where $\sigma^2$ is the variance  of a term in \eqref{Sk8}  and $B_\bbullet$ is standard  Brownian motion. 
\end{proof}

%


 \medskip 
 
  \section{Fluctuation exponents} 
 \label{sec:exp} 
 
  Before specializing to the two-dimensional   corner growth model with exponential weights, we discuss the general setting in the context of the Kardar-Parisi-Zhang (KPZ)  class of models, in arbitrary dimension. 
 
 \subsection{Heuristics in the KPZ class} \label{s:exp-gen}  Stochastic models in the KPZ 
 class are expected to have common values of two {\it fluctuation exponents} $\loexp$ and $\trexp$  that depend only on the dimension $d\ge 2$.  In the context of the last-passage percolation process $\Gpp_{0, \fl{N\vvec}}$  on $\Z^d$ discussed in Section \ref{s:LPPgen},  in a fixed direction $\vvec\in\R_{>0}^d$,   these exponents  are defined informally as follows:
 \begin{enumerate}
\item[(a)]   the fluctuations of the last-passage value $\Gpp_{0,\fl{N\vvec}}$ are expected to have order of magnitude $N^\loexp$; 
\item[(b)]  the  optimal path from $0$ to $\fl{N\vvec}$ is  expected to fluctuate  on the scale  $N^\trexp$ around the straight line segment from $0$ to ${N\vvec}$. 
 \end{enumerate}
%

Certain properties are expected to hold for these exponents.  We state two below and give heuristic justification. 

The {\it KPZ scaling relation} states that 
\be\label{kpz4} \loexp=2\trexp-1. 
\ee
It is expected to hold in broad generality in all dimensions.   A version of it has been proved by Chatterjee \cite{chat-13-aom} for first-passage percolation in two dimensions.  The proof was simplified and generalized by Auffinger and Damron \cite{auff-damr-14}  who also extended the result to positive temperature polymer models \cite{auff-damr-13-alea}. 

 The direction $\loexp\ge2\trexp-1$ is the easier of the two inequalities.  We give a nonrigorous  argument for this inequality as a consequence of the limiting shape function $\gpp$ having quadratic curvature.   The idea is simple: if the path wanders too much, it picks up inadmissibly deviant values of $\gpp$, by the assumed curvature of $\gpp$.   
  
Consider  the setting of the $d$-dimensional CGM with admissible steps $\range=\{e_1, \dotsc, e_d\}$ and an i.i.d.\ environment $\w=(\w_x)_{x\in\Z^d}$ with $\E(\abs{\w_x}^{d+\e})<\infty$ for some $\e>0$.   Then the continuous, concave, homogeneous shape function 
  \be\label{lln39}
\gpp(\uvec)= \lim_{N\to\infty} N^{-1} \Gpp_{0,\fl{N\uvec}}  
\ee
exists as an almost sure limit for all
  $\uvec\in\R_{\ge0}^d$.    
 
   Assume that for some fixed $\vvec\in\R_{>0}^d$, 
 \be\label{chi8}  \Gpp_{0,\fl{N\vvec}}=N\gpp(\vvec)+ \Theta_P(N^\loexp) \ee
 where $Y_N=\Theta_P(N^\loexp)$ means  that  $N^{-\loexp}\,Y_N$   is tight but does not  converge to zero.  
Suppose that  the second derivative of $\gpp$   at $\vvec$ is strictly negative definite:   for any vector $\wvec$, 
\[    D^2_\wvec\,\gpp(\vvec)=\frac{\partial^2}{\partial\theta^2}  \gpp(\vvec+\theta\wvec) \Big\vert_{\theta=0}= \wvec^T D^2\gpp(\vvec)\,\wvec<0. \]
 (The shape function is concave so a nonvanishing second derivative is negative.) 

If maximizing paths are to fluctuate on the scale $N^\trexp$, then let us assume that  the maximizing path for $\Gpp_{0,\,2N\vvec}$ goes through  the point $N\vvec+ N^\trexp \wvec_N$ for some nondegenerate, tight random vector $\wvec_N$.  (We drop integer parts to lighten up the notation.  In any case the argument presented below is not rigorous.)    Then 
\begin{align*}
&2N\gpp(\vvec) + \Theta_P(N^\loexp)  =   \Gpp_{0,\,2N\vvec}   = \Gpp_{0,\,N\vvec+ N^\trexp \wvec_N} +   \Gpp_{N\vvec+ N^\trexp \wvec_N,\,2N\vvec} \\[3pt]
&\qquad \qquad = N\gpp(\vvec+ N^{\trexp-1} \wvec_N)+ N\gpp(\vvec- N^{\trexp-1} \wvec_N) +  \Theta_P(N^\loexp) \\[3pt]
&\qquad \qquad \approx 2N\gpp(\vvec) +  N^{2\trexp-1}  D^2_{\wvec_N}\,\gpp(\vvec) +  \Theta_P(N^\loexp). 
\end{align*}
Cancelling the constant $2N\gpp(\vvec)$ leaves 
\[   \Theta_P(N^\loexp) \approx    N^{2\trexp-1}  D^2_{\wvec_N}\,\gpp(\vvec) +  \Theta_P(N^\loexp).\]
This forces  $2\trexp-1\le \loexp$.  

\medskip 

In two dimensions it is also expected that 
\be\label{kpz8} \loexp=\trexp/2. 
\ee
We give a heuristic justification for this relation under some natural assumptions.   Assume we have a Busemann function in   direction $\uvec$:
\be\label{heur78}   B^\uvec_{x,y}=\lim_{M\to\infty}  \bigl(  \Gpp_{x, M\uvec} - \Gpp_{y, M\uvec}\bigr). \ee
Assume that 
 in some direction $\vvec$, 
\be\label{heur79}    B^\uvec_{0, m\vvec} =  \nabla \gpp(\uvec)\cdot m\vvec  + \Theta_P( m^{1/2}) . \ee
This says that in direction $\vvec$,  $B^\uvec$ has fluctuation exponent  $\tfrac12$, in other words, diffusive fluctuations.   (These assumptions are satisfied in the exponential CGM, for example for $\vvec=e_1$ or $e_2$,  as verified by Theorem \ref{t:buse} and \eqref{EB8}.) 

 A natural {\it planar} assumption is that geodesics (maximizing paths) from different points in a particular direction eventually coalesce. Since maximizing paths fluctuate on the scale $N^\trexp$, it is reasonable to assume that if the initial points $x,y$  in \eqref{heur78} are separated by a distance of order $N^\trexp$, then  once $M$ is much larger than $N$,  the maximizing paths from $x$ and $y$ to $M\uvec$ have merged at $a_NN\uvec$ for some (tight random) constant $a_N$.  Then the limit value has been reached:
 \[   B^\uvec_{0, N^\trexp\vvec} = \Gpp_{0, \,a_NN\uvec}- \Gpp_{N^\trexp\vvec, \,a_NN\uvec}.\]
 Apply  \eqref{chi8} and \eqref{heur79} to the two sides above : 
 \begin{align*}   \nabla \gpp(\uvec)\cdot N^\trexp\vvec  + \Theta_P( N^{\trexp/2})
 &= N \gpp(a_N\uvec)- N \gpp(a_N\uvec-N^{\trexp-1}\vvec)  + \Theta_P( N^{\loexp})  \\
 &\approx  \nabla\gpp(a_N\uvec)\cdot  N^{\trexp}\vvec + \Theta_P( N^{\loexp}). 
 \end{align*}
 By homogeneity $\nabla \gpp(\uvec)=\nabla\gpp(a_N\uvec)$ and hence the above gives $\Theta_P( N^{\trexp/2}) \approx  \Theta_P( N^{\loexp})$.   This gives a heuristic justification of \eqref{kpz8}.

 \subsection{Exponents for the  stationary exponential corner growth model} 
 
 We derive the KPZ exponents $\loexp=\tfrac13$ and $\trexp=\tfrac23$ for the  stationary exponential CGM in 1+1 dimensions.  
 Recall the setting from Section \ref{v:s-stat-cgm}. 
The  parameter $0<\rho<1$ of the boundary weights is fixed.  We are   given  mutually independent random variables  
\be\label{e:IJw}  \{   \w_{x} ,\, I_{ie_1}, \, J_{je_2}  :  \,x\in \Z_{>0}^2, \, i,j\in \Z_{>0}\}  \ee
    with  marginal distributions 
\be\label{e:w7}\begin{aligned}
\w_{x}\sim \text{ Exp}(1),  \quad I_{ie_1}\sim \text{ Exp}(1-\rho)  , \quad \text{ and }  \quad  
J_{je_2}\sim \text{ Exp}(\rho)  .  
\end{aligned}\ee 
The last-passage process $G^\rho_{0,x}$  is defined for $x\ge 0$  by    $G^\rho_{0,0}=0$, 
  \be\label{e:Gr1} G^\rho_{0,\,me_1}=\sum_{i=1}^m I_{ie_1} 
  \quad\text{and}\quad
G^\rho_{0,\,ne_2}= \sum_{j=1}^n  J_{je_2} ,   \ee 
and then 
  for $(m,n)\in \Z_{>0}^2$, 
\be\label{e:Gr2}
G^\rho_{0,\,(m,n)}= \max_{1\le k\le m} \;  \Bigl\{  \;\sum_{i=1}^k I_{ie_1}  + G_{ke_1+e_2, \,(m,n)} \Bigr\}  
\bigvee
 \max_{1\le \ell\le n}\; \Bigl\{  \;\sum_{j=1}^\ell  J_{je_2}  + G_{\ell e_2+e_1, \,(m,n)} \Bigr\} .  
\ee
 For $a\in\Z_{>0}^2$ the quantity  $G_{a,(m,n)}$ inside the braces is the last-passage value defined in  \eqref{v:G} for i.i.d.\ Exp$(1)$ weights $\w$.

 We do not have a  closed form expression for $\Vvv[G^\rho_{0,(m,n)}]$ but we can access it well enough to show that it obeys the fluctuation exponent $1/3$ characteristic of the KPZ class.   However, there in an extra twist.   Notice in \eqref{e:w7} that the boundary weights  $\w_{ie_1}$ and $\w_{je_2}$ are larger on average than the bulk weights $\{\w_x\}_{x\in\Z_{>0}^2}$.  This implies that the boundaries are attractive to the maximizing path.  It turns out that only when we take the point $x$ to infinity in the  {characteristic direction} $((1-\rho)^2, \rho^2)$,   the pull of the boundaries balances out and  $G^\rho_{0,x}$ obeys KPZ fluctuations.  Otherwise the boundaries swamp the effects of the percolation and $G^\rho_{0,x}$ obeys the classical central limit theorem.  

Let $N$ be a scaling parameter that increases to $\infty$.  We consider the point-to-point last-passage percolation from $0$ to a point $(m,n)=(m(N), n(N))$ that is  taken to infinity as $N\to\infty$.     Let  $\kappa_N$ denote  the deviation of $(m,n)$ from the characteristic direction:  
\be\label{ch5} 
\kappa_N=\abs{\,m-N(1-\rho)^2\,}\vee\abs{\,n-N\rho^2\,}   .  
\ee

\begin{theorem}\label{thGr2}     Assume weight distributions \eqref{e:w7}   and let $a_0$ be a fixed positive constant.     Then there exists a constant  $0<C=C(\rho, a_0)<\infty$ such that 
\be\label{e:Gr11}   C^{-1}N^{2/3}\le  \Vvv[G^\rho_{0,(m,n)}]\le CN^{2/3}
\ee
holds for all  $(m,n)\in\Z_{>0}^2$ and $N\ge 1$ for which $\kappa_N$ in \eqref{ch5} satisfies 
 $\kappa_N\le a_0N^{2/3}$.  
\end{theorem} 

The proof gives the following dependence  of the bounds on $a_0$: 
 \be\label{e:Gr11.5}   c_1(\rho)  e^{-C_2(\rho)a_0^3} N^{2/3}\le  \Vvv[G^\rho_{0,(m,n)}]\le C_3(\rho)(1+a_0) N^{2/3}
 \ee
 where $c_1(\rho)$, $C_2(\rho)$ and $C_3(\rho)$ are functions of $\rho$ alone. 
 
As a fairly immediate corollary we obtain the behavior in off-characteristic  directions.   For concreteness, we state the result for the case where the horizontal direction is abnormally large.  

\begin{corollary}  \label{cGr}   Assume weight distributions \eqref{e:w7}.    Suppose $m,n\to\infty$.  Define parameter 
$N$ by $n=N\rho^2$, and assume that 
\[  N^{-\alpha}\bigl(m-N(1-\rho)^2\bigr) \to c_1 >0
\quad\text{as $m,n\to\infty$}   
\]
for some $\alpha>2/3$.  Then as $m,n\to\infty$, 
\[ N^{-\alpha/2} \bigl\{ G^\rho_{0,(m,n)} 
-\E [G^\rho_{0,(m,n)}]  \bigr\}\] converges in distribution
to a centered normal distribution  with variance 
$c_1(1-\rho)^{-2}$. 
\end{corollary}

\begin{proof}   Recall that overline means centering of a random variable. 
\begin{align*}
N^{-\alpha/2}  \,\overline G^\rho_{0,(m,n)}  = N^{-\alpha/2} \, \overline G^\rho_{0,(\fl{N(1-\rho)^2}  ,\,\fl{N\rho^2})  }  
\ + \ N^{-\alpha/2} \!\!\!\!\!\! \sum_{i=\fl{N(1-\rho)^2}+1}^m \bar\Iincr_{(i,n)}
\end{align*}
The mean square of the first term on the right is of order $N^{-\alpha}\cdot N^{2/3}$ and hence in the limit vanishes in $L^2$ and in probability.  The second term is a sum of approximately $c_1N^\alpha$  mean zero i.i.d.\ terms with variance $\E[\,\bar\Iincr^2_{(i,n)}]=(1-\rho)^{-2}$.  This sum gives the normal limit, by the standard central limit theorem. 
\end{proof}

As discussed in Section \ref{s:sh-g},  
for a given endpoint $(m,n)$ the last-passage problem \eqref{e:Gr2}  has  an almost surely  unique  maximizing path or geodesic  $\pi^{0,(m,n)}_\dbullet=(\pi^{0,(m,n)}_k)_{k=0}^{m+n}$   from $\pi^{0,(m,n)}_0=0$ to $\pi^{0,(m,n)}_{m+n}=(m,n)$ that satisfies  $G^\rho_{0,(m,n)}= \sum_{k=1}^{m+n} \w_{\pi^{0,(m,n)}_k}$, where we utilized the notational convention on the axes  that $\w_{ke_1}=I_{ke_1}$ and $\w_{\ell e_2}=J_{\ell e_2}$. 

  The theorem below quantifies the fluctuations of the maximizing path $\pi^{0,(m,n)}_\dbullet$ around a fixed point  on the  straight line from $0$ to $(m,n)$ when $(m,n)$ points in the characteristic direction of the parameter $\rho$.  As before, the endpoint $(m,n)$    is of order $N$.   Qualitatively speaking,  the theorem says that $\pi^{0,(m,n)}_\dbullet$ hits  a square centered on the diagonal  with high probability if  the square is large enough on the scale $N^{2/3}$, and $\pi^{0,(m,n)}_\dbullet$  misses the square  with at least a fixed positive  probability  if the square  is small enough  on the scale $N^{2/3}$.
  
 Let 
\[  A_r(x)=\{ y\in\R^2:  x_i\le y_i\le x_i+ r\text{ for }i=1,2\} \]
 denote the closed square    with side length $r$ and lower left corner at $x=(x_1,x_2)\in\R^2$.      $\P^\rho_{0,(m,n)}$ denotes the probability measure of the increment-stationary point-to-point  LPP process in the rectangle $[0,(m,n)]$. 

\begin{theorem}\label{thGr2.5}   
Fix $0<\rho<1$.  Assume weight distributions \eqref{e:w7}   and   let 
\[  (m,n)=(\fl{N(1-\rho)^2}, \fl{N\rho^2}). \]   
   
   {\rm (a)}    
     For each $0<\delta<1$  we have this upper bound, with a constant $C=C(\rho)$: 
\be\label{e:Gr13}    
\sup_{0\le t\le 1-\delta} \P^\rho_{0,(m,n)}\{ \text{$\pi^{0,(m,n)}_\dbullet$ does not  intersect $A_{rN^{2/3}}(tm,tn)$} \} \le \frac{C}{r^{3}} 
\quad \text{for $r\ge 1$ and $N\ge \frac1{\delta}$.}  
\ee

{\rm (b)}   Fix $0<t<1$.    There exist  $\delta, \e>0$ that depend on $\rho$  and $N_0=N_0(\rho, t)<\infty$ such that 
\be\label{e:Gr14}    
\P^\rho_{0,(m,n)}\{ \text{$\pi^{0,(m,n)}_\dbullet$ does not  intersect $A_{\delta N^{2/3}}(tm,tn)$} \} \ge \e  
\qquad \text{for $N\ge N_0$.}  
\ee

\end{theorem}

Estimate \eqref{e:Gr13} gives the upper bound $\trexp\le\tfrac23$ and  \eqref{e:Gr14} the lower bound $\trexp\ge\tfrac23$ for the path exponent $\trexp$. 
Since $\pi^{0,(m,n)}_0=0$,  for   $t=0$ the upper bound \eqref{e:Gr13} holds trivially while the lower bound \eqref{e:Gr14} fails trivially.  




\begin{remark}[Exponents for LPP with i.i.d.\ exponential weights]  After these theorems, the natural next question concerns the exponents for the last-passage value and maximizing path of the LPP process  $G_{0,y}$ of \eqref{v:G} with i.i.d.\ Exp(1) weights $\w$.  Presently the coupling approach of these notes gives the exponent $\loexp=\tfrac13$ and the upper bound $\trexp\le\tfrac23$ for the i.i.d.\ exponential case. Results to this effect can be found in \cite{bala-cato-sepp} for the exponential CGM, and in   \cite{more-sepp-valk-14, sepp-12-aop-corr} for positive temperature polymer models.  
\hfill$\triangle$\end{remark}

\begin{remark}[Fluctuation exponent for the Busemann function]  Through the distributional equality \eqref{GB45},  Theorem \ref{thGr2} gives the following corollary for the fluctuations of $B^\alpha$ in the characteristic direction $\xi(\alpha)$:  for each $0<\alpha<1$ there exists a finite constant $C=C(\alpha)$ such that 
\be\label{e:B11}   C^{-1}N^{2/3}\le  \Vvv[B^\alpha_{0, \,\fl{N\xi(\alpha)}}\,]\le CN^{2/3}
\ee
 for $N$ large enough so that $\fl{N\xi(\alpha)}\ne 0$. 
\hfill$\triangle$\end{remark}

We turn to proving  Theorems \ref{thGr2} and \ref{thGr2.5}.   The remainder of this  section is divided into four parts: a variance formula, an upper bound proof for Theorem \ref{thGr2},  a lower bound proof for Theorem \ref{thGr2}, and the proof of  Theorem \ref{thGr2.5}.  

\subsection{Variance formula}  
The first step towards the proof of Theorem \ref{thGr2} is an explicit  formula that ties together $\Vvv[G^\rho_{0,(m,n)}]$ and the amount of weight the maximizing path $\pi^{0,(m,n)}_\dbullet$  collects on the boundary.  
For $r=1,2$, define the exit time (or exit point) of the maximizing  path from the $e_r$ axis by 
\be\label{d-tt}  \exitt_{r}= \max\{ k\ge 0:  \pi^{0,(m,n)}_k\cdot e_{3-r}=0\}, \qquad r=1,2.  \ee 
If the first step of  the path $\pi^{0,(m,n)}_\dbullet$ from the origin is $e_r$, then $1\le \exitt_r\le (m,n)\cdot e_r$ and $\exitt_{3-r}=0$.  In other words,  almost surely  exactly one of $\exitt_1$ and $\exitt_2$ is positive (but which one is positive varies with the realization of the weights $\w$).  

Further, introduce the sums of weights along the axes:  
\[  
S_{1,k}=\sum_{i=1}^{k} I_{ie_1} \quad\text{and}\quad 
S_{2,\ell}=\sum_{j=1}^{\ell} J_{je_2}. 
 \]
Then  $S_{r, \exitt_r}$ is   the amount of weight that the maximizing path collects on the $e_r$-axis.  Again, for each weight configuration  $\w$,  exactly one of $S_{1, \exitt_1}$ and $S_{2, \exitt_2}$ is positive and the  other one zero. 

\medskip 

{\it Notational comment.}   The joint distribution of the random variables associated with the  increment-stationary point-to-point LPP process is determined by these parameters: the  planar rectangle $[0,(m,n)]$ and the parameter $\rho\in(0,1)$ of the boundary edge weights in \eqref{e:IJw}--\eqref{e:w7}.   The probability measure for this situation will be denoted by $\P^\rho_{0,(m,n)}$.   When processes with different values of $\rho$ are coupled in the same rectangle,  the  variables themselves  will  be adorned with superscripts, as in   $\exitt^\rho_r$ and $S^\rho_{r,k}$.   When the parameters $\rho$ or $(m,n)$ are understood from  the context or  already present in the notation for the random variables, as for example  in $G^\rho_{0,(m,n)}$ in \eqref{e:Gr11}, we omit the parameters from $\P$.  
\hfill$\triangle$
   
 \medskip 

Next we state the variance formula for the last-passage value in the increment-stationary CGM. Formulas of this type for exactly solvable models go back to the  work of Ferrari and Fontes \cite{ferr-font-94a} on the asymmetric simple exclusion process.  In the particle system context the formulas relate the variance of the current to a expectation of a second-class particle.  

\begin{theorem} \label{thGr3}  Assume weight distributions \eqref{e:w7}.    Then for $(m,n)\in\Z_{>0}^2$, 
\begin{align}\label{VGr}
\Vvv[G^\rho_{0,(m,n)}] &=  -\,\frac{m}{(1-\rho)^2}+\frac{n}{\rho^2} +\frac{2}{1-\rho} \, \E^\rho_{0,(m,n)}[S_{1,\exitt_1}] \\
\label{VGr0.1}
&=  \frac{m}{(1-\rho)^2}-\frac{n}{\rho^2} +\frac{2}{\rho} \, \E^\rho_{0,(m,n)}[S_{2,\exitt_2}].
\end{align} 
\end{theorem}

\begin{proof}   We prove  \eqref{VGr}. The second line \eqref{VGr0.1} follows for example by transposition.   
Begin with a covariance manipulation.   For this proof , abbreviate the $G^\rho$-increments on the sides of the rectangle $[0,(m,n)]$ as follows: 
\begin{align*}
\north =G^\rho_{0,(m,n)}-G^\rho_{0,(0,n)},\quad \south =G^\rho_{0,(m,0)},
\quad \east =G^\rho_{0,(m,n)}-G^\rho_{0,(m,0)},\quad \west =G^\rho_{0,(0,n)}. 
\end{align*}
Then 
\be\begin{aligned}
\Vvv\bigl[G^\rho_{0,(m,n)}\bigr]&=\Vvv(\west +\north )=\Vvv(\west )+\Vvv(\north )+2\,\Cvv(\west ,\north )\\
&=\Vvv(\west )+\Vvv(\north )+2\,\Cvv(\south +\east -\north ,\north )\\
 &= \Vvv(\west )-\Vvv(\north )+2\,\Cvv(\south ,\north )\\
 &=  \frac{n}{\rho^2}  -\frac{m}{(1-\rho)^2}+2\,\Cvv(\south ,\north ).  
\end{aligned}\label{aux3}\ee
The two last equalities came from Theorem \ref{v:tIJw1}(i) applied to the down-right path that gives the north and east boundaries of the rectangle $[0,(m,n)]$: first   the independence of $\east $ and $\north $ and then the marginal distributions of $\west$ and $\north$ as sums of i.i.d.\ exponentials.  

To prove \eqref{VGr} it remains to compute  $\Cvv(\south ,\north )$. In order  to differentiate with respect to the parameter   of the weights
$I_{(i,0)}$  on
the $x$-axis (term $\south $)  while fixing everything else, consider   a system
with two independent parameters $\alpha$ and  $\rho$  on the boundaries   and with initial weight 
distributions  (for $i,j\in\Z_{>0}$)  
\be\label{e:w7.4}\begin{aligned}
\w_{(i,j)}\sim \text{ Exp}(1),  \quad I_{ie_1}\sim \text{ Exp}(\alpha)  , \quad \text{ and }  \quad  
J_{je_2}\sim \text{ Exp}(\rho)  .  
\end{aligned}\ee 
Use the superscript $\alpha,\rho$ on $\Gpp^{\alpha,\rho}_{0,(m,n)}$,  $\P^{\alpha,\rho}$, and $\E^{\alpha,\rho}$ for the LPP process with these weights.
If $\alpha\ne1-\rho$ this is not a stationary LPP process,  so we do not know the distribution of $\north$.  We can still derive  the identity 
\be  \Cvv^{\alpha,\rho}(\south ,\north ) =-\,\frac{\partial}{\partial\alpha} \E^{\alpha,\rho}(\north ).  \label{cov3}\ee

The sum  $  \south =\sum_{i=1}^m I_{ie_1} $  has the gamma density
 \[  f_{m,\alpha}(s) =\frac{\alpha^m}{(m-1)!} s^{m-1}e^{-\alpha s}.  \]    By changing variables in the integrals,  with test functions  $\varphi$ and $\psi$, 
\begin{align*}
&\E^{\alpha,\rho} [\,\varphi(I_{e_1},\dotsc, I_{me_1}) \,\psi(\south)\,] 
=\int\limits_{\R_{>0}^m}   \varphi(x_1,\dotsc,x_m) \,\psi(\,\textstyle\sum_1^m x_i) \,\alpha^m \, e^{-\alpha\textstyle\sum_1^m x_i} \,dx_{1,m} \\
&
= \int_0^\infty  ds\; \tfrac{\alpha^m}{(m-1)!} s^{m-1}e^{-\alpha s} \, \psi(s) 
\int\limits_{\R_{>0}^{m-1}}   \ind_{\bigl\{{\textstyle\sum}_1^{m-1} x_i<s\bigr\}}  \,\varphi\bigl(x_1,\dotsc,x_{m-1}, s-{\textstyle\sum}_1^{m-1} x_i\bigr) \,  \, \tfrac{(m-1)!} {s^{m-1}}\,dx_{1,m-1}
\end{align*}
we see that the conditional distribution of $(I_{e_1},\dotsc, I_{me_1})$, given $\south$, is independent of the parameter $\alpha$.  

 We verify \eqref{cov3}. 
\begin{align*}
-\,\frac{\partial}{\partial\alpha} \E^{\alpha,\rho}(\north ) 
&=   -\,\frac{\partial}{\partial\alpha}  \int_0^\infty  \E^{\alpha,\rho}(\north \,\vert\, \south=s) \, f_{m,\alpha}(s) \,ds \\
&=\int_0^\infty  \E^{\alpha,\rho}(\north \, \vert \, \south=s) \, \Bigl(- \frac{\partial}{\partial\alpha} f_{m,\alpha}(s) \Bigr) \,ds \\
&=\int_0^\infty  \E^{\alpha,\rho}(\north \, \vert \, \south=s) \,  \Bigl( s-\frac{m}\alpha    \Bigr)    f_{m,\alpha}(s) \,ds   \\
&= \E^{\alpha,\rho}[\north\south]   - \E^{\alpha,\rho}[\north]  \cdot  \frac{m}\alpha
= \Cvv^{\alpha,\rho}(\south ,\north ) . 
\end{align*} 

On the other hand, we calculate the derivative in \eqref{cov3} by putting the $\alpha$-dependence directly on the weights, by realizing them as functions of  uniform random variables.   Let $U_1,\dotsc, U_m$ be i.i.d.\ Unif$(0,1)$ random variables.   The inverse function of the Exp$(\alpha)$ distribution function $F_\alpha(x)=1-e^{-\alpha x}$ is $F_\alpha^{-1}(u)=-\alpha^{-1}\log(1-u)$ and this satisfies $\frac{\partial}{\partial\alpha}F_\alpha^{-1}(u)=-\alpha^{-1} F_\alpha^{-1}(u)$.   For the Exp$(\alpha)$ weights on the $x$-axis take $I_{ie_1}=F_\alpha^{-1}(U_i)$. 
Let  $\wt\E$ denote  the expectation over $U_1,\dotsc,U_m$ and the other weights $J_{je_2}$ and $\w_{(i,j)}$ from \eqref{e:w7.4}.  $\wt\E$  does not depend on $\alpha$. $\exitt^\alpha_1$ is the exit point of the maximizing path from the $x$-axis when the edge weights on the $x$-axis are $I_{ie_1}=F_\alpha^{-1}(U_i)$.  
\be\label{cov78} \begin{aligned}
\frac{\partial}{\partial\alpha} \E^{\alpha,\rho}(\north ) 
&=\frac{\partial}{\partial\alpha} \E^{\alpha,\rho}[ G_{0,(m,n)}-G_{0,(0,n)} ]
=\frac{\partial}{\partial\alpha} \E^{\alpha,\rho}[ G_{0,(m,n)}] \\
&  =\frac{\partial}{\partial\alpha} \wt\E\Bigl[  \,\sum_{i=1}^{\exitt^\alpha_1} F_\alpha^{-1}(U_i) + \Gpp_{(\exitt^\alpha_1,1), (m,n)}  \Bigr]  
=  \wt\E\Bigl[  \, \sum_{i=1}^{\exitt^\alpha_1}  \frac{\partial}{\partial\alpha} F_\alpha^{-1}(U_i)
  \Bigr]   \\
&   =  -\alpha^{-1} \wt\E\Bigl[  \,\sum_{i=1}^{\exitt^\alpha_1} F_\alpha^{-1}(U_i)   \Bigr]    = -\alpha^{-1}  \E^{\alpha,\rho}[ S_{1,\exitt_1}].
\end{aligned}\ee 
An intuitive   justification of the  fourth equality above is that, with probability one,  if $\alpha'$ is close enough to $\alpha$, then $\exitt^{\alpha'}_1=\exitt^\alpha_1$. 
A detailed   justification of \eqref{cov78} follows.   

Let $\e>0$ be small.    Since $F_{\alpha-\e}^{-1}(U_i)>F_{\alpha}^{-1}(U_i)$, replacing $\alpha$ with $\alpha-\e$ increases weights on the $x$-axis without altering other weights.  Consequently $\exitt^{\alpha-\e}_1\ge \exitt^{\alpha}_1$ and $\Gpp^{\alpha-\e,\rho}_{0,(m,n)}\ge \Gpp^{\alpha,\rho}_{0,(m,n)}$.   Now consider the differences between the expectations above at $\alpha-\e$ and $\alpha$.  
\begin{align}
\nn &\wt\E\Bigl[  \,\sum_{i=1}^{\exitt^{\alpha-\e}_1} F_{\alpha-\e}^{-1}(U_i) + \Gpp_{(\exitt^{\alpha-\e}_1,1), (m,n)}  \Bigr]
-
\wt\E\Bigl[  \,\sum_{i=1}^{\exitt^\alpha_1} F_\alpha^{-1}(U_i) + \Gpp_{(\exitt^\alpha_1,1), (m,n)}   \Bigr] \\
\label{line4} &\qquad =  \wt\E\Bigl[  \,\sum_{i=1}^{\exitt^\alpha_1} \bigl(   F_{\alpha-\e}^{-1}(U_i)- F_\alpha^{-1}(U_i)\bigr)   \Bigr] \\
\label{line5}  &\qquad \qquad  + \wt\E\Bigl[  \,\sum_{i=\exitt^\alpha_1+1}^{\exitt^{\alpha-\e}_1} F_{\alpha-\e}^{-1}(U_i) +   \Gpp_{(\exitt^{\alpha-\e}_1,1), (m,n)}  - \Gpp_{(\exitt^\alpha_1,1), (m,n)} \Bigr]. 
\end{align}
The  expectation on line \eqref{line4} equals 
\[   \frac{\e}{\alpha-\e}\,\wt\E\Bigl[  \,\sum_{i=1}^{\exitt^\alpha_1} F_\alpha^{-1}(U_i)   \Bigr]   
= \bigl(\e+O(\e^2)\bigr)  \alpha^{-1}  \E^{\alpha,\rho}[ S_{1,\exitt_1}].  \]
We show that the  expectation on line \eqref{line5} is of order $o(\e)$.   The integrand inside the expectation  vanishes if $\exitt^{\alpha-\e}_1=\exitt^{\alpha}_1$. Hence we may multiply the integrand by the  indicator of the complementary event and then  bound it  as follows:
\begin{align*}
&\biggl(\; \sum_{i=\exitt^\alpha_1+1}^{\exitt^{\alpha-\e}_1} F_{\alpha-\e}^{-1}(U_i) +   \Gpp_{(\exitt^{\alpha-\e}_1,1), (m,n)}  - \Gpp_{(\exitt^\alpha_1,1), (m,n)}\biggr)
\cdot \ind_{\{\exitt^{\alpha-\e}_1\,>\,\exitt^{\alpha}_1\}} \\
&\le 
\biggl(\; \sum_{i=1}^{\exitt^{\alpha-\e}_1} F_{\alpha-\e}^{-1}(U_i) +   \Gpp_{(\exitt^{\alpha-\e}_1,1), (m,n)}  - \sum_{i=1}^{\exitt^\alpha_1} F_{\alpha}^{-1}(U_i) -  \Gpp_{(\exitt^\alpha_1,1), (m,n)}\biggr)
\cdot \ind_{\{\exitt^{\alpha-\e}_1\,>\,\exitt^{\alpha}_1\}} \\
&= 
\bigl(  \Gpp^{\alpha-\e,\rho}_{0,(m,n)}  - \Gpp^{\alpha,\rho}_{0,(m,n)} \bigr)
\cdot \ind_{\{\exitt^{\alpha-\e}_1\,>\,\exitt^{\alpha}_1\}} 
\; \le \;  \bigl(  S^{\alpha-\e}_{1,m}  - S^{\alpha}_{1,m}  \bigr)
\cdot \ind_{\{\exitt^{\alpha-\e}_1\,>\,\exitt^{\alpha}_1\}}.  
\end{align*} 
The last inequality comes from the observation that, since the weights away from the south boundary do not change as $\alpha$ changes,  the   increase $\Gpp^{\alpha-\e,\rho}_{0,(m,n)}  - \Gpp^{\alpha,\rho}_{0,(m,n)}$ is  bounded by the total increase $S^{\alpha-\e}_{1,m}  - S^{\alpha}_{1,m}$ in the weight on the south boundary.   We bound the  expectation on line \eqref{line5}:
\begin{align*}
& \wt\E\Bigl[  \,\sum_{i=\exitt^\alpha_1+1}^{\exitt^{\alpha-\e}_1} F_{\alpha-\e}^{-1}(U_i) +   \Gpp_{(\exitt^{\alpha-\e}_1,1), (m,n)}  - \Gpp_{(\exitt^\alpha_1,1), (m,n)} \Bigr]\\[4pt]
&\qquad \le   \wt\E\bigl[  S^{\alpha-\e}_{1,m}  - S^{\alpha}_{1,m} \,,\, 
 \exitt^{\alpha-\e}_1\,>\,\exitt^{\alpha}_1 \,\bigr]  
\le  \bigl\{  \, \wt\E\bigl[  (S^{\alpha-\e}_{1,m}  - S^{\alpha}_{1,m})^2 \,\bigr] \,\bigr\}^{1/2}  \,
\bigl(\,\wt\P\{ \exitt^{\alpha-\e}_1\,>\,\exitt^{\alpha}_1\}\, \bigr)^{1/2}   \\[4pt] 
&\qquad=  \frac{\e}{\alpha-\e}\, \bigl\{  \, \wt\E\bigl[  ( S^{\alpha}_{1,m})^2 \,\bigr] \,\bigr\}^{1/2}  \,
\bigl(\,\wt\P\{ \exitt^{\alpha-\e}_1\,>\,\exitt^{\alpha}_1\}\, \bigr)^{1/2}  
=  C(m,\alpha) \, \frac{\e}{\alpha-\e}\,  \bigl(\,\wt\P\{ \exitt^{\alpha-\e}_1\,>\,\exitt^{\alpha}_1\}\, \bigr)^{1/2} 
\end{align*}
 for a constant $C(m,\alpha)$.   It remains to argue that 
 \be\label{cov89} \nn 
 \lim_{\e\to0}   \wt\P\{ \exitt^{\alpha-\e}_1\,>\,\exitt^{\alpha}_1\} =0.  
 \ee
 This  comes from  dominated convergence as follows.  For almost every fixed  weights, there exists $\delta>0$ such that the maximizing  path $x^{\alpha, \rho*}_\bbullet$  for $\Gpp^{\alpha,\rho}_{0,(m,n)}$ has weight at least $\delta$ above the weight of the next best path.  Thus, once $\e$ is small enough so that $S^{\alpha-\e}_{1,m} <  S^{\alpha}_{1,m} +\delta$,  path $x^{\alpha, \rho*}_\bbullet$ is maximal also for $\Gpp^{\alpha-\e,\rho}_{0,(m,n)}$ and then $\exitt^{\alpha-\e}_1=\exitt^{\alpha}_1$.  Thus $\ddd\lim_{\e\to0} \ind_{\{\exitt^{\alpha-\e}_1\,>\,\exitt^{\alpha}_1\}}=0$ almost surely.   We have verified  that the  expectation on line \eqref{line5} is of order $o(\e)$. 

Now return to \eqref{line4}--\eqref{line5} to collect the steps to compute the derivative: 
\begin{align*}
&\frac{\partial}{\partial\alpha} \wt\E\Bigl[  \,\sum_{i=1}^{\exitt^\alpha_1} F_\alpha^{-1}(U_i) + \Gpp_{(\exitt^\alpha_1,1), (m,n)}  \Bigr]  \\
&\qquad =-\,\lim_{\e\to0}  \frac1\e\biggl\{  
\wt\E\Bigl[  \,\sum_{i=1}^{\exitt^{\alpha-\e}_1} F_{\alpha-\e}^{-1}(U_i) + \Gpp_{(\exitt^{\alpha-\e}_1,1), (m,n)}  \Bigr]
-
\wt\E\Bigl[  \,\sum_{i=1}^{\exitt^\alpha_1} F_\alpha^{-1}(U_i) + \Gpp_{(\exitt^\alpha_1,1), (m,n)}   \Bigr] 
\biggr\}  \\
&\qquad =-\,\lim_{\e\to0}  \frac1\e\Bigl\{   \bigl(\e+O(\e^2)\bigr)  \alpha^{-1}  \E^{\alpha,\rho}[ S_{1,\exitt_1}]  +   o(\e) \Bigr\} 
\;=\; -\alpha^{-1}  \E^{\alpha,\rho}[ S_{1,\exitt_1}] . 
\end{align*}
Calculation \eqref{cov78} has been justified.  

To finish the proof, return to treat the last term on the last line of \eqref{aux3}: 
\begin{align*}
 \Cvv(\south ,\north )&=  \Cvv^{1-\rho,\rho}(\south ,\north ) 
=- \frac{\partial}{\partial\alpha} \E^{\alpha,\rho}(\north ) \big\vert_{\alpha=1-\rho}\\
&=  \frac1{1-\rho} \E^{1-\rho,\rho}[ S_{1,\exitt_1}] =   \frac1{1-\rho} \E^\rho[ S_{1,\exitt_1}].
\end{align*} 
This completes the proof of identity  \eqref{VGr}. 
 \end{proof}  


\subsection{Upper bound for the passage time exponent}\label{s-Grho-ub}    This upper bound proof is an improved version of the proof originally given  in \cite{bala-cato-sepp}. Article \cite{bala-cato-sepp} was a corner growth model adaptation of   the arguments of the earlier work \cite{cato-groe-06} on increasing sequences among planar Poisson points.  

We couple the boundary variables for two different parameters  $0<\rho<\lambda<1$ as follows:  
\be\label{coup7} 
I^\lambda_{ie_1} = \frac{1-\rho}{1-\lambda} I^\rho_{ie_1}>I^\rho_{ie_1} \qquad\text{and}\qquad   J^\lambda_{je_2} = \frac{\rho}{\lambda} J^\rho_{je_2}<J^\rho_{je_2}. 
\ee
From this  follows for example  that   $\exitt^\lambda_1\ge \exitt^\rho_1$ and   $\exitt^\lambda_2\le \exitt^\rho_2$, and also 
\be\label{coup8} S^\lambda_{1,\ell}-S^\rho_{1,\ell}\le S^\lambda_{1,k}-S^\rho_{1,k} \qquad \text{for $0\le \ell \le k$. }\ee

We begin with auxiliary lemmas.  

\begin{lemma} \label{lm66}   Let $0<\e<1$.   Then there exists a constant $C=C(\e)$ such that, for $\e\le \rho<\lambda\le 1-\e$,  
\[  \Vvv[G^\lambda_{0,(m,n)}] \le    \Vvv[G^\rho_{0,(m,n)}]
 + Cm(\lambda-\rho). 
 \]
\end{lemma}

\begin{proof}    From \eqref{coup7} 
\[  S^\lambda_{2, \exitt_2}=\sum_{j=1}^{\exitt^\lambda_2}J^\lambda_{je_2} 
\le  \sum_{j=1}^{\exitt^\rho_2}J^\lambda_{je_2}  =   \frac{\rho}{\lambda}  \sum_{j=1}^{\exitt^\rho_2}J^\rho_{je_2}   =   \frac{\rho}{\lambda}  S^\rho_{2, \exitt_2} . 
\]
Using the second line of \eqref{VGr}, 
\begin{align*}
\Vvv[G^\lambda_{0,(m,n)}]  &=  \frac{m}{(1-\lambda)^2}-\frac{n}{\lambda^2} +\frac{2}{\lambda}  \E[S^\lambda_{2, \exitt_2}] \\
&\le  \frac{(1-\rho)^2}{(1-\lambda)^2} \cdot  \frac{m}{(1-\rho)^2}-\frac{\rho^2}{\lambda^2}\cdot \frac{n}{\rho^2} +\frac{\rho^2}{\lambda^2}\cdot \frac{2}{\rho}  \E[S^\rho_{2, \exitt_2}]\\
&= \frac{\rho^2}{\lambda^2}\cdot   \Vvv[G^\rho_{0,(m,n)}] +  \frac{m}{(1-\rho)^2}\biggl(  \frac{(1-\rho)^2}{(1-\lambda)^2}-  \frac{\rho^2}{\lambda^2}\biggr)  \\
&\le  \Vvv[G^\rho_{0,(m,n)}] + Cm(\lambda-\rho). 
\qedhere \end{align*}
\end{proof}

\medskip 

\begin{lemma} \label{lm70}     Let $0<\e<1$.   Then there exists a constant $C=C(\e)$ such that, for $\e\le \rho\le 1-\e$,  
\be\label{S9}  \E^\rho_{0,(m,n)}[S_{1,\exitt_1}]   \le   C \bigl(  \E^\rho_{0,(m,n)}[\exitt_1] + 1) . 
 \ee
\end{lemma}

\begin{proof}  Recall that $I_{ie_1}\sim$ Exp($1-\rho$).   The steps of the derivation are   self-explanatory.  
\begin{align*}
 \E [S_{1,\exitt_1}]  &= \E\biggl[\;\sum_{i=1}^{\exitt_1} I_{ie_1}\biggr]  =  \E\biggl[\;\sum_{i=1}^{\tau_1} \overline I_{ie_1}\biggr]  +  \frac{\E[\exitt_1]}{1-\rho} 
 =  \sum_{k=1}^m  \E[\;\overline S_{1,k} \,, \exitt_1=k ]  +  \frac{\E[\exitt_1]}{1-\rho} \\
 &\le \sum_{k=1}^m  \E[\;\overline S_{1,k} \,, \overline S_{1,k}\le k,\,  \exitt_1=k ]  +  \sum_{k=1}^m  \E[\;\overline S_{1,k} \,, \overline S_{1,k}> k,\,  \exitt_1=k ] + \frac{\E[\exitt_1]}{1-\rho} \\
 &\le \sum_{k=1}^m  k \,\P(  \exitt_1=k )  +  \sum_{k=1}^m  \E[\;\overline S_{1,k} \,, \overline S_{1,k}> k  ] + \frac{\E[\exitt_1]}{1-\rho} \\
  &\le \E[\exitt_1] \frac{2-\rho}{1-\rho}    +  \sum_{k=1}^m \bigl(  \E[\;\overline S_{1,k}^2 \,] \bigr)^{1/2}  \bigl(  P\{\overline S_{1,k}> k  \}\bigr)^{1/2}  \\
    &\le \E[\exitt_1] \frac{2-\rho}{1-\rho}    +  C(\rho)  \sum_{k=1}^\infty k^{1/2} e^{-c(\rho)k}   
 \le   C(\rho) \bigl(  \E[\exitt_1] + 1\bigr) .  
\qedhere   \end{align*}
\end{proof} 

The main estimate for the upper bound in \eqref{e:Gr11} is contained in the next proposition.    Recall the definition of $\kappa_N$ from  \eqref{ch5}: 
\[  
\kappa_N=\abs{\,m-N(1-\rho)^2\,}\vee\abs{\,n-N\rho^2\,}   .  
\]  

\begin{proposition} \label{tpr6}   Consider the increment-stationary CGM  in the rectangle $[0,(m,n)]$ for $(m,n)\in\Z_{>0}^2$ and    with weight distributions \eqref{e:w7} for a given $0<\rho<1$.    Three positive constants $a_0$, $a_1$ and $N_0$ are given and the assumption is that 
\be\label{a11}  \kappa_N\le a_0N^{2/3} \quad\text{and}\quad m\le a_1N \quad\text{for} \quad N\ge N_0. 
\ee
 
   Then  there exist constants $c_2, c_3<\infty$ such that  the following two bounds hold:  
\be\label{t8}  \P^\rho_{0,(m,n)}\{\exitt_1\ge \ell\} \le   c_3\Bigl(\,\frac{N^2}{\ell^3}  + (1+a_0)   \frac{N^{8/3}}{\ell^4}\Bigr)
\qquad 
\text{ for  $N\ge N_0$ and   $1\vee c_2\kappa_N\le \ell \le m$}  \ee
and 
\be\label{t9}  \E^\rho_{0,(m,n)}[(\exitt_1)^q\,] \le \bigl( c_2a_0^q+ \frac{c_3}{3-q}\Bigr)N^{2q/3}\qquad \text{for $N\ge N_0$ and   $1\le q<3$.}  \ee
The functional dependencies of the constants $c_2,c_3$ on the parameters is as follows:  \be\label{t10} c_2=c_2(\rho) \quad\text{ and }\quad  c_3=c_3(a_1,\rho).   \ee
Furthermore, $c_2$ and $c_3$ are locally bounded functions of their arguments. 
\end{proposition}


The upper variance  bound  in \eqref{e:Gr11} follows from a combination of  \eqref{VGr},   assumption \eqref{ch5}, \eqref{S9}, and \eqref{t9} for $q=1$. 

\begin{proof}[Proof of Proposition \ref{tpr6}]    Consider $N\ge N_0$ so that the assumptions are in force. 
    Assume that, for some $0<c_2<\infty$,  the integer $\ell$ satisfies 
\be\label{ell-ass}    1\vee c_2\kappa_N\le \ell \le m\le a_1N.   \ee
 The proof will choose $c_2=c_2(\rho)$  large enough. 
Let $0<r<1$ be a constant that will be set small enough in the proof.   Let 
\be\label{la1}
\lambda=\rho+\frac{r\ell}N. 
\ee
We take $r=r(a_1,\rho)$ at least small enough so that $ra_1<\tfrac12(1-\rho)$.  This guarantees  that for $N\ge N_0$, $\lambda\in(\rho, \frac{1+\rho}2)$ is also a legitimate parameter for an increment-stationary CGM. 

In the first inequality  below  use  $S^\lambda_{1,k}+G_{(k,1),(m,n)}\le G^\lambda_{0,(m,n)}$ and then  \eqref{coup8}.     Recall that $\overline X=X-\E X$ denotes a centered random variable.  
\begin{align}
\P\{ \exitt^\rho_1\ge \ell\} &=\P\{\,\exists k\ge \ell: \, S^\rho_{1,k}+G_{(k,1),(m,n)}=G^\rho_{0,(m,n)} \,\} \nn\\
&\le  \P\{\,\exists k\ge \ell: \,  S^\lambda_{1,k}-S^\rho_{1,k}\le G^\lambda_{0,(m,n)} -G^\rho_{0,(m,n)} \,\} \nn\\
&=  \P\{\,  S^\lambda_{1,\ell}-S^\rho_{1,\ell}\le G^\lambda_{0,(m,n)} -G^\rho_{0,(m,n)} \,\} \nn\\
&=\P\Bigl\{\,  \overline{S^\lambda_{1,\ell}}- \overline{S^\rho_{1,\ell}}  \le  \overline{G^\lambda_{0,(m,n)}} -\overline{G^\rho_{0,(m,n)}}  -  \bigl( \,\E[S^\lambda_{1,\ell}-S^\rho_{1,\ell}] -  \E[G^\lambda_{0,(m,n)}-G^\rho_{0,(m,n)}] \,\bigr)    \,\Bigr\} 
\label{L78}
\end{align}

Next we compute and bound  the means of the random variables in the last probability above.   This will show that the event in the probability is a deviation, and not a typical event.  First 
\begin{align}\label{65} 
&\E[S^\lambda_{1,\ell}-S^\rho_{1,\ell}]=  \ell\Bigl(\frac1{1-\lambda}-\frac1{1-\rho}\Bigr)
= \frac{\ell}{(1-\lambda)(1-\rho)} (\lambda-\rho)
= \frac{1}{(1-\lambda)(1-\rho)} \cdot  \frac{r\ell^2}{N}
\end{align} 
Introduce the quantities 
\be\label{ka4}   \kappa^1_N=m-N(1-\rho)^2 \quad\text{and}\quad  \kappa^2_N=n-N\rho^2 \ee 
that  satisfy (with $\kappa_N$ as in \eqref{ch5})  
\[  \abs{\kappa^1_N}\vee\abs{\kappa^2_N}= \kappa_N.  \] 
Then  the LPP values.  
\be\label{66} \begin{aligned}
&\E[G^\lambda_{0,(m,n)}-G^\rho_{0,(m,n)}]  
=m\Bigl(\frac1{1-\lambda}-\frac1{1-\rho}\Bigr) + n\Bigl(\frac1{\lambda}-\frac1{\rho}\Bigr)\\
&\quad 
=\Bigl( \,\frac{m}{(1-\lambda)(1-\rho)}-\frac{n}{\lambda\rho}\,\Bigr)(\lambda-\rho) \\
&\quad 
= N\Bigl( \,\frac{1-\rho}{1-\lambda}-\frac{\rho}{\lambda}\,\Bigr)(\lambda-\rho)
+ \Bigl( \frac{\kappa^1_N}{(1-\lambda)(1-\rho)}-\frac{\kappa^2_N}{\lambda\rho}\Bigr)(\lambda-\rho)  \\
&\quad 
=  \frac{N}{\lambda(1-\lambda)} (\lambda-\rho)^2
+ \Bigl( \frac{\kappa^1_N}{(1-\lambda)(1-\rho)}-\frac{\kappa^2_N}{\lambda\rho}\Bigr)(\lambda-\rho)  \\
&\quad 
=  \frac{r^2\ell^2}{\lambda(1-\lambda)N}  
+ \Bigl( \frac{\kappa^1_N}{(1-\lambda)(1-\rho)}-\frac{\kappa^2_N}{\lambda\rho}\Bigr)\frac{r\ell}N  \\
&\quad 
\le   \frac{r^2\ell^2}{\lambda(1-\lambda)N}  +  \Bigl( \frac1{(1-\lambda)(1-\rho)}+\frac1{\lambda\rho}\Bigr) \cdot  \frac{r\ell^2}{c_2N} .
\end{aligned}\ee
The last inequality used    $\abs{\kappa^1_N}\vee\abs{\kappa^2_N}=\kappa_N\le \ell/c_2$ from \eqref{ell-ass}. 

The key step is the appearance of the quadratic $(\lambda-\rho)^2$ in the middle of the development above.  This is how curvature enters  this upper bound proof.    The quadratic comes  because $\lambda\mapsto \E[G^\lambda_{0,(m,n)}]$ is minimized at $\lambda=\rho$ by virtue of the choice of $(m,n)$ as the (approximate) characteristic direction for $\rho$. 

On the last line above the first term is the important one and the second an error that came from the $\kappa_N$ discrepancies.  
Comparison of the means in \eqref{65} and  \eqref{66}  shows that if we choose $c_2$ large enough and then $r$ small enough, both  as functions of $(\lambda, \rho)$, then for a large enough constant  $c_3=c_3(\lambda, \rho)$, 
\be\label{68}  \E[S^\lambda_{1,\ell}-S^\rho_{1,\ell}]  > \E[G^\lambda_{0,(m,n)}-G^\rho_{0,(m,n)}]   +    \frac{r\ell^2}{c_3N}. \ee 
Since the range $\lambda\in(\rho, \frac{1+\rho}2)$ is determined by $\rho$, the dependence on $\lambda$ can be dropped and we have $c_2=c_2(\rho)$,  $r=r(a_1,\rho)$ and $c_3=c_3(\rho)$. 

We continue the bound on $\P\{ \exitt^\rho_1\ge \ell\} $  from line \eqref{L78} and apply \eqref{68}.    Below we subsume 
 $(r, a_1, \rho,\lambda)$-dependent factors into a constant $C=C(a_1,\rho)$.    Along the way we use Lemma \ref{lm66}, Theorem \ref{thGr3}, \eqref{ka4},  $\kappa_N\le c_2^{-1}\ell$, $m\le a_1N$,  and Lemma \ref{lm70}.  
\begin{align}
&\P\{ \exitt^\rho_1\ge \ell\}  \le  
\P\Bigl\{\,  \overline{S^\lambda_{1,\ell}}-\overline{S^\rho_{1,\ell}}  \le  \overline{G^\lambda_{0,(m,n)}} -\overline{G^\rho_{0,(m,n)}}  -  \frac{r\ell^2}{c_3N}  \,\Bigr\} \nn\\
&\quad\le 
\P\Bigl\{\,  \overline{S^\lambda_{1,\ell}}-\overline{S^\rho_{1,\ell}}  \le   -  \frac{r\ell^2}{2c_3N}  \,\Bigr\}
+
\P\Bigl\{\,    \overline{G^\lambda_{0,(m,n)}} -\overline{G^\rho_{0,(m,n)}}  \ge  \frac{r\ell^2}{2c_3N}  \,\Bigr\} \nn\\
&\quad\le 
\frac{CN^2}{\ell^4} \Vvv[S^\lambda_{1,\ell}-S^\rho_{1,\ell}] 
+   \frac{CN^2}{\ell^4} \Vvv[G^\lambda_{0,(m,n)} -G^\rho_{0,(m,n)}]  
\label{L82}
\nn\\
&\quad\le 
\frac{CN^2}{\ell^3}  +   \frac{CN^2}{\ell^4} \bigl(\, \Vvv[G^\lambda_{0,(m,n)} ] +\Vvv[G^\rho_{0,(m,n)}] \, \bigr) \nn\\
&\quad\le 
\frac{CN^2}{\ell^3}  +   \frac{CN^2}{\ell^4}  \bigl( \Vvv[G^\rho_{0,(m,n)}]  +   m(\lambda-\rho)\bigr) \nn\\
&\quad= 
\frac{CN^2}{\ell^3}  +   \frac{CN^2}{\ell^4}  \Bigl(   -\,\frac{m}{(1-\rho)^2}+\frac{n}{\rho^2} +\frac{2}{1-\rho}  \E[S^\rho_{1,\exitt_1}]  +  a_1N \cdot 
\frac{r\ell}N\,\Bigr) \nn\\
&\quad\le  
\frac{CN^2}{\ell^3}  +   \frac{CN^2}{\ell^4}  \bigl(     \E[\exitt^\rho_1]  +   \ell \bigr) 
\le   
\frac{CN^2}{\ell^3}  +   \frac{CN^2}{\ell^4}      \E[\exitt^\rho_1] . 
\end{align}


The fourth inequality above used the fact that the boundary variables have diffusive fluctuations, that is, the variance of $S^\lambda_{1,\ell}-S^\rho_{1,\ell}$ is of order $\ell$. 

  Define a constant 
\[ b=c_2a_0+1+C,  \] 
 with $c_2$ from \eqref{ell-ass} and  $C=C(a_1,\rho)$ from line \eqref{L82}  above.    By  the assumption  $\kappa_N\le a_0N^{2/3}$  this ensures $bN^{2/3}\ge c_2\kappa_N$,  which lets us use  bound \eqref{L82}  for integers $\ell\ge bN^{2/3}$.  By adjusting the constant $C$ in the front of \eqref{L82} we can apply the bound to all real $\ell\ge bN^{2/3}$.    Note first below that $\exitt^\rho_1\le m$ because we are considering the point-to-point LPP process in the rectangle $[0, (m,n)]$.    Then get an upper  bound  by increasing the upper integration limit to $\infty$. 
\begin{align*}
  \E[\exitt^\rho_1]  &= \int_0^m  \P(\exitt^\rho_1\ge s) \,ds 
\;  \le \; bN^{2/3} + 
 C \int_{bN^{2/3}}^\infty  \Bigl(   \frac{N^2}{s^3}  +   \frac{N^2}{s^4}      \E[\exitt^\rho_1] \Bigr)\,ds
\\
&= bN^{2/3} +   \frac{CN^{2/3}}{2b^2} +   \frac{C}{3b^3}    \E[\exitt^\rho_1]
\;  \le \;    bN^{2/3} +   \tfrac12N^{2/3}  +   \tfrac13    \E[\exitt^\rho_1]. 
\end{align*}  
From this we obtain the bound
\[     \E[\exitt^\rho_1]   \le   \bigl(c_2(\rho)a_0+ C_1(a_1,\rho)\bigr) N^{2/3} \]
and thereby \eqref{t9} has been proved for $q=1$.  Substituting this bound back into line \eqref{L82} gives 
\[  \P\{\exitt^\rho_1\ge \ell\} \le   C_2 \Bigl(\,\frac{N^2}{\ell^3}  + (1+a_0)   \frac{N^{8/3}}{\ell^4}\Bigr) \]
for a constant $C_2=C_2(a_1,\rho)$, verifying \eqref{t8}.   Another integration with  $  b=c_2(\rho)a_0+1+C_2  $
  proves  \eqref{t9}   for $1<q<3$:
 \begin{align*}
  \E[(\exitt^\rho_1)^q]  &= \int_0^\infty  \P(\exitt^\rho_1\ge s) \,qs^{q-1} \,ds 
\;  \le \; b^qN^{2q/3} + 
 C_2 \int_{bN^{2/3}}^\infty  \bigl(   {N^2}{s^{q-4}}  +  (1+a_0)  {N^{8/3}}{s^{q-5}}     \bigr)\,ds
\\[2pt]
&= b^qN^{2q/3} + \frac{C_2b^{q-3}}{3-q} N^{2q/3} +  \frac{C_2}{b^{4-q}} (1+a_0)N^{2q/3}\le  \Bigl(c_2(\rho)a_0^q + \frac{C_3}{3-q}\Bigr)N^{2q/3} 
 \end{align*}   
 where we summarized the  $(a_1,\rho)$-dependent constants into $C_3=C_3(a_1,\rho)$.   Note that $a_0$ acquired a power $q$ because it is contained in $b$.  
 
   This completes the proof of Proposition \ref{tpr6},   with $c_3$ defined as the constant that appears in front of the right-hand sides of \eqref{t8}--\eqref{t9}.  
\end{proof}

We state a corollary that quantifies the effect of deviating the direction from the characteristic.   Let $N$ be the scaling parameter and the endpoint 
\[  (m,n)=\bigl(  \fl{N(1-\rho)^2}\,,  \fl{N\rho^2}   \bigr). \]

\begin{corollary}\label{cor-lb8}    Fix $0<\rho<1$.  Then there exists a constant $C=C(\rho)$ such that  for $N>0$ such that $m\wedge n\ge 1$ and  $b>0$, 
\be\label{lb-23}   \P^\rho_{0,(m,\,n+\fl{bN^{2/3}})}\{\exitt_1\ge 1\} \le \frac{C}{b^3}
\ee
and 
\be\label{lb-24}   \P^\rho_{0,(m,\,n- \fl{bN^{2/3}})}\{\exitt_2\ge 1\} \le \frac{C}{b^3}
\ee
\end{corollary}

\begin{proof}    For \eqref{lb-23}   introduce another scaling parameter $M$ and a constant $d$  via
\[ M\rho^2= n+bN^{2/3} \quad\text{and}\quad d=b\bigl(\tfrac{1-\rho}\rho\bigr)^2. \] 
Then  $\fl{M\rho^2}= n+\fl{bN^{2/3}}$ while  
\begin{align*}
M(1-\rho)^2   = \frac{n(1-\rho)^2}{\rho^2} + d N^{2/3}  = m + d N^{2/3} + \frac{\fl{N\rho^2} (1-\rho)^2}{\rho^2} - \fl{N(1-\rho)^2}, 
\end{align*}
from which follows 
\[    \fl{M(1-\rho)^2} \ge  m + \fl{d N^{2/3}}   - \biggl(\frac{1-\rho}\rho\biggr)^2.  \] 

We can assume that $b\ge 2$.  By ensuring that  $C\ge 8$, \eqref{lb-23}  will then hold trivially for $0<b<2$.    Assumption $m\wedge n\ge 1$ forces also $N>1$. 
By  Lemma \ref{app-lm3}  in Appendix \ref{app:coupl}, 
\begin{align*}
 \P^\rho_{0,(m,\,n+\fl{bN^{2/3}})}\{\exitt_1\ge 1\} 
 &=  \P^\rho_{0,(\fl{M(1-\rho)^2}\,,  \,\fl{M\rho^2}  )}\{\,\exitt_1\ge 1+\fl{M(1-\rho)^2}-m\,\}\\
& \le  \P^\rho_{0,(\fl{M(1-\rho)^2}\,,  \,\fl{M\rho^2}  )}\Bigl\{\exitt_1\ge dN^{2/3} - \bigl(\tfrac{1-\rho}\rho\bigr)^2\Bigr\}  \\
& \le  \P^\rho_{0,(\fl{M(1-\rho)^2}\,,  \,\fl{M\rho^2}  )}\Bigl\{\exitt_1\ge \tfrac12b\bigl(\tfrac{1-\rho}\rho\bigr)^2N^{2/3}  \Bigr\}  \\
&\le \frac{C}{b^3}. 
\end{align*}
The last inequality is from   the upper bound \eqref{t8}. 

\medskip 

For bound \eqref{lb-24} apply again 
  Lemma \ref{app-lm3}  in Appendix \ref{app:coupl} and then   the upper bound \eqref{t8}: 
\[ 
 \P^\rho_{0,(m,\,n- \fl{bN^{2/3}})}\{\exitt_2\ge 1\} \le  \P^\rho_{0,(m,n)}\{\exitt_2\ge bN^{2/3}\} \le \frac{C}{b^3}
\qedhere \] 
\end{proof} 

\bigskip 

\subsection{Lower bound for the passage time exponent}   This lower bound proof is adapted to the corner growth model from the one given in \cite{bala-sepp-aom}  for the asymmetric simple exclusion process (ASEP).

The parameter $0<\rho<1$ of the increment-stationary LPP process is fixed.   Let $N$ be the scaling parameter that is sent to infinity and  $(m,n)$  the endpoint of the point-to-point LPP process.  As before, we measure deviation from the characteristic direction of $\rho$  by 
$\kappa_N=\abs{m-N(1-\rho)^2}\vee\abs{n-N\rho^2}$. 
We prove the lower  bounds  on the right and left tail stated in the theorem below.

\begin{theorem}\label{t-st-lb}  Assume that for some constants $a_0, k_0>0$, we have $\kappa_N\le a_0N^{2/3}$ for $N\ge k_0$.   Then we have the conclusions below. 

There exist  constants   $1<a_1(\rho), a_2(\rho),  N_0(\rho, s)<\infty$  such that, for $s\ge   a_2(\rho)(1+ a_0^2)$ and   $N\ge N_0(\rho, s)\vee k_0$, 
 \be\label{st-lb-1}
 \P\bigl\{ \w: G^\rho_{0,(m,n)}  \ge \E[G^\rho_{0,(m,n)}] +  sN^{1/3}  \bigr\}
 \ge  e^{-a_1(\rho)s^{3/2}}  .  
 \ee
 Furthermore, there exist  constants   $0<t=t(\rho, a_0), a_3(\rho, a_0),    N_1(\rho)<\infty$  such that for    $N\ge N_1(\rho)\vee k_0$, 
  \be\label{st-lb-2}
 \P\bigl\{ \w: G^\rho_{0,(m,n)}  \le \E[G^\rho_{0,(m,n)}]  - tN^{1/3}  \bigr\}
 \ge  a_3(\rho, a_0) t^2  .  
 \ee
\end{theorem}

From \eqref{st-lb-1} we derive  the lower variance bound in \eqref{e:Gr11}: 
\begin{align*}
\Vvv[G^\rho_{0,(m,n)}]  =\E\bigl[ \bigl( G_{0,(m,n)}  - \E^\rho[G_{0,(m,n)}]\bigr)^2\, \bigr]
\ge  s^2N^{2/3} \cdot   e^{-a_1(\rho)s^{3/2}} . 
\end{align*} 
Setting $s= a_2(\rho)(1+ a_0^2)$ gives the lower bound in \eqref{e:Gr11.5}.

%
%
%

\medskip 

In this proof also we perturb the parameter of the boundary weights.  
Introduce a quantity  $r>0$ which, in the end, will be a constant multiple of $s^{1/2}$ in \eqref{st-lb-1}.  Define another parameter for the increment-stationary CGM by 
\[\lambda=\rho+\frac{r}{N^{1/3}}. \]
To guarantee that  $\lambda\in(\rho, \frac{1+\rho}2)$ we   assume that 
\be\label{e:N4}   N\ge N_0=N_0(\rho, r)=8\bigl(\tfrac{r}{1-\rho}\bigr)^3.   \ee
$N_0$ can  increase during the proof,  but remains a function of $\rho$ and $r$. 

\medskip 

{\it Notational comment.}    In this section we argue by switching measures on the space of weights, in addition to  coupling different weight configurations.  Hence  it is  convenient to attach the parameters $\rho$ and  $\lambda$ to the measure $\P$ and the expectation $\E$ and variance $\Vvv$  to indicate which distribution is placed on the weights.    We denote all the weights   by $\w_x$.  So for example   under the probability measure   $\P^\rho_{0,(m,n)}$ the weights on the rectangle $[0,(m,n)]$  have distributions as in   \eqref{e:w7} but without the $I$ and $J$ notation,  namely 
 \be\label{e:w9} 
 \text{$\w_{x}\sim$  Exp$(1)$ for bulk vertices $x\in\Z_{>0}^2$,  \  $\w_{ie_1}\sim$   Exp$(1-\rho)$,  \   and \  $\w_{je_2}\sim$   Exp$(\rho)$.  }  
 \ee
The point-to-point LPP process is defined   by   
\be\label{e:w13}   G_{0,(m,n)}=\max_{x_{0,m+n}\in\Pi_{0,(m,n)}} \sum_{k=1}^{m+n} \w_{x_k} \ee 
with maximum over paths that satisfy $x_0=0$ and $x_{m+n}=(m,n)$. 
\hfill   $\triangle$ 

\medskip 
 
For $N\in\Z_{>0}$ and $r>0$ define the event
\be\label{ANr}  A_{N,r}=\bigl\{ (1-\rho)rN^{2/3}\le \exitt_1 \le 4\rho^{-1}rN^{2/3}\bigr\}. \ee 
Variable $\exitt_1$ is the exit point from the $e_1$-axis  of the maximizing path from $0$ to $(m,n)$ in \eqref{e:w13}, defined as in  \eqref{d-tt}.  
We develop a lower bound for the probability of $A_{N,r}$ under $\P^\lambda_{0,(m,n)}$, that is,  for the increment-stationary   process with parameter $\lambda$, restricted to the rectangle $[0,(m,n)]$.  Note that this rectangle is {\it not} of the characteristic shape for $\lambda$.  We take advantage of this in the proof.  
 
\begin{lemma}\label{lm-ANr}  There exists a constant $C_1=C_1(\rho)$ such that  for 
\be\label{r54}  r\ge 1\vee \frac{2a_0}{\rho\wedge(1-\rho)}\ee
 and for $N\ge N_0(\rho, r)$ we have this bound:
\be\label{t55} 
\P^\lambda_{0,(m,n)}(A_{N,r}) \ge 1- C_1r^{-3} . 
\ee
\end{lemma}

\begin{proof}   We derive first an upper bound for $ \P^\lambda_{0,(m,n)}\{ \exitt_1 > 4\rho^{-1}rN^{2/3}\}$.  Define a coordinate $\wt m$  and bound it as below: 
\begin{align*}
  \wt m=\bigl\lfloor n\lambda^{-2}(1-\lambda)^2\bigr\rfloor  
&\ge N\rho^2 \lambda^{-2}(1-\lambda)^2 - a_0 \lambda^{-2}(1-\lambda)^2 N^{2/3} -1\\
&\ge  N\rho^2 \lambda^{-2}(1-\lambda)^2 - a_0 \rho^{-2}(1-\rho)^2 N^{2/3} -1.
\end{align*}
The second inequality is from $\lambda>\rho$. 

Vector   $(\wt m,n)$ points in the characteristic direction of $\lambda$,  up to a constant    error   from integer parts. 
We have the bound 
\begin{align*}
 0< m-\wt m&\le N\bigl( (1-\rho)^2 -\rho^2\lambda^{-2}(1-\lambda)^2 \bigr)+a_0\bigl( 1+\rho^{-2}(1-\rho)^2\bigr)  N^{2/3} +1\\
 & \le  N\frac{\lambda+\rho-2\lambda\rho}{\lambda^2} (\lambda-\rho) + a_0 \rho^{-2}  N^{2/3}+1 \\
&\le  {2}\rho^{-1} rN^{2/3} + a_0 \rho^{-2}  N^{2/3}
\end{align*} 
for $N\ge N_0(\rho, r)$,  for a suitably chosen $N_0(\rho, r)$.  
 By Lemma \ref{app-lm3} in the appendix,  and then by the upper bound \eqref{t8},  taking into consideration \eqref{r54}, 
\begin{align*}
 \P^\lambda_{0,(m,n)}\{ \exitt_1 > 4\rho^{-1}rN^{2/3}\} &=   \P^\lambda_{0,(\wt m,n)}\{ \exitt_1 > 4\rho^{-1}rN^{2/3}-(m-\wt m)\}  \\
& \le   \P^\lambda_{0,(\wt m,n)}\{ \exitt_1 > \rho^{-1}rN^{2/3}\}
 \le \frac{c_4}{r^3}  
\end{align*}
where 
$c_4=c_4(\rho)$ contains $c_3$ from \eqref{t8}.  

Next we derive an upper  bound for $ \P^\lambda_{0,(m,n)}\{ \exitt_1 < (1-\rho)rN^{2/3}\}$.    Let 
\be\label{bar} (\bar m,\bar n)=\bigl(\fl{N(1-\lambda)^2},\,\fl{N\lambda^2}\bigr)
\ee
point in the characteristic direction $\lambda$.   Bound these differences:
\begin{align*}
m-\bar m &\ge  N((1-\rho)^2-(1-\lambda)^2) -a_0N^{2/3} -1 = N(\lambda-\rho)(2-\rho-\lambda)-a_0N^{2/3} -1  \\
&= (2-\rho-\lambda) rN^{2/3} -a_0N^{2/3}-1 \ge (1-\rho)rN^{2/3}  
\end{align*}
and 
\begin{align*}
\bar n-n \ge  N(\lambda^2-\rho^2) -a_0N^{2/3} -1 = N(\lambda-\rho)(\rho+\lambda) -a_0N^{2/3}-1 \ge \rho  rN^{2/3}, 
\end{align*}
again for large enough $N$ relative to $(\rho, r)$.
By Lemma \ref{app-lm5} in the appendix,  and then by the upper bound \eqref{t8},  
\be\label{64}\begin{aligned}
\P^\lambda_{0,(m,n)}\bigl\{ \exitt_1 < (1-\rho)rN^{2/3}\bigr\}
&\le \P^\lambda_{0,(m,n)}\bigl\{ \exitt_1 <  m-\bar m\bigr\}
= \P^\lambda_{0,(\bar m, \bar n)}\bigl\{ \exitt_2 > \bar n -n \bigr\} \\
&\le \P^\lambda_{0,(\bar m, \bar n)}\bigl\{ \exitt_2 > \rho rN^{2/3}  \bigr\} \le  \frac{c_5}{r^3} 
\end{aligned}\ee
where 
$c_5=c_5(\rho)$ contains $c_3$ from \eqref{t8}.  

Combine the bounds: 
\begin{align*}
& \P^\lambda_{0,(m,n)}\bigl\{ (1-\rho)rN^{2/3}\le \exitt_1 \le 4\rho^{-1}rN^{2/3}\bigr\} \\
 &\qquad = 1- \P^\lambda_{0,(m,n)}\{ \exitt_1 > 4\rho^{-1}rN^{2/3}\}- \P^\lambda_{0,(m,n)}\{ \exitt_1 < (1-\rho)rN^{2/3}\}\\
&\qquad \ge 1- C_1r^{-3} . 
\qedhere \end{align*}
This completes the proof. 
 \end{proof} 
 
 For the subsequent parts of the proof we could  have defined $ A_{N,r}=\{ 1 \le \exitt_1 \le 4\rho^{-1}rN^{2/3}\}$ instead of with the tighter lower bound in \eqref{ANr}.   But we still would have had to prove an estimate such as \eqref{64}. 
 
Following the first four lines of the computation  in \eqref{66}, 
\be\label{69}\begin{aligned} 
 &\E^\lambda[G_{0,(m,n)}]-\E^\rho[G_{0,(m,n)}]  
=m\Bigl(\frac1{1-\lambda}-\frac1{1-\rho}\Bigr) + n\Bigl(\frac1{\lambda}-\frac1{\rho}\Bigr)\\
&=  \frac{N}{\lambda(1-\lambda)} (\lambda-\rho)^2
+ \biggl\{ \, \frac{\kappa^1_N}{(1-\lambda)(1-\rho)} \, -\, \frac{\kappa^2_N}{\lambda\rho}\,\biggr\}(\lambda-\rho)  
\\
&\ge  \frac{r^2N^{1/3}}{\lambda(1-\lambda)}  -  \frac{a_0rN^{1/3}}{C(\rho)} 
\ge c_6r^2N^{1/3}
\end{aligned}\ee 
where $c_6=c_6(\rho)>0$ is a constant chosen small enough to satisfy the inequality above for all $N\ge N_0$ and $r\ge b(\rho)a_0$ where $b(\rho)$ is a constant that depends on $\rho$.  

Let $\Lambda_N$ denote the set of paths $x_{\bbullet} \in\Pi_{0,(m,n)}$ that satisfy
$x_1=e_1$ and $x_k\cdot e_2>0$  for $k>\fl{4\rho^{-1}rN^{2/3}}$.  In other words,  the path stays on the $e_1$-axis for a while after leaving the origin,  but  not   beyond the point $\fl{4\rho^{-1}rN^{2/3}}e_1$.    For any given weights $\{\w_x\}$ on  the rectangle $\{0,\dotsc,m\}\times\{0,\dotsc,n\}$, let 
\be\label{GL}   G_{0,(m,n)}(\Lambda_N) = \max_{x_{\bbullet}\in\Lambda_N} \sum_{k=1}^{m+n} \w_{x_k}
\ee
 denote the LPP value   whose maximum is restricted to the paths in $\Lambda_N$.  Observe that  $G_{0,(m,n)}(\Lambda_N)=G_{0,(m,n)} $ if event $A_{N,r}$ of \eqref{ANr}  occurs for weights $\{\w_x\}$.  (This would be true even if the lower bound in $A_{N,r}$ would be relaxed to $\exitt_1\ge 1$ instead of $\exitt_1\ge (1-\rho)rN^{2/3}$.) 
 
We derive our second probability bound.   Define the event 
\be\label{BNr}    B_{N,r}=\bigl\{ \w: G_{0,(m,n)}(\Lambda_N) \ge \E^\rho[G_{0,(m,n)}] + \tfrac12 c_6r^2N^{1/3}  \bigr\} 
\ee
where $c_6=c_6(\rho)$ is the constant from \eqref{69}. 
 
\begin{lemma}  There exists a constant $C_2=C_2(\rho)$ such that the bound below holds for $r\ge b(\rho)a_0\vee 1$ and $N\ge N_0$:
\be\label{t59} 
 \P^\lambda_{0,(m,n)}(B_{N,r}) 
 \ge 1- C_2r^{-3} . 
\ee
\end{lemma}

\begin{proof}   Since by \eqref{69} for $r\ge b(\rho)a_0$, 
\[ \E^\rho[G_{0,(m,n)}] + \tfrac12 c_6r^2N^{1/3} \le  \E^\lambda[G_{0,(m,n)}] - \tfrac12 c_6r^2N^{1/3} , \]
we can bound the complementary probability as follows.  Constant $C$  changes from line to line.    Below we use Lemma \ref{lm66}. 
\begin{align*}
 \P^\lambda_{0,(m,n)}(B_{N,r}^c)  &=\P^\lambda_{0,(m,n)}\bigl\{ G_{0,(m,n)}(\Lambda_N) < \E^\rho[G_{0,(m,n)}] + \tfrac12 c_6r^2N^{1/3}  \bigr\}\\
 &\le \P^\lambda_{0,(m,n)}\bigl\{ G_{0,(m,n)}(\Lambda_N) < \E^\lambda[G_{0,(m,n)}] - \tfrac12 c_6r^2N^{1/3} \bigr\}\\
 &\le \P^\lambda_{0,(m,n)}\bigl\{ G_{0,(m,n)}  < \E^\lambda[G_{0,(m,n)}] - \tfrac12 c_6r^2N^{1/3} \bigr\} +  \P^\lambda_{0,(m,n)}(A_{N,r}^c) \\
 &\le \frac{C}{r^4N^{2/3}} \Vvv^\lambda[G_{0,(m,n)}]  +  \frac{C_1}{r^{3}}  \\
 &\le \frac{C}{r^4N^{2/3}} \bigl(  \Vvv^\rho[G_{0,(m,n)}]  + m(\lambda-\rho)  \bigr)  +  \frac{C_1}{r^{3}}  \\
 &\le  \frac{C}{r^4} +  \frac{C_1}{r^{3}}  \le  \frac{C_2}{r^{3}}. 
\qedhere\end{align*} 
\end{proof} 
 
 With the preliminary work done, we turn to the proof of Theorem \ref{t-st-lb}. 

\begin{proof}[Proof of Theorem \ref{t-st-lb}] 
We construct a coupling of three environments. Let $\w^\rho$  and $\w^\lambda$ denote environments as described in \eqref{e:w9} with parameters $\rho$ and $\lambda$.  We assume that these environments are coupled so that in the bulk, for $x\in\Z_{>0}^2$,  
$\w^\rho_x=\w^\lambda_x=\w_x$, while the boundary variables 
$\{\w^\rho_{ie_1}, \w^\rho_{je_2}, \w^\lambda_{ie_1}, \w^\lambda_{je_2}: i,j\in\Z_{>0}\}$ are mutually independent.  

Construct a mixed environment $\wh\w$ as follows: 
\begin{align*} 
\wh\w_{ie_1}&=\w^\lambda_{ie_1} \quad\text{for} \quad  1\le i\le \fl{4\rho^{-1}rN^{2/3}} \\
\text{and} \qquad \wh\w_{x}&=\w^\rho_{x} \quad\text{for} \quad x\notin \bigl\{ie_1: 1\le i\le \fl{4\rho^{-1}rN^{2/3}}\bigr\} . 
\end{align*}
Thus in the bulk all weights agree and  are i.i.d.\ Exp(1):  for $x\in\Z_{>0}^2$,  $\wh\w_x=\w^\rho_x=\w^\lambda_x=\w_x$.   On the boundary $\wh\w$ follows $\w^\lambda$ on the segment that is relevant for the event $B_{N,r}$ and   elsewhere $\wh\w$ follows $\w^\rho$.     Note that $\w^\lambda\in B_{N,r}$ iff  $\wh\w\in B_{N,r}$. 

Let the distributions of the three environments  $\w^\rho$, $\w^\lambda$ and $\wh\w$, restricted    to the rectangle $\{0,\dotsc,m\}\times\{0,\dotsc,n\}$, be   denoted by $\P^\rho_{0,(m,n)}$, $\P^\lambda_{0,(m,n)}$  and $\wh\P_{0,(m,n)}$, respectively.    These are all probability measures on the product space $\R_+^{\{0,\dotsc,m\}\times\{0,\dotsc,n\}}$. 
The Radon-Nikodym derivative 
\[ f_N(\w) =\frac{d\wh\P_{0,(m,n)}}{d\P^\rho_{0,(m,n)}}(\w) =\prod_{i=1}^{\fl{4\rho^{-1}rN^{2/3}}}   \frac\lambda\rho \,e^{-(\lambda-\rho)\w_{ie_1}}  \]  
is a product of the Radon-Nikodym derivatives of the exponential single weight marginal distributions on that segment of the boundary where  $\w^\rho$ and   $\wh\w$  differ.   
Computation of the mean square gives
\begin{align*}  \E^\rho_{0,(m,n)}[   f_N^2 ]&= \biggl(\frac{\lambda^2}{\rho^2}   \int_0^\infty   
e^{-2(\lambda-\rho)s} \rho e^{-\rho s}\,ds\biggr)^{\fl{4\rho^{-1}rN^{2/3}}}  
=   \biggl(\frac{\lambda^2}{\rho(2\lambda-\rho)}   \biggr)^{\fl{4\rho^{-1}rN^{2/3}}}  \\
&=\exp\biggl\{  \fl{4\rho^{-1}rN^{2/3}}\biggl[ 2\log\Bigl(1+\frac{r}{\rho N^{1/3}}\Bigr) -  \log\Bigl(1+\frac{2r}{\rho N^{1/3}}\Bigr) \biggr]  \biggr\}  \\
& \le e^{4r^3\rho^{-3}}. 
\end{align*} 
The point is that the bound above is independent of $N$. 

Consider  $r\ge b(\rho)a_0\vee 1$ and  large enough relative to $C_2=C_2(\rho)$ from \eqref{t59}  so that $C_2r^{-3}<1/2$.  
\be\label{lb15}  \begin{aligned}
\tfrac12< 1- C_2r^{-3}&\le 
  \P^\lambda_{0,(m,n)}(B_{N,r})   = \wh\P_{0,(m,n)}(B_{N,r})  =\E^\rho_{0,(m,n)}[ \one_{B_{N,r}} f_N ]  \\
   &\le \bigl\{  \P^\rho_{0,(m,n)}(B_{N,r}) \bigr\}^{1/2} \bigl\{ \E^\rho_{0,(m,n)}[   f_N^2 ] \bigr\}^{1/2}\\
  &\le \bigl\{  \P^\rho_{0,(m,n)}(B_{N,r}) \bigr\}^{1/2}  e^{2r^3\rho^{-3}}. 
\end{aligned}\ee
Since  $G_{0,(m,n)}\ge G_{0,(m,n)}(\Lambda_N)$,  from this comes the lower bound
\be\label{lb7} \begin{aligned}  
 &\P^\rho_{0,(m,n)}\bigl\{ \w: G_{0,(m,n)}  \ge \E^\rho[G_{0,(m,n)}] + \tfrac12 c_6r^2N^{1/3}  \bigr\}\\
 &\qquad \ge  \P^\rho_{0,(m,n)}\bigl\{ \w: G_{0,(m,n)}(\Lambda_N) \ge \E^\rho[G_{0,(m,n)}] + \tfrac12 c_6r^2N^{1/3}  \bigr\}\\
&\qquad =\P^\rho_{0,(m,n)}(B_{N,r}) \ge \tfrac14e^{-4r^3\rho^{-3}}
\end{aligned} \ee
To complete the proof of  inequality \eqref{st-lb-1} of  Theorem \ref{t-st-lb}, set $s=\tfrac12 c_6r^2$ and let $a_1(\rho), a_2(\rho)$ be suitable functions of $\rho$, $c_6$ and $C_2$.  

To prove the second inequality \eqref{st-lb-2}, abbreviate temporarily $X=G_{0,(m,n)} - \E^\rho[G_{0,(m,n)}] $ and first derive this estimate from  inequality \eqref{st-lb-1}:
\begin{align*} 
0=\E^\rho[X] = \E^\rho[X^+] - \E^\rho[X^-] \ge  sN^{1/3}e^{-a_1(\rho)s^{3/2}}  - \E^\rho[X^-]
= 2t N^{1/3}  - \E^\rho[X^-]   
\end{align*}
where we set 
\[  t=\tfrac12 s \,e^{-a_1(\rho)s^{3/2}}  \]   
and then fix  $s\ge a_2(\rho)(1+a_0^2)$ to maximize $t$.  
 Next, 
\begin{align*}
2tN^{1/3}  \le \E^\rho[X^-]  &= \E^\rho[X^-, \, X^-< t N^{1/3}\,]  +  \E^\rho[X^-, \, X^-\ge  tN^{1/3}\,] \\
&\le  t N^{1/3} + \bigl(  \E^\rho[(X^-)^2] \bigr)^{1/2}  \bigl(\P^\rho\{X^-\ge t N^{1/3}\}\bigr)^{1/2} \\
&\le t N^{1/3}+  \bigl(  \Vvv^\rho[  G_{0,(m,n)}] \bigr)^{1/2}  \bigl(\P^\rho\{X^-\ge t N^{1/3}\}\bigr)^{1/2} . 
\end{align*}
The above and  the upper variance bound from Theorem \ref{thGr2} give,  for a positive constant $a_3(\rho, a_0)$, 
\begin{align*}
\P^\rho\{X^-\ge t N^{1/3}\} \ge a_3(\rho, a_0) t^2  . 
\end{align*}
This inequality is the same as  \eqref{st-lb-2}.   This completes the proof of Theorem \ref{t-st-lb}.
\end{proof}  


\subsection{Bounds for path fluctuations} \label{sec:p-exp}


In this section we prove Theorem \ref{thGr2.5} for fluctuations of the maximizing path  $\pi^{0,(m,n)}_\dbullet$ of the last-passage value $G^\rho_{0,(m,n)}$.   The parameter $0<\rho<1$ is fixed and    $(m,n)=(\fl{N(1-\rho)^2}, \fl{N\rho^2})$.  

\medskip 

{\bf Upper bound proof.} 
Consider $t\in(0,1)$ fixed for a while. 
 Let 
\[  v_0=(m_0, n_0)=(\fl{tm}, \fl{tn}). \]
 This is an  integer point that 
lies immediately  to the left and below the rectangle $A_{rN^{2/3}}(tm,tn)$, and coincides with the lower left corner $(tm,tn)$ of $A_{rN^{2/3}}(tm,tn)$ if this point has integer coordinates.  See Figure \ref{fig-p1}.

\begin{figure}
\begin{center}
\begin{picture}(200,150)(20,-20)
\put(40,0){\line(1,0){170}} 

 \put(90,32){\line(1,0){120}}\put(210,0){\line(0,1){110}}
\put(40,110){\line(1,0){170}}
\put(40,0){\line(0,1){110}} \put(90,32){\line(0,1){78}}
 
\put(37,-3){\Large$\bullet$} 
\put(30,-2){\small$0$}

\put(86.5,28.5){\Large$\bullet$} 
\put(77,26){\small$v_0$}

\put(126.5,28.5){\Large$\bullet$} 
\put(134,24){\small$x$}

%

\put(206,106){\Large$\bullet$} 
\put(214,108){\small$(m,n)$}


\linethickness{3pt} 
\put(44.5,0){\line(1,0){72}}  \put(115,0){\line(0,1){21}} \put(114,20){\line(1,0){17.5}}
\put(130,20){\line(0,1){8}}  \put(130,36){\line(0,1){24}} 
\put(128.5,61.5){\line(1,0){60}} \put(189,60){\line(0,1){21.5}}
\put(189,80){\line(1,0){22}}  \put(210,78.5){\line(0,1){27}}
 \multiput(94.5,32)(8.5,0){4}{\line(1,0){5.5}}
 
\put(94,37){\dashbox{2}(30,30){$A$}} 

 \end{picture}
\end{center}  
\caption{ \small Illustration of the upper bound proof of Theorem \ref{thGr2.5}. The maximizing path (solid thickset line) for $G^\rho_{0,(m,n)}$ goes through $x$ and avoids the square $A=A_{rN^{2/3}}(tm,tn)$.   By Lemma \ref{app-lm1}, the maximizing path for $G^{(0)}_{v_0, (m,n)}$  goes from $v_0$ to $(m,n)$  through $x$.   The distance from $v_0$ to $x$ equals $\exitt^{(v_0)}_1$. } \label{fig-p1}
\end{figure}

 The proof utilizes the coupling described in Appendix \ref{app:coupl}.    As in Lemma \ref{app-lm1}, let $G^{(0)}_{v_0, x}$ denote last-passage values for lattice points $x\in v_0+\Z_{\ge 0}^2$ defined as in \eqref{Gr1}--\eqref{Gr2},   using boundary weights $I_{v_0+ie_1}$ and $J_{v_0+je_2}$ for $i,j\ge 1$ and bulk weights $\w_x$ for $x\in v_0+\Z_{>0}^2$.   The bulk weights are the same as the ones used by $G^\rho_{0,(m,n)}$.  The  boundary edge weights are determined by the   process $G^\rho_{0,{\cbullet}}$ as in \eqref{IJG6}: 
\begin{align*}
I_{v_0+ie_1}&=G^\rho_{0,\,v_0+ie_1}-G^\rho_{0,\,v_0+(i-1)e_1}\\
J_{v_0+je_2}&=G^\rho_{0,\,v_0+je_2}-G^\rho_{0,\,v_0+(j-1)e_2}
\qquad \text{for $i,j\ge 1$.}  
\end{align*}

Let $x^{(v_0)*}_{\dbullet}\in\Pi_{v_0, x}$ denote the maximizing path for $G^{(0)}_{v_0, x}$.     Analogously with \eqref{d-tt}, let $\exitt_r^{(v_0)}$  denote the number of steps the maximizing path $x^{(v_0)*}_{\bbullet}$ takes on the $e_r$-axis, relative to the origin $v_0$:
\[  \exitt_r^{(v_0)} =\max\{k\ge 0:  v_0+ke_r \,\in\, x^{(v_0)*}_{\bbullet} \} , \qquad r\in\{1,2\} . \]  

Let $(m_1, n_1)=(m,n)-(m_0,n_0)$.   This point satisfies 
\be\label{pp97}    \abs{\,m_1-(1-t)N(1-\rho)^2\,}\vee\abs{\,n_1-(1-t)N\rho^2\,}  \le 1. \ee
Due to the stationarity of the LPP process,  we have the distributional equalities 
\be\label{pp98}   \P^\rho_{0,(m,n)}\{ \exitt_r^{(v_0)} \ge j\}=\P^\rho_{0,(m_1, n_1)}\{ \exitt_r \ge j\}. \ee

\medskip

 We are ready to prove the upper bound \eqref{e:Gr13}. 
  If $\pi^{0,(m,n)}_\dbullet$ does not  intersect $A_{rN^{2/3}}(tm,tn)$, then the edge along which $\pi^{0,(m,n)}_\dbullet$ enters the quadrant $(m_0, n_0)+\Z_{>0}^2$ is either $((k,n_0),(k,n_0+1))$ for some $k\ge m_0+rN^{2/3}$ or 
 $((m_0,\ell),(m_0+1,\ell))$ for some $\ell\ge n_0+rN^{2/3}$.    Lemma \ref{app-lm1} forces the occurrence of 
 \be\label{pp102} \exitt_1^{(v_0)}\ge rN^{2/3}\quad\text{ or }\quad  \exitt_2^{(v_0)}\ge rN^{2/3}. \ee
 We apply  the upper bound  \eqref{t8} of Proposition \ref{tpr6} to the right-hand side of  \eqref{pp98}.  By \eqref{pp97} the scaling parameter ``$N$''  of Proposition \ref{tpr6}  is now $N'=(1-t)N$.  To ensure that $N'\ge 1$ we assume $N\ge (1-t)^{-1}$.   Then we can take    $N'_0=1$, $a_0=1$ and $a_1=2$ in Proposition \ref{tpr6}.  By \eqref{t8},  for $r\ge 1$  the event in \eqref{pp102}  has probability at most $C(\rho)r^{-3}$.     This bound is uniform over $t\in(0,1)$, as long as $N\ge (1-t)^{-1}$.   We have  verified the upper bound \eqref{e:Gr13}.

\medskip 

{\bf Lower bound proof.} 
 Fix $t\in(0,1)$.  Take $N$ large enough so that $(m,n)\in\Z_{>0}^2$.   This time put 
\[  v_0=(m_0, n_0)=(\ce{tm}-1, \ce{tn}-1) \;\in\;\Z_{\ge0}^2 \]
so that $v_0$ is strictly below and to the left of the rectangle $A_{\delta N^{2/3}}(tm,tn)$.  Let again $(m_1, n_1)=(m,n)-(m_0,n_0)$. 
 By Lemma \ref{lm70}, variance identity \eqref{VGr}, and the lower variance bound in \eqref{e:Gr11},   with a constant $C=C(\rho)$ that may vary, 
\begin{align*}
\E^\rho_{0,(m,n)}[ \exitt_1^{(v_0)}] &=\E^\rho_{0,(m_1,n_1)}[ \exitt_1]  \ge C^{-1}  \E^\rho_{0,(m_1,n_1)}[S_{1,\exitt_1}]  -1 \ge C^{-1} \Vvv[G^\rho_{0,(m_1,n_1)}]-C \\
&\ge C^{-1} N^{2/3} -C \ge C^{-1} N^{2/3} . 
\end{align*} 
In the last step we take $N$ large enough relative to $C(\rho)$. 
Next,  combine the lower bound from above with the upper bound of \eqref{t9} for $q=2$. Then   for $\delta_1>0$,  again with $C=C(\rho)$ and for $N\ge N_0(\rho)$, 
\begin{align*}
C^{-1} N^{2/3} &\le \E^\rho_{0,(m,n)}[ \exitt_1^{(v_0)}] =\E^\rho_{0,(m_1,n_1)}[ \exitt_1] \\
&= \E^\rho_{0,(m_1,n_1)}[ \exitt_1, \,  \exitt_1\le \delta_1N^{2/3} \,]  + \E^\rho_{0,(m_1,n_1)}[ \exitt_1, \,  \exitt_1> \delta_1N^{2/3} \,] \\
&\le  \delta_1N^{2/3} + \bigl( \E^\rho_{0,(m_1,n_1)}[ \exitt_1^2]\bigr)^{1/2} \bigl( \P^\rho_{0,(m_1,n_1)}\{ \exitt_1> \delta_1N^{2/3}\}\bigr)^{1/2} \\
&\le  \delta_1N^{2/3} +  CN^{2/3} \bigl( \P^\rho_{0,(m_1,n_1)}\{ \exitt_1> \delta_1N^{2/3} \}\bigr)^{1/2} . 
\end{align*}
From this  we deduce that, for small enough $\delta_1$ relative to the  constant $C(\rho)$, there exists    $\e_1>0$ such that  
\[  \P^\rho_{0,(m,n)}\{ \exitt_1^{(v_0)}> \delta_1N^{2/3} \}\ge \e_1\qquad \text{for all $N\ge N_0(\rho,t)$. }\]
Note that $N$ has to be taken large enough relative to $t$ also so that $\delta_1N^{2/3}<m_1$, for otherwise the event above is empty.   Hence the dependence of $N_0(\rho,t)$ on $t$. 
 A similar estimate works for $\exitt_2^{(v_0)}$.

Now the proof of the  lower bound \eqref{e:Gr14}.  
If $\pi^{0,(m,n)}_\dbullet$   intersects $A_{\delta N^{2/3}}(tm,tn)$, then the path $\pi^{0,(m,n)}_\dbullet$  enters $A_{\delta N^{2/3}}(tm,tn)$ either through the south side or the west side,  and hence either $1\le \exitt_1^{(v_0)}\le \delta N^{2/3}$ or $1\le \exitt_2^{(v_0)}\le \delta N^{2/3}$.     (Note that if $v_0$ were the lower left corner of $A_{\delta N^{2/3}}(tm,tn)$, we could not make the last assertion because the path could run along either boundary $v_0+e_r\Z_{>0}$ and give $\exitt_r^{(v_0)}$ a larger value.)
Take $\delta<\delta_1$ and  utilize the fact that exactly one of $\exitt_1^{(v_0)}$ and $\exitt_2^{(v_0)}$ is positive.  
 \begin{align*}
 &\P^\rho_{0,(m,n)}\{ \text{$\pi^{0,(m,n)}_\dbullet$    intersects $A_{\delta N^{2/3}}(tm,tn)$} \} \\
&\qquad  \le \P^\rho_{0,(m,n)}\{ 1\le \exitt_1^{(v_0)}\le \delta N^{2/3}\} + \P^\rho_{0,(m,n)}\{ 1\le \exitt_2^{(v_0)}\le \delta N^{2/3}\}  \\
&\qquad = 1- \P^\rho_{0,(m,n)}\{ \exitt_1^{(v_0)}> \delta N^{2/3}\} - \P^\rho_{0,(m,n)}\{  \exitt_2^{(v_0)}> \delta N^{2/3}\}    \le  1-2\e_1. 
 \end{align*} 
 
 This completes the proof of Theorem \ref{thGr2.5}.

%



\appendix 
 
 \section{Coupling the corner growth model} \label{app:coupl} 
 
 
 In this section we develop the couplings used in the proofs, beginning with  a lemma for arbitrary deterministic weights. 
 Fix a point $a\in\Z^2$.  
 Suppose boundary weights  $\{\w_{a+ke_r}: k\in\Z_{>0}, r\in\{1,2\}\}$ on the south and west boundaries  of the quadrant  $a+\Z_{\ge0}^2$ and bulk weights $\{\w_x\}_{x\in a+\Z_{>0}^2}$ are given.  Put an irrelevant weight $\w_a=0$ in the corner $a$.   Let  $G_{a,x}$ denote the LPP value for points $x\in a+\Z_{\ge0}^2$, defined as in \eqref{v:G}.  
 
Let $b\ge a$ on $\Z^2$.   On the lattice $b+\Z_{\ge0}^2$,  put a corner weight $\eta_b=0$ and  define boundary weights 
 \be\label{eta6} \eta_{b+ke_r} =  G_{a,\,b+ke_r}- G_{a,\,b+(k-1)e_r} 
 \qquad\text{for $k\in\Z_{>0}$ and $r\in\{1,2\}$. }\ee
    In the bulk use $\eta_x=\w_x$ for $x\in b+\Z_{>0}^2$.    Denote the LPP process in $b+\Z_{\ge0}^2$ that uses weights $\{\eta_x\}_{x\in b+\Z_{\ge0}^2}$ by 
 \be\label{wtG8}   G^{(a)}_{b,x}=\max_{x_{\bbullet}\in\Pi_{b,x}}  \sum_{i=0}^{\abs{x-b}_1} \eta_{x_i}, \qquad 
 x\in b+\Z_{\ge0}^2. \ee
 The superscript $(a)$ reminds us that the LPP process $G^{(a)}$ uses boundary weights determined by the LPP process $G_{a,\cbullet}$  with lower left base point $a$.   See Figure \ref{fig-app1} for an illustration of the next lemma. 

\begin{figure}
\begin{center}
\begin{picture}(200,140)(20,-10)
\put(40,0){\line(1,0){170}} 

 \put(90,40){\line(1,0){120}}\put(210,0){\line(0,1){110}}
\put(40,110){\line(1,0){170}}
\put(40,0){\line(0,1){110}} \put(90,40){\line(0,1){70}}
 
\put(37,-3){\Large$\bullet$} 
\put(30,-2){\small$a$}

\put(86.5,36.5){\Large$\bullet$} 
\put(80,30){\small$b$}

\put(126.5,36.5){\Large$\bullet$} 
\put(134,32){\small$x$}

%

\put(206,106){\Large$\bullet$} 
\put(215,107){\small$v$}


\linethickness{3pt} 
\put(44.5,0){\line(1,0){72}}  \put(115,0){\line(0,1){21}} \put(114,20){\line(1,0){17.5}}
\put(130,20){\line(0,1){16}}  \put(130,44){\line(0,1){16}} 
\put(128.5,61.5){\line(1,0){60}} \put(189,60){\line(0,1){21.5}}
\put(189,80){\line(1,0){22}}  \put(210,78.5){\line(0,1){27}}
 \multiput(94.5,40)(8.5,0){4}{\line(1,0){5.5}}

 \end{picture}
\end{center}  
\caption{ \small Illustration of Lemma \ref{app-lm1}. Path $a$-$x$-$v$ is maximal for $G_{a,v}$ and path $b$-$x$-$v$ is maximal for $G^{(a)}_{b,v}$. } \label{fig-app1}
\end{figure}

\begin{lemma} \label{app-lm1}  Let $a\le b\le v$ in $\Z^2$.  Then  
$G_{a,v}=G_{a,b}+G^{(a)}_{b,v}$.    The restriction of any maximizing path for $G_{a,v}$ to $b+\Z_{\ge0}^2$ is part of a maximizing path for $G^{(a)}_{b,v}$.    The edges with one endpoint in $b+\Z_{>0}^2$ that belong to 
any particular  maximizing path for $G^{(a)}_{b,v}$  extend to a maximizing path for $G_{a,b}$.  
\end{lemma} 

\begin{proof}   If $v=b+ke_r$ (that is, $v$ is on the boundary of $b+\Z_{\ge0}^2$) the situation is straightforward and we omit the details.  

 Suppose $v>b$ coordinatewise.   Suppose a maximal path for $G_{a,v}$  enters the quadrant $b+\Z_{>0}^2$ through the edge  $(x,y)=(b+ke_r, b+ke_r+e_{3-r})$.   Suppose a maximal path for  $G^{(a)}_{b,v}$  enters  $b+\Z_{>0}^2$ through the edge  $(\wt x, \wt y)=(b+\ell e_s, b+\ell e_s+e_{3-s})$.   Then 
\begin{align*}
G_{a,v}&=G_{a,x}+G_{y,v} = G_{a,b}+\sum_{i=1}^k  \eta_{b+ie_r}  + G_{y,v} \\
&\le  G_{a,b}+  G^{(a)}_{b,v}
=    G_{a,b}+\sum_{i=1}^\ell  \eta_{b+ie_s}  + G_{\wt y,v} \\
&=G_{a,\wt x}+G_{\wt y,v} \le  G_{a,v} .
\end{align*}
Thus  the inequalities above are in fact equalities.   The claims of the lemma follow from this.  
\end{proof} 
 
 We return to the exponential CGM.  
 For  lattice points $0\le v$,   write $\P_{0,v}$ for the probability measure of the LPP process in the rectangle $[0,v]$ with boundary and bulk weights  \eqref{e:w7}.

 \begin{lemma}\label{app-lm3}
  Let $1\le k<k+\ell\le m$.    Then $\P_{0,(m,n)}(\exitt_1 \ge k+\ell )=\P_{0,(m-k,n)}(\exitt_1 \ge \ell )$.  
 \end{lemma}
 
 \begin{proof}  Take $a=0$, $b=(k,0)$ and $v=(m,n)$ in Lemma \ref{app-lm1}. 
 Then, under $\P_{0,(m,n)}$,  the LPP process $G^{(a)}_{b,x}$  in $[b,v]$  has the   same distribution, modulo the translation of the origin to $b$,  as an LPP process  under $\P_{0,(m-k,n)}$.   By Lemma \ref{app-lm1} the maximizing paths from $a$ and $b$ to $v$ agree in their  portions inside $[k+1,m]\times[0,n]$. 
 \end{proof} 
 
 \begin{lemma}\label{app-lm5} 
  Let $1\le \bar m< m$ and $1\le n< \bar n$.  Then $\P_{0,(m,n)}(\exitt_1 <m-\bar m)=\P_{0,(\bar m,\bar n)}( \exitt_2>\bar n-n)$.  
 \end{lemma}
 
 \begin{proof}
 We couple these LPP processes as follows.  Let 
 \[  a=(\bar m-m,0), \quad a'=(0, n-\bar n) \quad \text{ and }\quad v=(\bar m,n). \]  The origin $0$ takes the  role of $b$ in Lemma \ref{app-lm1}. 
 
 Let i.i.d.\ Exp(1) weights $\{\w_x\}_{x\in\Z^2}$ be given.  Then 
place  independent boundary edge  weights with distributions dictated by \eqref{e:w7}  on the south and west boundaries of the lattice region $(a+\Z_{\ge0}^2)\cup(a'+\Z_{\ge0}^2)$: 
\begin{enumerate}
\item[(a)] On horizontal boundary edges put Exp($1-\rho$) weights $\sigma_{(i-1)e_1, ie_1}$ for $\bar m-m+1\le i\le 0$ and  $\sigma_{a'+(i-1)e_1, \,a'+ie_1}$ for $i\in\Z_{>0}$. 
\item[(b)] On vertical boundary edges   put  Exp($\rho$) weights $\sigma_{(j-1)e_2,\,je_2}$ for $n-\bar n+1\le j\le 0$ and  $\sigma_{a+(j-1)e_2,\,a+je_2}$ for $j\in\Z_{>0}$.  
 \end{enumerate} 
 
 Next consider two LPP processes    that emanate from $a$ and $a'$ and use the boundary weights described above in (a) and (b):   $G_{a,\ell e_2}$ for points $\ell e_2$, $\ell\ge 1$, on the $y$-axis, and $G_{a',ke_1}$ for points $ke_1$, $k\ge 1$,  on the $x$-axis.  
  Let these processes   define boundary weights on $\Z_{\ge0}^2$:  
$\eta_{(i-1)e_1,\,ie_1}=G_{a', \,ie_1}-G_{a', (i-1)e_1}$  and 
  $\eta_{(j-1)e_2,\,je_2}=G_{a, \, je_2}-G_{a, (j-1)e_2}$ for $i,j\in\Z_{>0}$.    See Figure \ref{fig-app3} for an illustration of the setting described here.  
 

 \begin{figure}
\begin{center}
\begin{picture}(200,150)(20,-10)
\put(90,0){\line(1,0){120}} 
\put(150,2){\small$\sigma$} 

 \put(40,40){\line(1,0){170}}\put(210,0){\line(0,1){110}}
 \put(63,42){\small$\sigma$}   \put(150,43){\small$\eta$}
\put(40,110){\line(1,0){170}}
\put(40,40){\line(0,1){70}} \put(90,0){\line(0,1){110}}
 \put(31,74){\small$\sigma$}
  \put(82,17){\small$\sigma$}
  \put(82,74){\small$\eta$} 
 
\put(37,37){\Large$\bullet$} 
\put(-30,38){\small$a=(\overline m-m,0)$}

\put(86.5,36.5){\Large$\bullet$} 
\put(93,43){\small$b=0$}

\put(87,-3){\Large$\bullet$} 
\put(50,-10){\small$a'=(0, n-\bar n)$}

\put(207,37){\Large$\bullet$} 
\put(215,37){\small$(\overline m,0)$}

\put(206,106){\Large$\bullet$} 
\put(215,107){\small$v=(\overline m,n)$}

\put(87,106){\Large$\bullet$} 
\put(80,118){\small$(0, n)$}

 \end{picture}
\end{center}  
\caption{ \small Illustration of the proof of Lemma \ref{app-lm5}. 
The $\eta$ and $\sigma$ labels on the south and west boundaries of rectangles indicate which boundary edge weights are placed on that side. } \label{fig-app3}
\end{figure}

 
 Now consider three LPP processes with lower left corners $a$, $0$ and $a'$:
\begin{enumerate}
\item[(i)] $\wt G_{a, x}$ uses boundary weights $\sigma_{(i-1)e_1,\,ie_1}$ for $\bar m-m+1\le i\le 0$ and $\eta_{(i-1)e_1,\,ie_1}$ for $i\in\Z_{>0}$ on  the horizontal axis emanating from $a$, and boundary weights  $\sigma_{a+(j-1)e_2,\,a+je_2}$ for $j\in\Z_{>0}$ on the vertical axis emanating from $a$.   
 
\item[(ii)] $\wt G_{0,x}$   uses boundary weights $\eta_{(i-1)e_1,\,ie_1}$  and 
  $\eta_{(j-1)e_2,\,je_2}$ on the standard axes emanating from $0$. 
  
\item[(iii)]  $\wt G_{a', x}$  uses boundary weights  $\sigma_{a'+(i-1)e_1, \,a'+ie_1}$ for $i\in\Z_{>0}$ on the horizontal axis emanating from $a'$,  and boundary  weights $\sigma_{(j-1)e_2,\,je_2}$ for $n-\bar n+1\le j\le 0$ and  $\eta_{(j-1)e_2,\,je_2}$ for $j\in\Z_{>0}$ on the vertical axis emanating from $a'$.   
\end{enumerate} 

  Let $\wt \P$ denote the probability measure under which this coupling has been constructed, that is, the probability measure of the independent weights $\w_x$ and $\sigma_{x,\,x+e_k}$.  
   
  Let $A$ be the event that the (a.s.\ unique) maximal path for $\wt G_{a,v}$ does not  go through the origin.  Let   $B$  the event that the (a.s.\ unique) maximal path for $\wt G_{a',v}$  goes through the point $e_2$.   Lemma \ref{app-lm1} applies to the pair $\wt G_{a, v}$ and $\wt G_{0, v}$, and also to the pair $\wt G_{a', v}$ and $\wt G_{0, v}$.  Thus the maximizing paths for  $\wt G_{a,v}$ and  $\wt G_{a',v}$ agree from that edge  onwards through which they enter  the positive first quadrant $\Z_{>0}^2$.  Both $A$ and $B$ are equivalent to the statement that this edge emanates from some point $\ell e_2$ for $\ell\ge 1$.    Hence $A=B$.  
 
 On the other hand,  LPP processes   $\{\wt G_{a,\,a+x}\}_{x\in\Z_{\ge0}^2}$  and $\{\wt G_{a',\,a'+x}\}_{x\in\Z_{\ge0}^2}$ both  have  the same distribution  as  the  LPP process $\{G^\rho_{0,x}\}_{x\in\Z_{\ge0}^2}$ with stationary increments.    Event $A$ is equivalent to the condition that the maximizing path for $\wt G_{a,v}$ takes at most $m-\bar m-1$ consecutive $e_1$-steps from $a$, which is the same as $\exitt_1<m-\bar m$ for  $G^\rho_{0,(m,n)}$. Similarly,  event $B$ says that 
 the   maximizing path for $\wt G_{a',v}$  takes at least $\bar n-n+1$ consecutive $e_2$-steps from $a'$, which for  $G^\rho_{0,(\bar m,\bar n)}$ is the same as $\exitt_2>\bar n-n$.    Thus 
 \[ \P_{0,(m,n)}(\exitt_1 <m-\bar m)=\wt \P(A)=\wt \P(B)=\P_{0,(\bar m,\bar n)}( \exitt_2>\bar n-n).  \qedhere \]
 \end{proof}

\medskip 

\small

\bibliographystyle{plain}

\bibliography{growthrefs}

\end{document}